\documentclass[final,1p]{elsarticle}

\usepackage{hyperref}

\makeatletter
\def\ps@pprintTitle{
 \let\@oddhead\@empty
 \let\@evenhead\@empty
 \let\@oddfoot\@empty
 \let\@evenfoot\@oddfoot
}
\makeatother

\usepackage{enumitem}

\usepackage{soul}

\usepackage{color}
\usepackage[dvipsnames]{xcolor}

\usepackage{amsmath,amsthm,amscd,amssymb,latexsym,upref,stmaryrd}
\usepackage{mathrsfs,mathtools,relsize}

\numberwithin{equation}{section}
\usepackage{hyperref}

\newtheorem{theorem}{Theorem}[section]
\newtheorem{proposition}[theorem]{Proposition}
\newtheorem{corollary}[theorem]{Corollary}
\newtheorem{lemma}[theorem]{Lemma}
\newtheorem{remark}[theorem]{Remark}

\DeclareMathOperator{\col}{col}

\DeclareMathOperator{\diag}{diag}

\DeclareMathOperator{\dom}{dom}

\DeclareMathOperator*{\esssup}{ess\,sup}

\newcommand{\oneto}[1]{\{1, \ldots, {#1}\}}

\newcommand{\oneton}{\oneto{n}}

\let\Im\relax
\DeclareMathOperator{\Im}{Im}
\let\Re\relax
\DeclareMathOperator{\Re}{Re}

\newcommand{\wt}{\widetilde}
\newcommand{\wh}{\widehat}

\renewcommand{\le}{\leqslant}
\renewcommand{\ge}{\geqslant}
\renewcommand{\atop}[2]{\genfrac{}{}{0pt}{2}{#1}{#2}}
\newcommand{\abs}[1]{\left|{#1}\right|}
\newcommand{\bigabs}[1]{\bigl|{#1}\bigr|}

\newcommand{\floor}[1]{\left\lfloor{#1}\right\rfloor}
\newcommand{\ceil}[1]{\left\lceil{#1}\right\rceil}

\newcommand{\eps}{\varepsilon}
\newcommand{\alp}{\alpha}
\newcommand{\gam}{\gamma}

\renewcommand{\l}{\lambda}
\renewcommand{\L}{\Lambda}

\def\bC{\mathbb{C}}

\def\bN{\mathbb{N}}

\def\bR{\mathbb{R}}
\def\bZ{\mathbb{Z}}

\journal{Journal of Functional Analysis}

\begin{document}

\sloppy

\begin{frontmatter}

\title
{On the trace formulas and completeness property of root vectors
systems for \texorpdfstring{$2 \times 2$}{2 x 2} Dirac type operators}

\author{Anton~A.~Lunyov}
\address{
Facebook, Inc. \\
1 Hacker Way, Menlo Park, California, 94025 \\
United States of America}

\ead{A.A.Lunyov@gmail.com}

\author{Mark~M.~Malamud}

\address{
Peoples Friendship University of Russia (RUDN University) \\
6 Miklukho-Maklaya St., Moscow, 117198 \\
Russian Federation}

\ead{malamud3m@gmail.com}

\begin{abstract}
The paper is concerned with the completeness property of the system of root vectors of a boundary value problem for the following $2 \times 2$ Dirac type equation
\begin{align*}
 & L y = -i B^{-1} y' + Q(x) y = \lambda y , \quad y= \col(y_1, y_2),
 \quad x \in [0,1], \\
 & B = \diag(b_1, b_2),
 \quad b_1 < 0 < b_2, \quad\text{and}\quad
 Q \in W_1^n[0,1] \otimes \bC^{2 \times 2},
\end{align*}
subject to general non-regular two-point boundary conditions $C y(0) + D y(1) = 0$.
If $b_2 = -b_1 = 1$ this equation is equivalent to the one dimensional Dirac equation.

We establish asymptotic expansion of the characteristic determinant of this boundary value problem. This expansion directly yields new completeness result for the system of root vectors of such boundary value problem with \emph{non-regular and even degenerate} boundary conditions. We also present several explicit completeness results in terms of the values $Q^{(j)}(0)$ and $Q^{(j)}(1)$. In the case of degenerate boundary conditions and analytic $Q(\cdot)$, the criterion of completeness property is established.
\end{abstract}


\end{frontmatter}

\renewcommand{\contentsname}{Contents}
\tableofcontents

\section{Introduction}
In this paper we continue investigations from the cycle of our previous papers~\cite{MalOri12,LunMal14IEOT,LunMal15JST,LunMal16JMAA,LunMal22JDE} on spectral properties of first order $n\times n$ systems of ODE.
Here we restrict ourselves to the case of $2 \times 2$ Dirac type system of the form
\begin{equation}\label{eq:system}
 -i B^{-1} y'+Q(x)y=\l y, \qquad y=\col(y_1,y_2), \qquad x\in[0,1],
\end{equation}
where
\begin{equation}\label{eq:BQ}
 B = \diag(b_1, b_2), \quad b_1 < 0 < b_2 \quad \text{and}\quad
 Q = \begin{pmatrix} 0 & Q_{12} \\ Q_{21} & 0 \end{pmatrix}
 \in L^1([0,1];\bC^{2 \times 2}).
\end{equation}
In the case $-b_1 = b_2 =1$ system~\eqref{eq:system} is equivalent to the classical $2 \times 2$ Dirac system.

With system~\eqref{eq:system} one associates, in a natural way,
the maximal operator $L = L(Q)$ acting in $L^2([0,1]; \bC^2)$ on
the domain
\begin{equation}
 \dom(L) = \{y \in W_1^1([0,1]; \bC^2) : Ly \in L^2([0,1]; \bC^2)\}.
\end{equation}
Here $W_p^n[a,b]$ denotes the Sobolev space of functions $f$ having $n-1$ absolutely continuous derivatives on $[a,b]$ and satisfying $f^{(n)}\in L^p[a,b]$, in particular $W_p^0[a,b] = L^p[a,b]$

To obtain a boundary value problem (BVP), equation~\eqref{eq:system} is subject to the following boundary conditions
\begin{equation}\label{eq:Udef}
 U_j(y) := a_{j 1}y_1(0) + a_{j 2}y_2(0) + a_{j 3}y_1(1) + a_{j 4}y_2(1)= 0,
 \quad j \in \{1,2\}.
\end{equation}
Going forward we set
\begin{equation} \label{eq:Ajk.Jjk.def}
 A_{jk} := \begin{pmatrix} a_{1j} & a_{1k} \\ a_{2j} & a_{2k} \end{pmatrix}
 \qquad\text{and}\qquad J_{jk} := \det (A_{jk}), \qquad j,k\in \oneto{4}.
\end{equation}

Denote by $L_U := L_U(Q)$ the operator associated in
$L^2([0,1]; \bC^n)$ with the
BVP~\eqref{eq:system}--\eqref{eq:Udef}. It is defined as the
restriction of the maximal operator $L = L(Q)$ to the domain
\begin{equation} \label{eq:dom}
 \dom(L_U) = \{y \in \dom(L) : U_1(y) = U_2(y) = 0\}.
\end{equation}

The general spectral problem for $n \times n$ first order system of the form~\eqref{eq:system}
for the first was investigated by G.\;Birkhoff and R.\;Langer~\cite{BirLan23}. More precisely, they introduced the concepts of \emph{regular and strictly regular boundary conditions}, investigated the asymptotic behavior of eigenvalues and eigenfunctions and proved \emph{a pointwise convergence result} on spectral decompositions for the corresponding differential operator.

The first completeness result for such systems was established by V.P. Ginzburg~\cite{Gin71} in the case $B = I_n$, $Q(\cdot) = 0$.

Recall that boundary conditions~\eqref{eq:Udef} are called \emph{regular}, if
\begin{equation} \label{eq:reg.BC.intro}
 J_{32} = \det(A_{12} P_+ + A_{34} P_-) \ne 0 \quad\text{and}\quad
 J_{14} = \det(A_{12} P_- + A_{34} P_+) \ne 0,
\end{equation}
where $P_+$ (resp.\ $P_-$) is the spectral projection
onto positive (resp.\ negative) part of the spectrum of the matrix $B = B^*$.

V.A.~Marchenko~\cite{Mar77} was the first to establish
completeness property for the system of root functions of the Dirac operator $L_U$ ($-b_1 = b_2$) with regular boundary conditions and continuous potential matrix $Q$. The last restriction occurs because the transformation operators used for the proof was constructed in~\cite{Mar77} only for continuous $Q$.

Later, L.L.~Oridoroga and one of the authors~\cite{MalOri12} established completeness property for \emph{$B$-weakly regular} boundary value problems for arbitrary $n \times n$ first order systems of ODE with integrable matrix potential $Q \in L^1([0,1]; \bC^{n \times n})$ (originally this result was announced in~\cite{MalOri00} much earlier).
In particular, for $2 \times 2$ Dirac type system (with $Q\in L^1([0,1]; \bC^{2 \times 2})$ )
\emph{the condition of $B$-weakly regularity turns into~\eqref{eq:reg.BC.intro}} and hence ensures the completeness property for such BVP with regular boundary conditions.

Let us also briefly mention that the Riesz basis property in $L^2([0,1];\bC^2)$ of BVP~\eqref{eq:system}--\eqref{eq:Udef}
for $2 \times 2$ Dirac system with various smoothness
assumptions on the potential matrix $Q$ was investigated in numerous papers
(see~\cite{DjaMit10,DjaMit12UncDir,DjaMit12Crit,LunMal14Dokl,SavShk14, LunMal16JMAA,SavSad15,KurAbd18,KurGad20,Lun23,Iba23} and the references therein).
The first general result for non-smooth potentials was obtained by P.~Djakov and B.~Mityagin~\cite{DjaMit10, DjaMit12UncDir}
who proved under the assumption $Q \in L^2([0,1]; \bC^{2 \times 2})$ that the system of root vectors of the BVP~\eqref{eq:system}--\eqref{eq:Udef} \emph{with strictly regular boundary conditions} forms a Riesz basis and forms \emph{a block Riesz basis} whenever boundary conditions are only regular.
Note however that the methods of these papers substantially rely on $L^2$-techniques (such as Parseval equality, Hilbert-Schmidt operators, etc.) and that cannot be applied to $L^1$-potentials.
The most complete result on the Riesz basis property for $2\times 2$ Dirac and Dirac-type systems, respectively, with $Q \in L^1([0,1]; \bC^{2 \times 2})$ and strictly regular boundary conditions was obtained independently, by different methods and at the same time by A.M.~Savchuk and A.A.~Shkalikov~\cite{SavShk14} and by the authors~\cite{LunMal14Dokl,LunMal16JMAA}.
The case of regular boundary conditions is treated in~\cite{SavShk14} for the first time.

Let us denote by $\L_Q := \L_{U,Q} = \{\l_{Q,n}\}_{n \in \bZ} := \{\l_n\}_{n \in \bZ}$ the spectrum of the operator $L_U(Q)$, counting multiplicity.
Namely, if $m_a(\l_n)$ is the algebraic multiplicity of the eigenvalue $\l_n$ of the operator $L_U(Q)$, then the number $\l_n$ appears in the sequence $\L_Q$ exactly $m_a(\l_n)$ times.
It is worth mentioning, that the algebraic multiplicity $m_a(\l_n)$ of the eigenvalue $\l_n$ coincides with the multiplicity of the zero $\l_n$ of the characteristic determinant $\Delta_Q(\cdot) = \Delta_{Q,U}(\cdot)$, an entire function given by the following formula:
\begin{multline} \label{eq:Delta.intro}
 \Delta_Q(\l) = J_{12} + J_{34}e^{i(b_1+b_2)\l} \\
 + J_{32}\varphi_{11}(1, \l) + J_{13}\varphi_{12}(1, \l)
 + J_{42}\varphi_{21}(1, \l) + J_{14}\varphi_{22}(1, \l).
\end{multline}
where $(\varphi_{jk}(x,\l))_{j,k=1}^2 := \Phi(x,\l)$ is the fundamental matrix of the system~\eqref{eq:system} (uniquely) determined by the initial condition $\Phi(0,\l)=I_2$.
Recall also that
\begin{equation}\label{eq:Delta0.intro}
 \Delta_0(\l) = J_{12} + J_{34}e^{i(b_1+b_2)\l}
 + J_{32}e^{ib_1\l} + J_{14}e^{ib_2\l}
\end{equation}
is the characteristic determinant of problem~\eqref{eq:system}--\eqref{eq:Udef} with $Q=0$.

Note that \emph{boundary conditions~\eqref{eq:reg.BC.intro} are regular if and only if the entire function $\Delta_0(\cdot)$ is of maximal growth in both half-planes $\bC_{\pm}$}. In this connection we recall one of the key results of our previous work~\cite{LunMal16JMAA} which states
that the characteristic determinant $\Delta_Q(\cdot)$ of the problem~\eqref{eq:system}--\eqref{eq:Udef} with $Q(\cdot) \in L^1([0,1]; \bC^{2 \times 2})$ admits the following representation:
\begin{align} \label{eq:Delta=Delta0+_Intro}
 \Delta_Q(\l) &= \Delta_0(\l) + \int^1_0 g_1 (t) e^{i b_1 \l t} dt
 + \int^1_0 g_2 (t) e^{i b_2 \l t} dt,
\end{align}
with certain functions $g_1, g_2 \in L^1[0,1]$ expressed via kernels of the transformation operators. In turn, this formula ensures that the
determinant $\Delta_Q(\cdot)$ is also an entire sine-type function with the indicator diagram $[-ib_2, i|b_1|] \subset i \bR$,
i.e. has the exponential type $|b_{1}| (b_2)$ in $\bC_+ (\bC_-)$.
However, the converse is not true: \emph{certain non-regular boundary conditions} also lead to the characteristic
determinant $\Delta_Q(\cdot)$ with the same indicator diagram $[-ib_2, i|b_1|]$ (see~\cite{MalOri12} and~\cite{LunMal15JST})
although the indicator diagram of $\Delta_0(\cdot)$ is narrower.
A stronger effect (of the maximal growth of $\Delta_Q(\cdot)$) for system~\eqref{eq:system} with variable matrix
$B(\cdot)$ was recently discovered in~\cite{LunMal22POMI}.

Emphasize that representation~\eqref{eq:Delta=Delta0+_Intro} plays a crucial role in several recent investigations. For instance, A.S.\;Makin~\cite{Mak20,Mak21DE} applied representation~\eqref{eq:Delta=Delta0+_Intro} to prove that a given entire function of exponential type is a characteristic determinant of a degenerate BVP ($\Delta_0(\cdot) = 0$) for $2 \times 2$
Dirac system if and only if this function is of Paley-Wiener class.
In~\cite{Mak22} Makin used this representation to find explicit algebraic conditions on a potential matrix $Q(\cdot)$ that guarantee the (non-block!) Riesz basis property for arbitrary regular but not necessary strictly regular $2 \times 2$ Dirac operator.
Recently the authors~\cite{LunMal22JDE} applied this representation to evaluate Lipschitz dependance of the spectral
data on the potential matrix $Q(\cdot)$ running through compact sets of $L^p([0,\ell]; \bC^{2 \times 2})$, $p \in [1,2]$.
Moreover, in~\cite{LunMal21} we generalized representation~\eqref{eq:Delta=Delta0+_Intro} to the case of $n \times n$ system of ODE and applied it to establish the Riesz basis property in $L^2([0,1];\bC^n)$ of the root vectors system of the corresponding strictly regular BVP.

If the functions $g_1, g_2 \in W_1^n[0,1]$, one can integrate by parts in~\eqref{eq:Delta=Delta0+_Intro} and use Riemann-Lebesgue Lemma to obtain the following asymptotic expansion for $\Delta_Q(\l)$:
\begin{align}
\label{eq:Delta+}
 \Delta_Q(\l) &= e^{i b_1 \l} \cdot \Bigl(J_{32}
 - \sum_{k=1}^n \frac{g_1^{(k-1)}(1)}{(-i b_1 \l)^k}
 + o(\l^{-n})\Bigr), \qquad \Im \l \to + \infty, \\
 \label{eq:Delta-}
 \Delta_Q(\l) &= e^{i b_2 \l} \cdot \Bigl(J_{14}
 - \sum_{k=1}^n \frac{g_2^{(k-1)}(1)}{(-i b_2 \l)^k}
 + o(\l^{-n})\Bigr), \qquad \Im \l \to - \infty,
\end{align}
(see Remark~\ref{rem:byparts} for details).

Note in this connection that similar representation for the characteristic determinant of the Sturm-Liouville operators was used in~\cite{Mal08} to establish potential-dependent completeness property of the BVP with
degenerate boundary conditions and a potential $q \in C^n[0,1]$.
Later A.S.~Makin~\cite{Mak14} significantly improved this result by relaxing the continuity assumption.

In this paper we investigate completeness property for the problem~\eqref{eq:system}--\eqref{eq:Udef} \emph{with non-regular boundary conditions}, i.e.\ when $J_{32} J_{14} = 0$. In this case BVP~\eqref{eq:system}--\eqref{eq:Udef} necessarily has \emph{an incomplete system of root functions whenever $Q(\cdot) \equiv 0$}.

In~\cite{MalOri12} L.L.~Oridoroga and one of the authors (see also~\cite{MalOri10}) established the first potential-dependent completeness result for the operator $L_U(Q)$ with non-regular boundary conditions assuming that $Q \in C^1([0,1]; \bC^{2 \times 2})$. It reads as follows:
\begin{theorem}[Theorem 5.1 in~\cite{MalOri12}] \label{th:MalOri}
Let $Q_{12}, Q_{21} \in C^1[0,1]$. Then the system of root functions of the operator $L_U(Q)$ is complete in $L^2([0,1]; \bC^2)$ whenever the following two conditions hold:
\begin{align}
\label{eq:A32}
 & |J_{32}| + |b_1 J_{13} Q_{12}(0) + b_2 J_{42} Q_{21}(1) | \ne 0, \\
\label{eq:A14}
 & |J_{14}| + |b_1 J_{13} Q_{12}(1) + b_2 J_{42} Q_{21}(0) | \ne 0.
\end{align}
\end{theorem}
The proof relies on formulas~\eqref{eq:Delta+}--\eqref{eq:Delta-} for $n=1$.
Namely, using triangular transformation operators the authors
obtained explicit form of the coefficients $g_1(1)$ and $g_2(1)$ in~\eqref{eq:Delta+}--\eqref{eq:Delta-} in terms of a potential matrix $Q(\cdot)$ (see also~\cite[Remark 4.6]{LunMal15JST} on the discussion about smoothness assumption in this result).
In~\cite{AgiMalOri12}, similar results using the same method were obtained in the case of $B \ne B^*$ and analytic $Q(\cdot)$.

In~\cite{LunMal14IEOT} and~\cite{LunMal15JST}, we generalized and clarified~\cite[Theorem 5.1]{MalOri12} to establish \emph{potential-dependent completeness as well as spectral synthesis property} for the system of root functions of the $n \times n$ system with not necessarily selfadjoint matrix $B$ and non-weakly-regular boundary conditions.
In particular, for $2 \times 2$ Dirac-type systems~\cite[Proposition 4.5]{LunMal15JST} improves~\cite[Theorem 5.1]{MalOri12} by relaxing $C^1$-smoothness of the potential $Q(\cdot)$ on $[0,1]$ to continuity of $Q(\cdot)$ at the endpoints $\{0,1\}$.

In a very recent paper~\cite{KosShk21} A.P.\;Kosarev and A.A.\;Shkalikov extended completeness results from~\cite{MalOri12,AgiMalOri12,LunMal15JST} to the case of $2 \times 2$ Dirac-type operators with non-constant matrix $B = \diag(b_1(x), b_2(x))$ and degenerate boundary conditions of a special form ($y_1(0) = y_2(1) = 0$) under the smoothness assumption $b_1, b_2, Q_{12}, Q_{21} \in W_1^1[0,1]$.
In another recent preprint~\cite{Mak23}, A.S.\;Makin generalized and substantially improved our result~\cite[Proposition 4.5]{LunMal15JST} in the Dirac case by proposing a new interesting approach (see Theorem~\ref{th:Makin}). Namely, instead of continuity condition on $Q(\cdot)$ he proposed a new more general integral limit condition.

In this paper, we aim to further refine results from~\cite{MalOri12} and~\cite{LunMal15JST} when potential matrix $Q(\cdot)$ has additional smoothness.
However, finding explicit form of the derivatives $g_1^{(k-1)}(1)$ and $g_2^{(k-1)}(1)$ in asymptotic formulas~\eqref{eq:Delta+}--\eqref{eq:Delta-} directly using transformation operators is rather difficult.
In this paper we propose another approach to compute these coefficients by adapting Marchenko's method from~\cite{Mar77}.
Namely, we establish asymptotic expansion of
the characteristic determinant $\Delta_Q(\l)$ of BVP~\eqref{eq:system}--\eqref{eq:Udef} of the form~\eqref{eq:Delta+}--\eqref{eq:Delta-} with coefficients
explicitly expressed as polynomials in boundary values of the functions $Q_{12}, Q_{21} \in W_1^n[0,1]$ and their derivatives as well as numbers $b_1$, $b_2$, $J_{14}$, $J_{32}$, $J_{13}$, $J_{42}$ (see Theorem~\ref{th:Delta}).
In turn, we apply this result to provide general explicit conditions of completeness in Theorem~\ref{th:compl.gen.2x2}.
For instance, in the case of $Q_{12}, Q_{21} \in W_1^2[0,1]$, it reads as follows.
\begin{proposition} \label{prop:n=2}
Let $Q_{12}, Q_{21} \in W_1^2[0,1]$. Assume that either condition~\eqref{eq:A32} holds or the following condition holds:
\begin{equation} \label{eq:c2+ne0.intro}
 b_1 J_{13} Q_{12}'(0) - b_2 J_{42} Q_{21}'(1)
 + i b_1 b_2 J_{14} Q_{12}(0) Q_{21}(1) \ne 0.
\end{equation}
Further, assume also that either condition~\eqref{eq:A14} holds or the following condition holds:
\begin{equation} \label{eq:c2-ne0.intro}
 b_1 J_{13} Q_{12}'(1) - b_2 J_{42} Q_{21}'(0)
 - i b_1 b_2 J_{32} Q_{12}(1) Q_{21}(0) \ne 0
\end{equation}
Then the system of root vectors of the BVP~\eqref{eq:system}--\eqref{eq:Udef} is complete and minimal in $L^2([0,1];\bC^2)$.
\end{proposition}
It is evident that Proposition~\ref{prop:n=2} (announced in~\cite{LunMal13Dokl}) even in the classical Dirac case is not covered by previous completeness results from~\cite{MalOri12,AgiMalOri12,LunMal15JST,KosShk21,Mak23}, if one of conditions~\eqref{eq:A32}--\eqref{eq:A14} is violated (see Remark~\ref{rem:not.covered} for detailed discussion).

Note, however, that general form of the coefficients in the expansions~\eqref{eq:Delta+}--\eqref{eq:Delta-} is somewhat cumbersome. If $J_{14} = J_{32} = 0$, then we can present more explicit and refined condition of completeness announced in~\cite[Theorem 4]{LunMal13Dokl}.

The following function is highly involved in the following series of results,
\begin{equation} \label{eq:Pdef.intro}
 P(x) := J_{13} b_1 Q_{12}(x) + J_{42} b_2 Q_{21}(1-x).
\end{equation}
\begin{theorem} \label{th:J32=J14=0.intro}
Let $Q_{12}, Q_{21} \in W_1^n[0,1]$ for some $n \in \bN$ and let
\begin{equation} \label{eq:J32=J14=0}
 J_{32}=J_{14}=0, \qquad J_{13} J_{42} \ne 0.
\end{equation}
Further, let for some $n_0, n_1 \in \{0, 1, \ldots, n-1\}$ the following conditions hold
\begin{align}
\label{eq:Pk0.intro}
 P^{(k)}(0) &= 0, \qquad k \in \{0, 1, \ldots, n_0 - 1\}, \\
\label{eq:Pk1=0.intro}
 P^{(k)}(1) &= 0, \qquad k \in \{0, 1, \ldots, n_1 - 1\}.
\end{align}
Then the system of root vectors of the BVP~\eqref{eq:system}--\eqref{eq:Udef} is
complete and minimal in $L^2([0,1];\bC^2)$ provided that
\begin{equation} \label{eq:Pn0.Pn1}
 |n_1 - n_0| \le 1 \qquad\text{and}\qquad
 P^{(n_0)}(0) \ne 0, \qquad P^{(n_1)}(1) \ne 0.
\end{equation}
\end{theorem}
Extended version of Theorem~\ref{th:J32=J14=0.intro}, which also covers the case $|n_1 - n_0| > 1$, is contained in Theorem~\ref{th:J32=J14=0}.

Let us outline a few additional simple sufficient conditions guaranteeing completeness property. In all results stated below we assume that $Q_{12}, Q_{21} \in W_1^n[0,1]$ for some $n \ge 3$.
First, we mention result that demonstrates substantial difference between Sturm-Liouville and Dirac operators (see Corollary~\ref{cor:P0.P1=0} and Remark~\ref{rem:Dirac.vs.SL}). Namely, if $J_{32}=J_{14}=0$, $J_{13} J_{42} \ne 0$, then the following condition ensures the desired completeness property,
\begin{equation}
 Q_{12}(1) \ne 0, \qquad P^{(n-3)}(0) \ne 0, \qquad
 P^{(k)}(1) = 0, \quad k \in \{0,1,\ldots,n-1\}.
\end{equation}
In particular, completeness is possible if $P(x) = 0$, $x \in [a,1]$, for some $a \in (0,1)$.

Further, under some algebraic assumption on the potential matrix $Q(\cdot)$, we can eliminate any restrictions between $n_0$ and $n_1$ in condition~\eqref{eq:Pn0.Pn1}. Indeed, according to Corollary~\ref{cor:P0.P1=0} (see also Proposition~\ref{prop:Q=0.P0.P1} for more general result), if $J_{32}=J_{14}=0$, $J_{13} J_{42} \ne 0$,
then desired completeness property holds whenever both of the following conditions hold
\begin{align}
\label{eq:Q12.Q21=0.intro}
 & Q^{(j)}(0) = Q^{(j)}(1) = 0, \quad j \in \{0,1,\ldots,m-1\},
 \quad\text{where}\quad m = \text{\scalebox{0.85}{$\ceil{\frac{n-2}{3}}$}}, \\
\label{eq:Pn00.Pn11.intro}
 & P^{(n_0)}(0) P^{(n_1)}(1) \ne 0 \qquad\text{for some}\quad
 n_0, n_1 \in \{m, m+1, \ldots, n-1\}.
\end{align}
Similar result is valid if $J_{32} = 0$, while $J_{14} \ne 0$. Namely, in this case, Corollary~\ref{cor:Makin.gen1} ensures the desired completeness property whenever
\begin{equation} \label{eq:Pn.Q12.Q21j.intro}
 P^{(n-1)}(0) \ne 0, \quad
 Q_{12}^{(j)}(0) = Q_{21}^{(j)}(1) = 0, \quad 0 \le j < m,
 \quad\text{where}\quad m = \text{\scalebox{0.85}{$\ceil{\frac{n-1}{2}}$}}.
\end{equation}
See Proposition~\ref{prop:Q12.Q21=0} and Remark~\ref{rem:Makin1} for more details.

We also find explicit completeness condition in the case $J_{32} = J_{13} = 0$ (see Corollaries~\ref{cor:J32=J42=J13=0} and~\ref{cor:J32=J13=0}). Namely, if
\begin{equation} \label{eq:Q12j0.Q21j1.intro}
 J_{32} = J_{13} = J_{42} = 0, \qquad J_{14} \ne 0, \qquad\text{and}\qquad
 Q_{12}^{(j_0)}(0) Q_{21}^{(j_1)}(1) \ne 0,
\end{equation}
for some $j_0, j_1 \in \{0, 1, \ldots, n-2\}$ such that $j_0 + j_1 \le n-2$,
then the system of root vectors of the BVP~\eqref{eq:system}--\eqref{eq:Udef} is
complete and minimal in $L^2([0,1];\bC^2)$. Moreover, completeness property is preserved
whenever one of the following conditions holds,
\begin{align}
\label{eq:J14ne0.Q211}
 & J_{32} = J_{13} = 0, \quad J_{42} \ne 0, \qquad J_{14} \ne 0,
 \qquad Q_{21}^{(j_1)}(1) \ne 0, \\
\label{eq:J14=0.Q2101}
 & J_{32} = J_{13} = 0, \quad J_{42} \ne 0, \qquad J_{14} = 0,
 \qquad Q_{21}^{(j_0)}(0) Q_{21}^{(j_1)}(1) \ne 0,
\end{align}
for some $j_0, j_1 \in \{0, 1, \ldots, n-1\}$.

Note also that an $m$-dissipative operator $L_U(0)$ automatically meets the condition $J_{14} \ne 0$. Therefore, in this case, condition~\eqref{eq:J14ne0.Q211} shows that under certain additional algebraic restrictions, completeness property for $L_U(Q)$ depends only on behavior of $Q_{21}(\cdot)$ at the endpoint $1$. This result and its applications to the spectral synthesis of dissipative operators (previously investigated in~\cite{LunMal14IEOT}) will be discussed elsewhere.

It is surprising that the function $Q_{12}(\cdot)$ is not involved in conditions~\eqref{eq:J14ne0.Q211}--\eqref{eq:J14=0.Q2101}. In turn, these conditions
imply the following peculiar criterion of completeness when the potential $Q(\cdot)$ is an \emph{analytic matrix function} on $[0,1]$ and boundary conditions are of the form
\begin{equation} \label{eq:y2(1)=0.intro}
 \alpha_1 y_1(0) + \alpha_2 y_2(0) + \alpha_3 y_1(1) = 0,
 \quad y_2(1)=0, \quad \alpha_2 \ne 0.
\end{equation}
Namely, in accordance with Corollary~\ref{cor:criterion}, the system of root vectors of the BVP~\eqref{eq:system}, \eqref{eq:y2(1)=0.intro} is complete in $L^2([0,1];\bC^2)$ \emph{if and only if $Q_{21}(\cdot) \not \equiv 0$}.

The paper is organized as follows.
In Section~\ref{sec:prelim} we recall key results from~\cite{LunMal16JMAA} on transformation operators and define special functional spaces $X_\infty$ and $X_\infty^0$. Here we also prove our first new completeness result, Proposition~\ref{prop:2x2.notR}, which easily follows from previous results in~\cite{LunMal14IEOT,LunMal15JST}.
In Section~\ref{sec:trace} we establish asymptotic expansions for solutions to system~\eqref{eq:system} (see Theorem~\ref{th:asymp}) and ``trace formulas'' for certain solution quotients (Proposition~\ref{prop:sigma}).
In Section~\ref{sec:Delta} we apply these trace formulas to derive asymptotic expansion of the characteristic determinant $\Delta_Q(\cdot)$ (Theorem~\ref{th:Delta}).
In Section~\ref{sec:compl.gen} using this asymptotic expansion we establish a general refined completeness result (Theorem~\ref{th:compl.gen.2x2}).
In Sections~\ref{sec:sigmajx} and~\ref{sec:compl.refined} we ``decipher'' Theorem~\ref{th:compl.gen.2x2} to establish a series of explicit completeness results outlined above.

The main results of this paper
were announced in our short communication~\cite{LunMal13Dokl} published in 2013.
However, the recent publications~\cite{KosShk21,Mak23} (especially paper~\cite{Mak23} by A.S.~Makin) influenced us to compare our results from~\cite{LunMal13Dokl} with the new ones. Namely, we substantially extended the previous version of the paper written along~\cite{LunMal13Dokl} by adding several new results on completeness (see Section~\ref{sec:compl.refined}).
\section{Preliminaries}
\label{sec:prelim}
Following~\cite{Mal94,LunMal16JMAA} denote by $X_\infty := X_\infty(\Omega)$ the linear space composed of (equivalent classes of) measurable functions defined on the domain
\begin{equation} \label{eq:Omega.def}
 \Omega := \{(x,t) : 0 \le t \le x \le 1\}
\end{equation}
satisfying
\begin{equation} \label{eq:B1.norm.def}
 \|f\|_{X_\infty} := \esssup_{x \in [0,1]} \int_0^x |f(x,t)| dt < \infty.
\end{equation}
It can easily be shown that the space $X_\infty$ equipped with the norm~\eqref{eq:B1.norm.def} forms a Banach space that is not separable. Denote by $X_\infty^0$ the (separable) subspace of $X_\infty$ obtained by taking the closure of the set of continuous functions $C(\Omega)$.
Evidently, the set $C^1(\Omega)$ of smooth functions is also dense in the space $X_\infty^0$.
In the sequel, we need to following important property of the space $X_\infty^0$ established in~\cite{LunMal16JMAA}.
\begin{lemma} \label{lem:K.exp.X0}
Let $K \in X_\infty^0$ and $b \in \bR \setminus \{0\}$. Then for each $\delta >0$ there exists $R_\delta = R_\delta(K, b) > 0$ such that the following uniform estimate holds
\begin{equation}
 \abs{\int_0^x K(x,t) e^{i b \l t} \, dt} < \delta \cdot (e^{-b \Im \l} + 1),
 \qquad |\l| > R_\delta, \quad x \in [0,1].
\end{equation}
\end{lemma}
\begin{remark}
Note that~\cite[Lemma 3.2]{LunMal16JMAA} is formulated with additional restriction $|\Im \l| \le h$, but the proof in more general case remains the same. See also~\cite[Lemma~5.12]{LunMal22JDE} for stronger statement concerning compact sets in $X_\infty^0(\Omega)$.
\end{remark}
The space $X_\infty^0$ played crucial role in~\cite{LunMal16JMAA} to establish existence of triangular transformation operators for solutions to system~\eqref{eq:system}.
In particular, for each $a \in [0,1]$, the trace operator $i_a: X_\infty^0(\Omega) \to L^1[0,a]$ (originally defined on $C(\Omega)$ via $i_a\bigl(N(x,t)\bigr):=N(a,t)$) is correctly extended to $X_\infty^0(\Omega)$.

To formulate the next result (\cite[Proposition~3.1]{LunMal16JMAA}), denote by $\Phi(x,\l)$ the fundamental matrix
of the system~\eqref{eq:system} (uniquely) determined by the initial condition $\Phi(0,\l)=I_2$, i.e.,
\begin{equation} \label{eq:Phi.def}
 \Phi(x,\l) := \begin{pmatrix} \Phi_1(x,\l) & \Phi_2(x,\l) \end{pmatrix}, \quad
 \Phi_k(x,\l) := \begin{pmatrix} \varphi_{1k}(x,\l) \\ \varphi_{2k}(x,\l)
 \end{pmatrix}, \quad k\in\{1,2\},
\end{equation}
where $\displaystyle \Phi_1(0,\l) := \binom{1}{0},\ \Phi_2(0,\l)
= \binom{0}{1}$.
\begin{proposition} \label{prop:phi.jk=e+int}
Let $Q \in L^1([0,1]; \bC^{2 \times 2})$.
Then the functions $\varphi_{jk}(\cdot, \l)$ admit the following representations
\begin{equation} \label{eq:phijk}
 \varphi_{jk}(x,\l) = \delta_{jk} e^{i b_k \l x}
 + \int_0^x R_{1,j,k}(x,t) e^{ib_1\l t}dt
 + \int_0^x R_{2,j,k}(x,t) e^{ib_2\l t}dt,
\end{equation}
for $x \in [0,1]$, $\l \in \bC$, where $R_{l,j,k} \in X_\infty^0(\Omega),$ \ $j,k,l \in \{1,2\}$.
\end{proposition}
Integral representations~\eqref{eq:phijk} for the entries of the fundemantal matrix $\Phi(x,\l)$ played crucial role in our papers~\cite{LunMal14Dokl, LunMal16JMAA}. Moreover, the functions $g_1, g_2 (\in L^1[0,1])$ in formula~\eqref{eq:Delta=Delta0+_Intro} are linear combinations of traces $i_1\bigl(R_{l,j,k}(x,t)\bigr):= R_{l,j,k}(1,t)$ and due to inclusions $R_{l,j,k} \in X_\infty^0(\Omega)$ are well-defined. Note, that for Sturm-Liouville operator the kernels of transformation operators are always continuous, and similar difficulties in representation~\eqref{eq:Delta=Delta0+_Intro} do not occur.
Besides, for the case of Dirac operator with continuous potential $Q(\cdot)$ treated in~\cite{Mar77}, the transformation operator matrix kernel is continuous: $R_{l,j,k} \in C(\Omega)$, and therefore the traces $R_{l,j,k}(1,t)$ and hence the functions $g_1, g_2$ in~\eqref{eq:Delta=Delta0+_Intro} are also continuous.

In the sequel we need alternative, equivalent to~\eqref{eq:phijk} representation of the fundamental matrix $\Phi(x,\l)$.
\begin{corollary} \label{cor:Phi+-}
Let $Q \in L^1([0,1]; \bC^{2 \times 2})$.
Then the fundamental matrix $\Phi(x,\l)$ of the system~\eqref{eq:system} admits the following representation
\begin{equation} \label{eq:Phi+-}
 \Phi(x,\l) = \begin{pmatrix}
 e^{i b_1 \l x} v_1^+(x,\l) & e^{i b_2 \l x} v_2^-(x,-\l) \\
 e^{i b_1 \l x} v_2^+(x,\l) & e^{i b_2 \l x} v_1^-(x,-\l)
 \end{pmatrix}, \quad x \in [0,1], \quad \l \in \bC,
\end{equation}
where
\begin{equation} \label{eq:vkpm.def}
 v_k^\pm (x,\l) = \delta_{1k} + \int_0^x R^\pm _k(x,t)
 e^{i (b_2 - b_1) \l t}dt, \qquad k \in \{1,2\}.
\end{equation}
and $R^\pm _1, R^\pm _2\in X_\infty^0(\Omega)$ are certain integrable kernels.
\end{corollary}
\begin{proof}
Let $j \in \{1,2\}$. Making a change of variable $b_1 t = (b_2 - b_1) s + b_1 x$ we get for $x \in [0,1]$ and $\l \in \bC$,
\begin{equation}
 \int_0^x R_{1,j,1}(x,t) e^{i b_1 \l t} \,dt
 = e^{i \l b_1 x} \int_0^{\frac{-b_1 x}{b_2 - b_1}} R_{1,j,1}\left(x,
 \text{\scalebox{0.75}{$\frac{(b_2-b_1)s + b_1 x}{b_1}$}}\right)
 e^{i (b_2-b_1) \l s} \,ds.
\end{equation}
Similarly, making a change of variable $b_2 t = (b_2 - b_1) s + b_1 x$ we get
\begin{equation}
 \int_0^x R_{2,j,1}(x,t) e^{i b_2 \l t} \,dt
 = e^{i \l b_1 x} \int_{\frac{-b_1 x}{b_2 - b_1}}^{x} R_{2,j,1}\left(x,
 \text{\scalebox{0.75}{$\frac{(b_2-b_1)s + b_1 x}{b_2}$}}
 \right) e^{i (b_2-b_1) \l s} \,ds.
\end{equation}
Inserting these two relations into~\eqref{eq:phijk} we arrive at
\begin{equation} \label{eq:phij1}
 \varphi_{j1}(x,\l) = e^{i b_1 \l x} \left(\delta_{j1}
 + \int_0^x R_j^+(x, t) e^{i (b_2-b_1) \l t} \, dt \right),
 \qquad x \in [0,1], \quad \l \in \bC,
\end{equation}
where
\begin{equation}
 R_j^+(x,t) = \begin{cases}
 R_{1,j,1}\left(x,\frac{(b_2-b_1)t + b_1 x}{b_1}\right), &
 \quad t \in \left[0, \frac{-b_1 x}{b_2 - b_1}\right], \vspace{5pt} \\
 R_{2,j,1}\left(x,\frac{(b_2-b_1)t + b_1 x}{b_2}\right), &
 \quad t \in \left[\frac{-b_1 x}{b_2 - b_1}, x\right].
 \end{cases}
\end{equation}
Starting from inclusions $R_{1,j,k} \in X_\infty^0(\Omega)$, it can be shown that $R_j^+ \in X_\infty^0(\Omega)$. This proves part of relations~\eqref{eq:Phi+-}--\eqref{eq:vkpm.def} for $\Phi_1(\cdot, \l)$. Similar formula for $\Phi_2(\cdot,\l)$ can be proved by performing change of variable $b_k t = (b_1 - b_2) s + b_2 x$ in the integral $\int_0^x R_{k,j,2}(x,t) e^{i b_k \l t} \,dt$ for $k \in \{1,2\}$.
\end{proof}
Next, we recall the following abstract completeness theorem for operators $L_U(Q)$ proved in~\cite{LunMal14IEOT,LunMal15JST}. For BVP~\eqref{eq:system}--\eqref{eq:Udef} it takes the following form.
\begin{theorem}[Theorem 2.3 in~\cite{LunMal14IEOT}] \label{th:compl.gen}
Assume that there exist $C,R>0$ and $m \in \bZ_+$ such
that
\begin{equation} \label{eq:Delta(+-it)>=}
 |\Delta_Q(i t)| \ge \frac{C e^{-b_1 t}}{t^m}, \qquad
 |\Delta_Q(-i t)| \ge \frac{C e^{b_2 t}}{t^m}, \qquad t > R.
\end{equation}
Then the system of root functions of the
BVP~\eqref{eq:system}--\eqref{eq:Udef} $($of the operator
$L_U(Q))$ is complete and minimal in $L^2([0,1];
\bC^2)$.
\end{theorem}
As a first application of this result and formula for the asymptotic behavior of solutions to system~\eqref{eq:system} obtained in~\cite{LunMal15JST}, we establish a simple completeness result in the case of degenerate boundary conditions of the form $y_1(0) = y_2(1) = 0$. We formulate it in more general case that covers both Dirac-type system and the case when $b_1 / b_2 \notin \bR$.
\begin{proposition} \label{prop:2x2.notR}
Let $\arg b_1 \ne \arg b_2$, let $Q_{12}$, $Q_{21}$ be continuous at the endpoints 0 and 1, and let the following condition hold
\begin{equation} \label{eq:J13=0}
 J_{32} = J_{42} = J_{13} = 0, \qquad J_{14} \ne 0,
 \qquad Q_{12}(0) Q_{21}(1) \ne 0.
\end{equation}
Then the system of root functions of the operator $L_U(Q)$ is complete in $L^2([0,1]; \bC^2)$.
\end{proposition}
\begin{proof}
Since $J_{32} = J_{42} = J_{13} = 0$ then formula~\eqref{eq:Delta.intro} for $\Delta_Q(\cdot)$ simplifies to
\begin{equation} \label{eq:Delta.J14}
 \Delta_Q(\l) = J_{12} + J_{34}e^{i(b_1+b_2)\l} + J_{14}\varphi_{22}(1, \l),
 \qquad \l \in \bC.
\end{equation}
Since $\arg b_1 \ne \arg b_2$, then there exists $z \in \bC$ such that $\Re(i b_1 z) < 0 < \Re(i b_2 z)$. Since $J_{14} \ne 0$, it easily follows (cf.~\cite[formula~(3.38)]{MalOri12} and~\cite[Proposition~3.4]{LunMal15JST}) that
\begin{equation} \label{eq:2x2.Delta+}
 |\Delta_Q(\l)| \ge C e^{\Re(i b_2 \l)}, \qquad |\arg \l - \arg z| < \eps,
 \quad |\l| \ge R,
\end{equation}
for sufficiently small $\eps>0$ and sufficiently large $R$.

Next, we estimate $\Delta_Q(\l)$ on the ray $\arg \l = -\arg z$. Condition $\Re(i b_1 z) < 0 < \Re(i b_2 z)$ implies that
\begin{equation} \label{eq:argl=-argz}
 \Re(i b_2 \l) < 0 < \Re(i b_1 \l), \qquad \arg \l = -\arg z.
\end{equation}
According to~\cite[Proposition~3.2]{LunMal15JST} system~\eqref{eq:system} has a matrix solution $Y(x,\l) = \begin{psmallmatrix} y_{11}(x,\l) & y_{12}(x,\l) \\ y_{21}(x,\l) & y_{22}(x,\l) \end{psmallmatrix}$ satisfying
\begin{equation}
 Y(0,\l) = \begin{pmatrix}
 1 & \frac{b_1 Q_{12}(0)+o(1)}{b_1-b_2} \cdot \frac{1}{\l} \\
 0 & 1 \end{pmatrix} \quad\text{as}\quad \l \to \infty, \quad \arg \l = -\arg z,
\end{equation}
and
\begin{equation}
 Y(1,\l) = \begin{pmatrix}
 (1+o(1)) \cdot e^{i b_1 \l} & 0 \\
 \frac{b_2 Q_{21}(1)+o(1)}{b_2-b_1} \cdot \frac{e^{i b_1 \l}}{\l}
 & (1+o(1)) \cdot e^{i b_2 \l}
 \end{pmatrix} \quad\text{as}\quad \l \to \infty,\ \ \arg \l = -\arg z.
\end{equation}
Evidently
\begin{equation}
 \Phi(x,\l) = Y(x,\l) [Y(0,\l)]^{-1}, \qquad \arg \l = -\arg z.
\end{equation}
Therefore,
\begin{multline} \label{eq:DeltaQ=frac.frac}
 \varphi_{22}(\l) = y_{22}(1, \l) - y_{12}(0, \l) y_{21}(1, \l) \\
 = (1+o(1)) \cdot e^{i b_2 \l}
 - \frac{b_1 Q_{12}(0)+o(1)}{b_1-b_2} \cdot
 \frac{b_2 Q_{21}(1)+o(1)}{b_2-b_1} \cdot \frac{e^{i b_1 \l}}{\l^2}
\end{multline}
as $\l \to \infty$ and $\arg \l = -\arg z$. Inserting~\eqref{eq:DeltaQ=frac.frac} into~\eqref{eq:Delta.J14} with account of inequalities~\eqref{eq:argl=-argz}
and condition $J_{14} Q_{12}(0) Q_{21}(1) \ne 0$ (see~\eqref{eq:J13=0}) now yields
\begin{equation} \label{eq:2x2.Delta-}
 |\Delta_Q(\l)| \ge C \frac{e^{\Re(i b_1 \l)}}{|\l|^2},
 \qquad \arg \l = -\arg z, \quad |\l| \ge R,
\end{equation}
for some $C > 0$ and sufficiently large $R > 0$. Now completeness and minimality of the system of root functions of the BVP~\eqref{eq:system}--\eqref{eq:Udef} follows from Theorem 2.3 in~\cite{LunMal14IEOT} due of estimates~\eqref{eq:2x2.Delta+} and~\eqref{eq:2x2.Delta-}.
\end{proof}
\begin{remark}
Proposition~\ref{prop:2x2.notR} was implicitly contained in~\cite{LunMal15JST} but was not included due to the lack of space.
Note that it significantly relaxes smoothness assumptions compared to~\cite[Theorems 1.5]{AgiMalOri12} where the same result was proved for analytic $Q(\cdot)$ and non-Dirac-type case $(b_1 / b_2 \notin \bR)$.

In a very recent paper~\cite{KosShk21} A.P.\;Kosarev and A.A.\;Shkalikov considered the case of $2 \times 2$ Dirac-type operators with non-constant matrix $B = \diag(b_1(x), b_2(x))$ and degenerate boundary conditions of special form $(y_1(0) = y_2(1) = 0)$. Namely, in~\cite[Theorem 4]{KosShk21} they established completeness property of the root vectors system of the corresponding BVP under the smoothness assumption $b_1, b_2, Q_{12}, Q_{21} \in W_1^1[0,1]$ assuming the same algebraic condition $Q_{12}(0) Q_{21}(1) \ne 0$. Note that in the case of a constant matrix $B$, Proposition~\ref{prop:2x2.notR} ensures the completeness property under slightly relaxed smoothness assumption on the potential matrix $Q(\cdot)$: continuity at the endpoints 0 and 1 instead of absolute continuity on $[0,1]$.
\end{remark}
\section{Asymptotic expansion of solutions to the Dirac type system}
\label{sec:trace}
Here we obtain announced in~\cite{LunMal13Dokl}, refined asymptotic formulas for the solutions to the system~\eqref{eq:system} following Marchenko's method from~\cite{Mar77}, where these expansions were obtained for the classical Dirac system under the assumptions $Q \in C^n([0,1]; \bC^{2 \times 2})$, $n \ge 0$, and $Q \in W_2^n([0,1]; \bC^{2 \times 2})$, $n \ge 1$ (see problems 1 and 3 in~\cite[\S1.4]{Mar77}, respectively).

We generalize these results to Dirac-type systems and relax smoothness assumption to condition $Q \in W_1^n([0,1]; \bC^{2 \times 2})$, $n \ge 0$. Namely, condition $Q \in L^1([0,1]; \bC^{2 \times 2})$ is covered by our Corollary~\ref{cor:Phi+-} (the case $n=0$), while the case $n \ge 1$ is treated in Theorem~\ref{th:asymp} below, where the space $X_\infty^0$ again plays significant role.

Recall that $B = \diag(b_1, b_2)$, where $b_1 < 0 < b_2$. To formulate the next result, we set
\begin{equation} \label{eq:b+b-q+q-}
 b^+:=i b_1, \quad b^-:=i b_2, \quad q_+(x):=-i b_1 Q_{12}(x),
 \quad q_-(x):=-i b_2 Q_{21}(x).
\end{equation}
Then, system~\eqref{eq:system} takes the form
\begin{equation}\label{eq:system.new}
 \begin{cases}
 y_1' = b^+ \l y_1 + q_+(x) y_2, \\
 y_2' = b^- \l y_2 + q_-(x) y_1.
 \end{cases}
\end{equation}
\begin{theorem} \label{th:asymp}
Let $n \in \bN$ and $q_+, q_- \in W_1^n[0,1]$. Then system~\eqref{eq:system.new} has a matrix solution of the form
\begin{equation}\label{eq:Y(x,l).def}
 Y(x,\l) =
 \begin{pmatrix}
 e^{b^+ \l x}u_1^+(x,\l) & e^{b^- \l x}u_2^-(x,-\l) \\
 e^{b^+ \l x}u_2^+(x,\l) & e^{b^- \l x}u_1^-(x,-\l)
 \end{pmatrix}, \quad x \in [0,1], \quad \l \in \bC \setminus \{0\},
\end{equation}
where the functions $u_j^\pm (x,\l)$, $j\in \{1,2\}$, admit the following representation
\begin{align}
 \label{eq:u1}
 & u_1^\pm (x,\l) =
 \sum_{k=0}^n \frac{b_k^\pm (x)}{(i (b_1 - b_2) \l)^k}
 + \frac{b_n^\pm (x,\l)}{(i (b_1 - b_2) \l)^n}, \\
 \label{eq:u2}
 & u_2^\pm (x,\l) =
 \sum_{k=0}^n \frac{a_k^\pm (x)}{(i (b_1 - b_2) \l)^k}
 + \frac{a_n^\pm (x,\l)}{(i (b_1 - b_2) \l)^n},
\end{align}
where
\begin{align}
 \label{eq:a0b0}
 & a_0^\pm (x) = 0, \qquad b_0^\pm (x) = 1, \\
 \label{eq:ak}
 & a_k^\pm (x) = q_\mp(x) b_{k-1}^\pm (x) - \bigl(a_{k-1}^\pm (x)\bigr)',
 \quad k \in \oneton, \\
 \label{eq:bk}
 & b_k^\pm (x) = \int_0^x q_\pm (t) a_{k}^\pm (t)dt,
 \quad k \in \oneton,
\end{align}
Besides, the functions $a_n^\pm (x,\l)$ and $b_n^\pm (x,\l)$ are given by
\begin{eqnarray} \label{eq:an(x,l).def}
 a_n^\pm (x,\l) & = & \int_0^x A_n^\pm (x,t) e^{i (b_2 - b_1) \l t}dt, \\
 \label{eq:bn(l)}
 b_n^\pm (x,\l) & = & \int_0^x \left(\int_t^x A_n^\pm (s,t) q_\pm (s) ds\right) e^{i (b_2 - b_1) \l t}dt,
\end{eqnarray}
where $A_n^\pm \in X_\infty^0$ is the unique solution to the linear integral equation
\begin{equation} \label{eq:A.int.int}
 A_n^\pm (x,\xi) = v_n^\pm (x-\xi) + \int_0^{\xi} \int_t^{t+x-\xi}
 q_\mp(t+x-\xi) q_\pm (s) A_n^\pm (s,t) \,ds \,dt,
\end{equation}
for $0 \le \xi \le x \le 1$, and $v_n^\pm (\cdot) := - \left(a_n^\pm (\cdot)\right)' + q_\mp(\cdot) b_n^\pm (\cdot) \in L^1[0,1]$.
\end{theorem}
\begin{proof}[Proof]
Substituting the first column of~\eqref{eq:Y(x,l).def}
into~\eqref{eq:system.new} we obtain
\begin{eqnarray}
 b^+ \l \ e^{b^+ \l x} u_1^+ + e^{b^+ \l x} (u_1^+)'
 & = & b^+ \l \ e^{b^+ \l x} u_1^+ + q_+(x) e^{b^+ \l x} u_2^+, \\
 b^+ \l \ e^{b^+ \l x} u_2^+ + e^{b^+ \l x} (u_2^+)'
 & = & b^- \l \ e^{b^+ \l x} u_2^+ + q_-(x) e^{b^+ \l x} u_1^+,
\end{eqnarray}
or
\begin{align}
 \label{eq:u1+'=u2+}
 \bigl(u_1^+(x,\l)\bigr)' & = q_+(x) u_2^+(x,\l), \\
 \label{eq:u2+'=u1+}
 \bigl(u_2^+(x,\l)\bigr)' & = q_-(x) u_1^+(x,\l) - (b^+ - b^-) \l u_2^+(x,\l).
\end{align}
Similarly, inserting second column of~\eqref{eq:Y(x,l).def}
into~\eqref{eq:system.new} we arrive at
\begin{eqnarray}
 \label{eq:u1-'=u2-}
 \bigl(u_1^-(x,-\l)\bigr)' & = & q_-(x) u_2^-(x,-\l), \\
 \label{eq:u2-'=u1-}
 \bigl(u_2^-(x,-\l)\bigr)' & = & q_+(x) u_1^-(x,-\l) - (b^- - b^+) \l u_2^-(x,-\l).
\end{eqnarray}
Changing $\l$ to $-\l$ in~\eqref{eq:u1-'=u2-}
and~\eqref{eq:u2-'=u1-} we can combine~\eqref{eq:u1+'=u2+},
\eqref{eq:u2+'=u1+}, \eqref{eq:u1-'=u2-} and~\eqref{eq:u2-'=u1-}
into
\begin{eqnarray}
 \label{eq:u1'=u2}
 \bigl(u_1^\pm (x,\l)\bigr)' & = & q_\pm (x) u_2^\pm (x,\l), \\
 \label{eq:u2'=u1}
 \bigl(u_2^\pm (x,\l)\bigr)' & = & q_\mp(x) u_1^\pm (x,\l)
 - (b^+ - b^-) \l u_2^\pm (x,\l).
\end{eqnarray}
Next we show that functions
\begin{equation}
a_k^\pm (\cdot), \ b_k^\pm (\cdot), \quad k\in\{0,1, \ldots, n\},
\end{equation}
are well defined by formulas~\eqref{eq:a0b0}--\eqref{eq:bk}. To this end we prove by induction
that
\begin{equation}\label{eq:ak,bk in W}
 a_k^\pm (\cdot) \in W_1^{n-k+1}[0,1], \quad
 b_k^\pm (\cdot) \in W_1^{n-k+2}[0,1], \quad k \in \{0, 1, \ldots, n\}.
\end{equation}
For $k=0$ this statement is clear in view of~\eqref{eq:a0b0}. Assume that~\eqref{eq:ak,bk in W} holds for some $k \in \{0,1, \ldots, n-1\}$ and let us prove it for $k+1$. From~\eqref{eq:ak} we have
\begin{equation}\label{eq:ak+1}
 a_{k+1}^\pm (\cdot) = q_\mp(\cdot) b_k^\pm (\cdot)
 - \bigl(a_{k}^\pm (\cdot)\bigr)' \in W_1^{n-k}[0,1],
\end{equation}
since $a_{k}^\pm \in W_1^{n-k+1}[0,1]$ (and hence
$(a_{k}^\pm )' \in W_1^{n-k}[0,1]$), $q_\mp \in
W_1^{n}[0,1]$, and $b_k^\pm \in W_1^{n-k+1}[0,1]$ for $0 \le
k < n$. From~\eqref{eq:bk} we have
\begin{equation}\label{eq:b1 in W}
 b_{k+1}^\pm (x) = \int_0^x q_\pm (t) a_{k+1}^\pm (t) dt \in W_1^{n-k+1}[0,1],
\end{equation}
because $q_\pm \in W_1^n[0,1]$ and $a_{k+1}^\pm \in
W_1^{n-k}[0,1]$. Thus~\eqref{eq:ak,bk in W} holds and we have as
a consequence that
\begin{equation}\label{eq:ak,bk in W1}
 a_k^\pm (\cdot), b_k^\pm (\cdot) \in W_1^1[0,1] ,
 \quad k \in \{0,1, \ldots, n\}.
\end{equation}
Let's rewrite equations~\eqref{eq:u1'=u2} and~\eqref{eq:u2'=u1} in terms of functions $a_n^\pm (x,\l)$ and $b_n^\pm (x,\l)$. Inserting~\eqref{eq:u1} and~\eqref{eq:u2} into~\eqref{eq:u1'=u2} and setting
\begin{equation} \label{eq:c.def}
 c := b^+ - b^- = i (b_1 - b_2),
\end{equation}
we get
\begin{multline}\label{eq:sum bk' = sum ak}
 \sum_{k=0}^n \bigl(b_k^\pm (x)\bigr)' (c\l)^{-k}
 + \bigl(b_n^\pm (x,\l)\bigr)' (c\l)^{-n}
 \\= \sum_{k=0}^n q_\pm (x) a_k^\pm (x) (c\l)^{-k}
 + q_\pm (x) a_n^\pm (x,\l) (c\l)^{-n}.
\end{multline}
In view of~\eqref{eq:a0b0} and~\eqref{eq:bk} we have
\begin{equation}\label{eq:bk'=q ak}
 \bigl(b_k^\pm (x)\bigr)' = q_\pm (x) a_k^\pm (x),
 \quad k \in \{0,1, \ldots, n\}.
\end{equation}
Hence~\eqref{eq:sum bk' = sum ak} is equivalent to
\begin{equation} \label{eq:bn(l)'=an(l)}
 \bigl(b_n^\pm (x,\l)\bigr)' = q_\pm (x) a_n^\pm (x,\l).
\end{equation}
Inserting~\eqref{eq:u1} and~\eqref{eq:u2}
into~\eqref{eq:u2'=u1} we derive
\begin{multline}
 \sum_{k=0}^n \bigl(a_k^\pm (x)\bigr)' (c\l)^{-k}
 + \bigl(a_n^\pm (x,\l)\bigr)' (c\l)^{-n} \\
 = \sum_{k=0}^n q_\mp(x) b_k^\pm (x) (c\l)^{-k}
 + q_\mp(x) b_n^\pm (x,\l) (c\l)^{-n} \\
 - \sum_{k=0}^n a_k^\pm (x) (c\l)^{-k+1} - a_n^\pm (x,\l) (c\l)^{-n+1}.
\end{multline}
This identity with account of~\eqref{eq:a0b0} can be rewritten as
\begin{multline}\label{eq:sum_0{n-1}=a,k}
 \sum_{k=0}^{n-1}\left(\left(a_k^\pm (x)\right)'
 - q_\mp(x)b_k^\pm (x) + a_{k+1}^\pm (x)\right)(c \l)^{-k} =
 \Bigl(- \left(a_n^\pm (x)\right)' - \left(a_n^\pm (x,\l)\right)' \Bigr. \\
 \Bigl. + q_\mp(x) b_n^\pm (x,\l) + q_\mp(x) b_n^\pm (x)
 - c \l a_n^\pm (x,\l)\Bigr)(c\l)^{-n}.
\end{multline}
Due to~\eqref{eq:ak} each summand in the left hand side of~\eqref{eq:sum_0{n-1}=a,k} with $k \ge 1$ vanishes. Besides, combining~\eqref{eq:ak} with~\eqref{eq:a0b0} ensures that the first summand (with $k=0$) in~\eqref{eq:sum_0{n-1}=a,k} also vanishes and the identity~\eqref{eq:sum_0{n-1}=a,k} turns into
\begin{equation} \label{eq:an(l)'}
 \bigl(a_n^\pm (x,\l)\bigr)' = q_\mp(x) b_n^\pm (x,\l)
 - c \l a_n^\pm (x,\l) + v_n^\pm (x).
\end{equation}
where
\begin{equation} \label{eq:v}
 v_n^\pm (x) := - \left(a_n^\pm (x)\right)' + q_\mp(x) b_n^\pm (x), \qquad n\ge 1.
\end{equation}
It follows from~\eqref{eq:ak,bk in W1} that $v_n^\pm \in
L^1[0,1]$.

Next we show that combining~\eqref{eq:an(x,l).def} with~\eqref{eq:bn(l)} implies~\eqref{eq:bn(l)'=an(l)} for any function $A_n^\pm (\cdot,\cdot)\in L^1(\Omega)$ where $\Omega = \{(x,t):0 \le t \le x \le 1\}$ is given by~\eqref{eq:Omega.def}. Indeed, changing order of integration and taking definition~\eqref{eq:an(x,l).def} into account we get
\begin{equation}\label{eq:bnl=int anl}
 b_n^\pm (x,\l) = \int_0^x q_\pm (s)
 \left(\int_0^s A_n^\pm (s,t) e^{-c \l t}dt\right) ds
 = \int_0^x q_\pm (s) a_n^\pm (s,\l) ds.
\end{equation}
Differentiating this identity leads to~\eqref{eq:bn(l)'=an(l)}.

Making change of variables in~\eqref{eq:an(x,l).def} we get
\begin{equation} \label{eq:a_n(x,l).with.x-u}
 a_n^\pm (x,\l) = \int_0^x A_n^\pm (x,x-u) e^{-c\l(x-u)} du.
\end{equation}
Assume for the moment that (we prove this fact later)
\begin{equation} \label{eq:A.in.AC}
 A_n^\pm (\cdot,\cdot-u) \in W_1^1[u,1], \quad \text{for a.e.} \ \ u \in [0,1].
\end{equation}
Assuming this, we can differentiate in~\eqref{eq:a_n(x,l).with.x-u} for almost every $x \in [0,1]$ and with account of~\eqref{eq:an(x,l).def}
obtain
\begin{multline}\label{eq:a'+cla}
 \left(a_n^\pm (x,\l)\right)' = A_n^\pm (x,0) \\
 + \int_0^x \left(\frac{d}{dx}\bigl(A_n^\pm (x,x-u)\bigr) e^{-c \l (x-u)}
 - c \l A_n^\pm (x,x-u) e^{-c \l (x-u)} \right) du \\
 = A_n^\pm (x,0) - c \l a_n^\pm (x,\l)
 + \int_0^x \frac{d}{dx}\bigl(A_n^\pm (x,x-u)\bigr) e^{-c \l
 (x-u)}du.
\end{multline}
Inserting~\eqref{eq:bn(l)} and~\eqref{eq:a'+cla}
into~\eqref{eq:an(l)'} and making change of variable in the
integral at the right-hand side we arrive at the following equation
\begin{multline}\label{eq:A int}
 A_n^\pm (x,0) + \int_0^x
 \frac{d}{dx}\bigl(A_n^\pm (x,x-u)\bigr) e^{-c \l (x-u)} du \\
 = v_n^\pm (x) + q_\mp(x) \int_0^x
 \left(\int_{x-u}^x A_n^\pm (s,x-u) q_\pm (s) ds\right) e^{-c \l
 (x-u)}du
\end{multline}
Evidently, this equation will turn into identity if we set
\begin{gather}
 \label{eq:Ax0}
 A_n^\pm (x,0) = v_n^\pm (x), \quad x \in [0,1], \\
 \label{eq:A int dif}
 \frac{d}{dx}\bigl(A_n^\pm (x,x-u)\bigr)
 = q_\mp(x) \int_{x-u}^x A_n^\pm (s,x-u) q_\pm (s) ds,
 \quad 0 \le u \le x \le 1.
\end{gather}
Thus, it suffices to show that integro-differential equation~\eqref{eq:A int dif} with initial condition~\eqref{eq:Ax0} has a solution $A_n^\pm(\cdot,\cdot) \in X_\infty^0(\Omega)$.

Changing $x$ to $t$ in~\eqref{eq:A int dif} and integrating thus
obtained equality over $t\in [u, x]$ we get
\begin{equation}
 \int_u^x \frac{d}{dt}\bigl(A_n^\pm (t,t-u)\bigr) dt
 = \int_u^x q_\mp(t) \int_{t-u}^t A_n^\pm (s,t-u) q_\pm (s) ds dt
\end{equation}
or in view of~\eqref{eq:A.in.AC} and~\eqref{eq:Ax0}
\begin{multline}\label{eq:A int int u}
 A_n^\pm (x,x-u) - v_n^\pm (u)
 = \int_0^{x-u} \int_t^{t+u} q_\mp(t+u) q_\pm (s) A_n^\pm (s,t) ds dt, \\
 0 \le u \le x \le 1.
\end{multline}
Setting $u = x-\xi$ here we easily derive
\begin{multline}\label{eq:A int int}
 A_n^\pm (x,\xi) = v_n^\pm (x-\xi) + \int_0^{\xi} \int_t^{t+x-\xi}
 q_\mp(t+x-\xi) q_\pm (s) A_n^\pm (s,t) ds dt, \\
 0 \le \xi \le x \le 1.
\end{multline}
It is clear that~\eqref{eq:A int int} is equivalent to the
initial value problem~\eqref{eq:Ax0}--\eqref{eq:A int dif}.
Hence we need to prove that the integral equation~\eqref{eq:A
int int} has a solution in the class $X_\infty^0$.

Note that for $x \in [0,1]$ one has
\begin{equation}
 \|v_n^\pm (x-\cdot)\|_{L^1[0,x]} = \int_0^x \abs{v_n^\pm (x-t)} dt
 = \int_0^x \abs{v_n^\pm (t)} dt \le \|v_n^\pm \|_{L^1[0,1]}.
\end{equation}
Hence the function $V_n^\pm (x,\xi) := v_n^\pm (x-\xi)$ belongs to the space $X_\infty$. It is also straightforward to show that $V_n^\pm \in X_\infty^0$ by approximating $v_n^\pm(\cdot)$ in $L^1$-norm with continuous functions.

Consider the following operator $T^\pm$ from $X_\infty$ into $X_\infty$,
\begin{equation} \label{eq:Tpm.def}
 (T^\pm F)(x,\xi) := \int_0^{\xi} \int_t^{t+x-\xi} q_\mp(t+x-\xi) q_\pm (s) F(s,t) ds dt,
 \quad 0 \le \xi \le x \le 1.
\end{equation}
Integral equation~\eqref{eq:A int int} can be rewritten in the following abstract form
\begin{equation} \label{eq:(I-K)A=f}
 (I_H - T^\pm ) A_n^\pm = V_n^\pm ,
\end{equation}
where $V_n^\pm$ is a fixed element of $X_\infty^0$. Let us prove existence of $A_n^\pm \in X_\infty^0$ satisfying~\eqref{eq:(I-K)A=f}.

At first let us rewrite the definition of $T^\pm$ in more convenient way
\begin{multline} \label{eq:K.cool}
 (T^\pm F)(x_1,x_2)
 = \iint\limits_{\genfrac{}{}{0pt}{2}{0 \le t_2 \le x_2}{0 \le t_1-t_2 \le x_1-x_2}}
 q_\mp(t_2+x_1-x_2) q_\pm (t_1) F(t_1,t_2) dt_1 dt_2, \\
 0 \le x_2 \le x_1 \le 1.
\end{multline}
Recall that $q_\pm \in W_1^n[0,1] \subset C[0,1]$. Let us prove by induction in $k$ that
\begin{multline} \label{eq:K^(k+1)A<=...}
 |((T^\pm )^{k+1}F)(x_1,x_2)| \\
 \le C^{k+1} \iint\limits_{\atop{0 \le t_2 \le x_2}
 {0 \le t_1-t_2 \le x_1-x_2}}
 \frac{(x_2-t_2)^k}{k!} \cdot \frac{((x_1-x_2)-(t_1-t_2))^k}{k!}
 \cdot |F(t_1,t_2)| dt_1 dt_2, \\
 0 \le x_2 \le x_1 \le 1.
\end{multline}
where
\begin{equation} \label{eq:C.def}
 C = \|q_+\|_{C[0,1]} \cdot \|q_-\|_{C[0,1]}.
\end{equation}
The base case $k=0$ directly follows from~\eqref{eq:K.cool}
and~\eqref{eq:C.def}. Now assuming that~\eqref{eq:K^(k+1)A<=...}
is proved for $k=l-1$ we prove it for $k=l$, where $l \in
\bN$. We have for $0 \le x_2 \le x_1
\le 1$
\begin{multline}
 \abs{((T^\pm )^{l+1}F)(x_1,x_2)}
 \le C \iint\limits_{\atop{0 \le t_2 \le x_2}{0 \le t_1-t_2 \le x_1-x_2}}
 |((T^\pm )^{l}F)(t_1,t_2)| dt_1 dt_2 \\
 \le C^{l+1} \iiiint\limits_{\atop{0 \le s_2 \le t_2 \le x_2}
 {0 \le s_1-s_2 \le t_1-t_2 \le x_1-x_2}}
 \frac{(t_2-s_2)^{l-1}}{(l-1)!} \times
 \frac{((t_1-t_2)-(s_1-s_2))^{l-1}}{(l-1)!} \\
 \times |F(s_1,s_2)| ds_1 ds_2 dt_1 dt_2
 = C^{l+1} \iint\limits_{\atop{0 \le s_2 \le x_2}
 {0 \le s_1-s_2 \le x_1-x_2}} |F(s_1,s_2)| ds_1 ds_2 \\
 \iint\limits_{\atop{s_2 \le t_2 \le x_2}
 {s_1-s_2 \le t_1-t_2 \le x_1-x_2}} \frac{(t_2-s_2)^{l-1}}{(l-1)!}
 \times \frac{((t_1-t_2)-(s_1-s_2))^{l-1}}{(l-1)!} dt_1 dt_2 \\
 = C^{l+1} \iint\limits_{\atop{0 \le s_2 \le x_2}
 {0 \le s_1-s_2 \le x_1-x_2}} \frac{(x_2-s_2)^{l}}{l!}
 \times \frac{((x_1-x_2)-(s_1-s_2))^{l}}{l!} \times |F(s_1,s_2)| ds_1 ds_2 \\
 \le \frac{C^{l+1}}{l!^2} \iint\limits_{\atop{0 \le s_2 \le x_2}
 {0 \le s_1-s_2 \le x_1-x_2}} |F(s_1,s_2)| ds_1 ds_2.
\end{multline}
So~\eqref{eq:K^(k+1)A<=...} is proved. Now by~\eqref{eq:K^(k+1)A<=...} we have for $k \in \bZ_+$ and $x_1 \in [0,1]$
\begin{multline}
 \|((T^\pm )^{k+1} F)(x_1,\cdot)\|_{L^1[0,x_1]}
 = \int_0^{x_1} \abs{((T^\pm )^{k+1} F)(x_1,x_2)} dx_2 \\
 \le \frac{C^{k+1}}{k!^2}
 \iiint\limits_{\atop{0 \le t_2 \le x_2 \le x_1}
 {0 \le t_1-t_2 \le x_1-x_2}} |F(t_1,t_2)| dt_1 dt_2 dx_2 \\
 \le \frac{C^{k+1}}{k!^2}
 \int_0^{x_1} \int_0^{x_1} \left(\int_0^{t_1} \abs{F(t_1,t_2)} dt_2 \right)dt_1 dx_2 \\
 \le \frac{C^{k+1}}{k!^2}
 \int_0^{x_1} \int_0^{x_1} \|F\|_{X_\infty}\,dt_1 dx_2
 \le \frac{C^{k+1}}{k!^2} \|F\|_{X_\infty}.
\end{multline}
Hence
\begin{equation}
 \|(T^\pm )^{k+1} F\|_{X_\infty} \le \frac{C^{k+1}}{k!^2} \|F\|_{X_\infty},
 \quad F \in X_\infty, \quad k \in \bZ_+,
\end{equation}
or
\begin{equation} \label{eq:||K^(k+1)A||}
 \|(T^\pm )^{k}\| \le \frac{C^{k}}{((k-1)!)^2}, \quad k \in \bN.
\end{equation}

This implies that the operator $T^\pm$ has zero spectral radius. Hence
the operator $I_H - T^\pm$ has bounded inverse and $A_n^\pm = (I_H - T^\pm )^{-1} V_n^\pm \in X_\infty$.

Let us prove that, in fact, the operator $T^\pm$ maps $X_\infty^0$ into itself. Let $F \in X_\infty^0$ and let $F_m \to F$ in $X_\infty$ as $m \to \infty$ where $F_m \in C(\Omega)$, $m \in \bN$. Since $q_\pm \in C[0,1]$, it is evident from~\eqref{eq:Tpm.def} that $T^\pm F_m \in C(\Omega)$, $m \in \bN$. Combining~\eqref{eq:Tpm.def} with~\eqref{eq:C.def} we get for $0 \le \xi \le x \le 1$,
\begin{multline} \label{eq:TF-Tfm}
 \abs{(T^\pm F)(x,\xi) - (T^\pm F_m)(x,\xi)}
 \le C \int_0^{\xi} \int_t^{t+x-\xi} \abs{F(s,t) - F_m(s,t)} \,ds \,dt \\
 = C \int_0^x \int_{\max\{0, s-x+\xi\}}^{\min\{s,\xi\}}
 \abs{F(s,t) - F_m(s,t)} \,dt \,ds \\
 \le C \int_0^x \|F(s, \cdot) - F_m(s, \cdot)\|_{L^1[0,s]} \, ds
 \le C \|F - F_m\|_{X_\infty}.
\end{multline}
This estimate trivially implies that $\|T^\pm F_m - T^\pm F\|_{X_\infty} \to 0$ as $m \to \infty$. Hence $T^\pm F \in X_\infty^0$. This concludes the proof of the desired inclusion $A_n^\pm \in X_\infty^0$.

Finally let's prove~\eqref{eq:A.in.AC}. In view of~\eqref{eq:A
int int u} it is equivalent to
\begin{equation} \label{eq:B.in.L1}
 B^\pm (\cdot,u) \in L^1[0,1-u], \quad u \in [0,1],
\end{equation}
where
\begin{equation} \label{eq:Btu.def}
 B^\pm (t,u) := \int_t^{t+u} q_\mp(t+u) q_\pm (s) A_n^\pm (s,t)\,ds,
 \quad t,u \ge 0,\ \ t+u \le 1.
\end{equation}
Combining~\eqref{eq:C.def} with inclusion $A_n^\pm \in X_\infty$ and doing estimation similarly to~\eqref{eq:TF-Tfm} implies
\begin{equation} \label{eq:intBtu}
 \int_0^{1-u} |B^\pm (t,u)|dt
 \le C \int_0^{1-u} \int_t^{t+u} |A_n^\pm (s,t)| \,ds \,dt
 \le C \|A_n^\pm\|_{X_\infty}.
\end{equation}
Hence~\eqref{eq:B.in.L1} holds and the proof is complete.
\end{proof}
Lemma~\ref{lem:K.exp.X0} allows us to estimate growth of remainder terms $a_n^\pm(x,\l)$ and $b_n^\pm(x,\l)$ in expansions~\eqref{eq:u1}--\eqref{eq:u2} in the upper half-plane.
\begin{lemma} \label{lem:anbn=o1}
Let $h \in \bR$ be fixed. Then under the assumption of Theorem~\ref{th:asymp} the functions $a_n^\pm(x,\l)$ and $b_n^\pm(x,\l)$ given by~\eqref{eq:an(x,l).def}--\eqref{eq:bn(l)} satisfy the following asymptotic relations uniformly in $x \in [0,1]$,
\begin{equation} \label{eq:anbn=o1}
 a_n^\pm(x,\l) = o(1), \qquad b_n^\pm(x,\l) = o(1),
 \qquad \l \to \infty,\ \ \Im \l \ge h.
\end{equation}
\end{lemma}
\begin{proof}
Let $\delta > 0$. Combining~\eqref{eq:an(x,l).def} with the inclusion $A_n^\pm \in X_\infty^0$, Lemma~\ref{lem:K.exp.X0} and inequality $b_1 - b_2 < 0$, we have
\begin{multline} \label{eq:anmpxl}
 \abs{a_n^\pm (x,\l)}
 = \abs{\int_0^x A_n^\pm (x,t) e^{i (b_2 - b_1) \l t} dt}
 \le \delta \cdot (e^{(b_1 - b_2) \Im \l} + 1) \\
 \le \delta \cdot (e^{(b_1 - b_2) h} + 1),
 \qquad \Im \l \ge h, \quad|\l| > R_\delta, \quad x \in [0,1].
\end{multline}
This trivially implies the first relation in~\eqref{eq:anbn=o1}.

Recall that $q_\pm \in W_1^n[0,1] \subset C[0,1]$. Hence
formula~\eqref{eq:bnl=int anl} implies estimate
\begin{equation} \label{eq:|bnl|<=sup|q|.sup|anl|}
 \abs{b_n^\pm (x,\l)}
 \le \|q_\pm \|_{C[0,1]} \cdot \sup_{t \in [0,1]}|a_n^\pm (t,\l)|,
 \quad \l \in \bC \setminus \{0\}.
\end{equation}
Combining the first relation in~\eqref{eq:anbn=o1}
with estimate~\eqref{eq:|bnl|<=sup|q|.sup|anl|} implies the second relation in~\eqref{eq:anbn=o1}.
\end{proof}
\begin{remark}
Note that in the recent papers~\cite{KosShk21,KosShk23} A.P.~Kosarev and A.A.~Shkalikov established asymptotic expansions similar to~\eqref{eq:u1}--\eqref{eq:u2} with estimate of the remainder term as in~\eqref{eq:anbn=o1} for $2 \times 2$ Dirac-type operators with non-constant matrix $B(\cdot)$ using Birkhoff-Tamarkin method (see also~\cite{SavShk20}).
Unique feature of Marchenko method, that we used in the proof of formulas~\eqref{eq:u1}--\eqref{eq:u2}, is the integral representation of the remainder term as a Fourier transform of the solution to certain explicit integral equation.
\end{remark}
Formulas~\eqref{eq:ak} and~\eqref{eq:bk} contain undesirable
integration operations, which are avoidable by considering the
quotient $u_2^\pm (x,\l)[u_1^\pm (x,\l)]^{-1}$.

First we need the following abstract result of folklore nature on the asymptotic behavior of rational functions.
\begin{lemma} \label{lem:P1/P2}
Let holomorphic in unbounded domain $D \subset \bC \setminus \{0\}$ functions $u_1(\l)$ and $u_2(\l)$ be given by
\begin{align}
\label{eq:P1.def}
 & u_1(\l) = \sum_{k=0}^n b_k \cdot (c \l)^{-k} + b_n(\l) \cdot (c \l)^{-n},
 \qquad \l \in D, \\
\label{eq:P2.def}
 & u_2(\l) = \sum_{k=0}^n a_k \cdot (c \l)^{-k} + a_n(\l) \cdot (c \l)^{-n},
 \qquad \l \in D,
\end{align}
where $b_0 = 1$, $c \in \bC \setminus \{0\}$ is fixed, and
\begin{equation} \label{eq:an.bn=o1}
 a_n(\l) = o(1), \quad b_n(\l) = o(1)
 \quad\text{as}\quad \l \to \infty, \quad \l \in D.
\end{equation}
Assume that $u_1(\l) \ne 0$ for $\l \in D$. Then the following identity holds
\begin{equation} \label{eq:P2/P1}
 \frac{u_2(\l)}{u_1(\l)} = \sum_{k=0}^n \sigma_k \cdot (c \l)^{-k}
 + \sigma_n(\l) \cdot (c \l)^{-n}, \qquad \l \in D,
\end{equation}
where
\begin{equation} \label{eq:sigmak}
 \sigma_0 = a_0, \qquad \sigma_k = a_k - \sum_{j=1}^k b_j \sigma_{k-j},
 \qquad k \in \oneton,
\end{equation}
and
\begin{multline} \label{eq:sigmanl}
 \sigma_n(\l) = \frac{1}{u_1(\l)} \left(a_n(\l)
 - b_n(\l) \sum_{k=0}^n \sigma_k \cdot (c \l)^{-k} \right. \\
 \left. - \sum_{k=n+1}^{2n} \sum_{j=n-k}^n b_j \sigma_{k-j}
 \cdot (c \l)^{k-n} \right), \qquad \l \in D.
\end{multline}
In particular,
\begin{equation} \label{eq:sigmanl=o1}
 \sigma_n(\l) = o(1) \quad\text{as}\quad \l \to \infty, \quad \l \in D.
\end{equation}
\end{lemma}
\begin{proof}
Multiplying~\eqref{eq:P2/P1} by $u_1(\l)$, expanding the product of two sums and comparing coefficients at powers of $\l$ one can arrive at formulas~\eqref{eq:sigmak}--\eqref{eq:sigmanl} after straightforward calculations. Since $b_0=1$, it is evident that $u_1(\l) = 1 + o(1)$ as $\l \to \infty$, $\l \in D$. Inserting this asymptotic formula and formulas~\eqref{eq:an.bn=o1} into~\eqref{eq:sigmanl} one readily arrives at~\eqref{eq:sigmanl=o1}.
\end{proof}
First we apply Lemma~\ref{lem:P1/P2} to establish asymptotic expansion of the quotient $\frac{u_2^\pm (x,\l)}{u_1^\pm (x,\l)}$ with coefficients $\sigma^\pm _k(x)$ that are expressed via $a_k^\pm(x), b_k^\pm(x)$.
\begin{proposition} \label{prop:sigma.easy}
Under the assumptions of Theorem~\ref{th:asymp} consider
the functions $u_2^\pm (x,\l)$ and $u_1^\pm (x,\l)$ $($see~\eqref{eq:u1} and~\eqref{eq:u2}$)$ from the representation of the fundamental solution $Y(x,\l)$ of the form~\eqref{eq:Y(x,l).def} to system~\eqref{eq:system.new}.
Let $h \in \bR$ be fixed. Then the following representation holds for the quotient of these functions
\begin{equation} \label{eq:sigma}
 \sigma^\pm (x,\l) := \frac{u_2^\pm (x,\l)}{u_1^\pm (x,\l)}
 = \sum_{k=1}^{n} \frac{\sigma^\pm _k(x)}{(c\l)^{k}}
 + \frac{\sigma^\pm _n(x,\l)}{(c\l)^{n}},
 \quad \l \in \bC \setminus \{0\},
\end{equation}
where $c = b^+ - b^- = i (b_1 - b_2)$ $($see~\eqref{eq:c.def}$)$ and
\begin{align}
\label{eq:sigmak=ak.bk}
 &\sigma_k^\pm(x) = a_k^\pm(x) - \sum_{j=1}^{k-1} b_j^\pm(x)
 \sigma_{k-j}^\pm(x), \quad x \in [0,1],
 \quad k \in \oneton, \\
\label{eq:sigma_n(0,l)}
 & \sigma^\pm_n(0,\l) = 0,
 \qquad \l \in \bC \setminus \{0\}, \\
\label{eq:sigma_n(x,l)}
 & \sigma^\pm _n(x,\l) = o(1), \quad \l \to \infty,\ \Im \l \ge h,
 \ \ \text{uniformly in} \ \ x \in [0,1].
\end{align}
\end{proposition}
\begin{proof}
With account notation~\eqref{eq:c.def} formulas~\eqref{eq:u1}--\eqref{eq:u2} take the form
\begin{align}
 \label{eq:u1=P1n}
 u_1^\pm (x,\l) = \sum_{k=0}^n b_k^\pm (x)(c\l)^{-k}
 + b_n^\pm (x,\l)(c\l)^{-n}, \qquad \l \ne 0, \\
 \label{eq:u2=P2 n}
 u_2^\pm (x,\l) = \sum_{k=0}^n a_k^\pm (x)(c\l)^{-k}
 + a_n^\pm (x,\l)(c\l)^{-n}, \qquad \l \ne 0.
\end{align}
Lemma~\ref{lem:anbn=o1} implies asymptotic relations~\eqref{eq:anbn=o1} for $a_n^\pm (x,\l)$ and $b_n^\pm (x,\l)$.
Relations~\eqref{eq:u1=P1n}--\eqref{eq:u2=P2 n}, \eqref{eq:ak,bk in W1},
\eqref{eq:anbn=o1} and equalities $a_0^\pm (x) = 0$, $b_0^\pm (x) = 1$ (see~\eqref{eq:a0b0}) imply that
\begin{equation}\label{eq:P1n=1+o(1),...}
 u_1^\pm (x,\l) = 1+o(1), \quad u_2^\pm (x,\l) = o(1)
 \quad\text{as}\quad \l \to \infty, \ \ \Im \l \ge h,
\end{equation}
uniformly in $x \in [0,1]$. Set
\begin{equation} \label{eq:Theta.h.R}
 \Theta_{h,R} := \{\l \in \bC : \Im \l \ge h,\ |\l| > R\}.
\end{equation}
The first relation in~\eqref{eq:P1n=1+o(1),...} yields
\begin{equation} \label{eq:u1.ne0}
 u_1^{\pm}(x,\l) \ne 0, \qquad \l \in \Theta_{h,R}, \quad x \in [0,1],
\end{equation}
for some $R > 0$. Hence the function $\sigma^{\pm}(x,\l)$ is correctly define on $[0,1] \times \Theta_{h,R}$.

Therefore, for each $x \in [0,1]$ Lemma~\ref{lem:P1/P2} is applicable to functions $u_1(\l) = u_1^\pm(x,\l)$ and $u_2(\l) = u_2^\pm(x,\l)$ in the domain $D = \Theta_{h,R}$. In turn, it implies formulas~\eqref{eq:sigmak=ak.bk} for the coefficients in the expansion~\eqref{eq:sigma}.

Next we establish uniform asymptotic relation $\sigma^\pm _n(x,\l) = o(1)$ as $\l \to \infty$ and $\Im \l \ge h$. Combining identity~\eqref{eq:sigmak=ak.bk} with inclusions~\eqref{eq:ak,bk in W} one can prove the following inclusions via induction:
\begin{equation} \label{eq:sigmak.in.W}
 \sigma_k^\pm \in W_1^{n-k+1}[0,1] \subset C[0,1], \quad k \in \oneton.
\end{equation}
Combining these inclusions with inclusions~\eqref{eq:ak,bk in W1}, uniform asymptotic relations~\eqref{eq:anbn=o1} and formula~\eqref{eq:sigmanl} for $\sigma_n^\pm(x,\l)$ implies uniform estimate~\eqref{eq:sigma_n(x,l)}.
Further, combining~\eqref{eq:bk} with~\eqref{eq:an(x,l).def} and~\eqref{eq:bn(l)} yields
\begin{gather} \label{eq:bk.an.bn.0}
 b_1^\pm (0) = \ldots = b_n^\pm (0) = 0, \qquad
 a_n^\pm (0,\l) = b_n^\pm (0,\l) = 0.
\end{gather}
Inserting these equalities in formula~\eqref{eq:sigmanl} for $\sigma_n^\pm(0,\l)$ trivially implies that $\sigma_n^\pm(0,\l) = 0$. The proof is now complete.
\end{proof}
\begin{remark}
Let us calculate several first values of $\sigma_k^{\pm}(\cdot)$ using formulas~\eqref{eq:sigmak=ak.bk} and~\eqref{eq:a0b0}--\eqref{eq:bk}. For brevity we omit the argument:
\begin{equation} \label{eq:sigma1.a1}
 \sigma_1^\pm = a_1^\pm = q_\mp.
\end{equation}
\begin{equation} \label{eq:sigma2.a2}
 \sigma_2^\pm
 = a_2^\pm - b_1^\pm \sigma_1^\pm
 = q_\mp b_1^\pm - (a_1^\pm)' - b_1^\pm q_\mp
 = - q_\mp'.
\end{equation}
\begin{multline} \label{eq:sigma3.a3}
 \sigma_3^\pm
 = a_3^\pm - b_1^\pm \sigma_2^\pm - b_2^\pm \sigma_1^\pm
 = q_\mp b_2^\pm - (a_2^\pm)' + b_1^\pm q_\mp' - b_2^\pm q_\mp \\
 = - (q_\mp b_1^\pm - (a_1^\pm)')' + b_1^\pm q_\mp'
 = q_\mp'' - q_\mp (b_1^\pm)'
 = q_\mp'' - q_\mp q_\pm a_1^\pm
 = q_\mp'' - q_\mp^2 q_\pm.
\end{multline}
Surprisingly, despite the fact that formulas for $b_k^\pm(\cdot)$ contain integration operations, the final formula for $\sigma_k^\pm(\cdot)$ is only expressed via functions $q_\pm$ and their derivatives.
In fact, this is valid for all $k \in \oneton$. Direct proof of this fact via formulas~\eqref{eq:sigmak=ak.bk} and~\eqref{eq:a0b0}--\eqref{eq:bk} is rather tedious and cumbersome.
We will overcome this difficulty in the next result which was announced in~\cite[Proposition 1]{LunMal13Dokl}, by using Marchenko's idea to establish equation of Ricatti type for the quotient $\sigma^\pm(x,\l)$.
\end{remark}
\begin{proposition} \label{prop:sigma}
Under the assumptions of Proposition~\ref{prop:sigma.easy} the following explicit recurrent formulas hold for the functions $\sigma_k^\pm(\cdot)$,
\begin{align}
\label{eq:sigma1}
 \sigma^\pm _1(x) & = q_\mp(x), \\
\label{eq:sigma(k+1)}
 \sigma^\pm _{k+1}(x) & = -(\sigma^\pm _k(x))'
 - q_\pm (x) \cdot \sum_{j=1}^{k-1}
 \sigma^\pm _j(x) \sigma^\pm _{k-j}(x),
\end{align}
for $k \in \oneto{n-1}$ and $x \in [0,1]$.
\end{proposition}
\begin{proof}
First we derive a differential equation on the function $\sigma^\pm (x,\l)$ given by~\eqref{eq:sigma}. From its definition we have
\begin{equation} \label{eq:u2=u1.sigma}
 u_2^\pm (x,\l) = u_1^\pm (x,\l) \sigma^\pm (x,\l).
\end{equation}
Substituting this into~\eqref{eq:u1'=u2} we get
\begin{equation} \label{eq:u1'=q.u1.sigma}
 (u_1^\pm (x,\l))'=q_\pm (x) u_1^\pm (x,\l) \sigma^\pm (x,\l).
\end{equation}
Inserting~\eqref{eq:u1'=q.u1.sigma} into~\eqref{eq:u2'=u1.with.sigma} we arrive at the second equation
\begin{equation} \label{eq:u2'=u1.with.sigma}
 (u_1^\pm )' \sigma^\pm + u_1^\pm (\sigma^\pm )'
 = q_\mp(x) u_1^\pm - c\l u_1^\pm \sigma^\pm .
\end{equation}
Excluding the derivative $(u_1^\pm )'$ from system~\eqref{eq:u1'=q.u1.sigma}--\eqref{eq:u2'=u1.with.sigma} we arrive at the following relation
\begin{equation} \label{eq:u1.with.sigma,eq}
 q_\pm (x) u_1^\pm (\sigma^\pm )^2 + u_1^\pm (\sigma^\pm )'
 = q_\mp(x) u_1^\pm - c\l u_1^\pm \sigma^\pm .
\end{equation}
Taking into account that $u_1^{\pm}(x,\l) \ne 0$ on $[0,1] \times \Theta_{h,R}$ (see~\eqref{eq:u1.ne0}) and dividing equality~\eqref{eq:u1.with.sigma,eq} by $u_1^{\pm}(x,\l)$ we arrive at the following non-linear differential equation (of Ricatti type) on $\sigma^\pm (t,\l)$,
\begin{equation} \label{eq:Rikatti}
 (\sigma^\pm (x,\l))' + c\l \sigma^\pm (x,\l)
 + q_\pm (x) (\sigma^\pm(x,\l))^2 = q_\mp(x),
 \quad x \in [0,1], \quad \l \in \Theta_{h,R}.
\end{equation}
To avoid cumbersome expression for derivate $(\sigma_n^{\pm}(x,\l))'$ when inserting expression~\eqref{eq:sigma} for $\sigma^{\pm}(x,\l)$ into equation~\eqref{eq:Rikatti} we integrate this equation first (after replacing $x$ with $t$) for $t\in [0, x]$:
\begin{multline} \label{eq:Rikatti.int}
 \sigma^\pm (x,\l) - \sigma^\pm (0,\l) + c\l \int_0^x \sigma^\pm (t,\l) dt
 + \int_0^x q_\pm (t) (\sigma^\pm (t,\l))^2 dt \\
 = \int_0^x q_\mp(t)dt, \quad x \in [0,1], \quad \l \in \Theta_{h,R}.
\end{multline}
Now substituting r.h.s. of expansion~\eqref{eq:sigma} into~\eqref{eq:Rikatti.int} implies the following equation
\begin{equation} \label{eq:ck.series}
 c^\pm _0(x) + \frac{c^\pm _1(x)}{c\l} + \ldots + \frac{c^\pm _{n-1}(x)}{(c\l)^{n-1}}
 + \frac{c^\pm _{n}(x,\l)}{(c\l)^{n}} = 0, \quad x \in [0,1], \quad \l \in \Theta_{h,R},
\end{equation}
where
\begin{align}
\label{eq:c0}
 & c_0^{\pm}(x) := \int_0^x(\sigma_1^\pm (t)-q_\mp(t))dt, \\
\label{eq:ck}
 & c_k^{\pm}(x) := \int_0^x \sigma_{k+1}^\pm (t)dt
 + \sigma_{k}^\pm (x) - \sigma_{k}^\pm (0)
 + \sum_{j=1}^{k-1} \int_0^x q_\pm (t) \sigma_{j}^\pm (t)
 \sigma_{k-j}^\pm (t) dt,
\end{align}
for $k \in \oneto{n-1}$ and
\begin{multline} \label{eq:cnxl}
 c_n^{\pm}(x,\l) := c \l \int_0^x \sigma_{n}^\pm (t,\l)dt
 + \sigma_{n}^\pm (x,\l) - \sigma_{n}^\pm (0,\l) \\
 + \sigma_{n}^\pm (x) - \sigma_{n}^\pm (0)
 + \sum_{j=1}^{n-1} \int_0^x q_\pm (t)
 \sigma_{j}^\pm (t) \sigma_{n-j}^\pm (t) dt \\
 + \sum_{k=1}^n (c\l)^{-k} \sum_{j=k}^n \int_0^x q_\pm (t)
 \sigma_{j}^\pm (t) \sigma_{n+k-j}^\pm (t) dt \\
 + \sum_{k=1}^n 2 (c\l)^{-k} \int_0^x q_\pm (t)
 \sigma_k^\pm (t) \sigma_n^\pm (t,\l) dt
 + (c\l)^{-n} \int_0^x q_\pm (t) (\sigma_n^\pm (t,\l))^2 dt.
\end{multline}

To estimate the remainder term in~\eqref{eq:ck.series} we establish the following uniform asymptotic relation:
\begin{equation} \label{eq:cnxl=o(l)}
 c_n^{\pm}(x,\l) = o(\l), \quad \l \to \infty, \ \ \Im \l \ge h,
 \quad \text{uniformly in} \ \ x \in [0,1].
\end{equation}
With account of boundedness of coefficients $\sigma_k^\pm$ and $q_\pm$
(see~\eqref{eq:sigmak.in.W}), asymptotic relation~\eqref{eq:sigma_n(x,l)} yields the following uniform in $x \in [0,1]$ relations as $\l \to \infty$ and $\Im \l \ge h$:
\begin{eqnarray}
 \label{eq:int.sigma.nxl}
 c \l \int_0^x \sigma_{n}^\pm (t,\l)dt & = & o(\l), \\
 \label{eq:sigma.nxl-sigma.n0l}
 \sigma_{n}^\pm (x,\l) - \sigma_{n}^\pm (0,\l) & = & o(1), \\
 \label{eq:sigma.nx-sigma.n0+int}
 \sigma_{n}^\pm (x) - \sigma_{n}^\pm (0)
 + \sum_{j=1}^{n-1} \int_0^x q_\pm (t)
 \sigma_{j}^\pm (t) \sigma_{n-j}^\pm (t) dt & = & O(1), \\
 \label{eq:sum.sigma.kx}
 \sum_{k=1}^n (c\l)^{-k} \sum_{j=k}^n \int_0^x
 q_\pm (t) \sigma_{j}^\pm (t) \sigma_{n+k-j}^\pm (t) dt & = & O(\l^{-1}), \\
 \label{eq:sum.sigma.nxl}
 \sum_{k=1}^n 2 (c\l)^{-k}\int_0^x q_\pm (t)
 \sigma_k^\pm (t) \sigma_n^\pm (t,\l) dt & = & o(\l^{-1}), \\
 \label{eq:int.sigma.nxl^2}
 (c\l)^{-n} \int_0^x q_\pm (t) (\sigma_n^\pm (t,\l))^2 dt & = & o(\l^{-n}).
\end{eqnarray}
Inserting asymptotic estimates~\eqref{eq:int.sigma.nxl}--\eqref{eq:int.sigma.nxl^2} into equality~\eqref{eq:cnxl} we arrive at the desired asymptotic relation~\eqref{eq:cnxl=o(l)}.

Now fixing $x=x_0 \in [0,1]$ in~\eqref{eq:ck.series}, multiplying obtained equality by $(c\l)^{n-1}$ and tending $\l$ to infinity with account of asymptotic relation~\eqref{eq:cnxl=o(l)} implies the following equalities,
\begin{equation} \label{eq:ck=0}
 c_k^\pm (x) = 0, \quad x \in [0,1], \quad k \in \oneton.
\end{equation}
Differentiating equalities~\eqref{eq:ck=0} with account of definitions~\eqref{eq:c0} and~\eqref{eq:ck} and smoothness of coefficients $\sigma_k^\pm$, we arrive at the desired relations~\eqref{eq:sigma1} and~\eqref{eq:sigma(k+1)}, which completes the proof.
\end{proof}
\begin{remark} \label{rem:sigma123456}
\textbf{(i)} Note that recurrent formulas~\eqref{eq:sigma(k+1)} for coefficients $\sigma_{k}^\pm$ provide an explicit form of local polynomial integrals of motion for non-linear Shr\"odinger equation $($see~\cite{ZMNovPit80}$)$.

\textbf{(ii)} Formula~\eqref{eq:sigma(k+1)} allows to easily and explicitly calculate values of $\sigma_k^\pm(\cdot)$ via entries $q_\pm(\cdot)$ of the potential matrix $Q(\cdot)$, which would be quite tedious to do via formulas~\eqref{eq:sigmak=ak.bk} and~\eqref{eq:a0b0}--\eqref{eq:bk}. \href{https://tinyurl.com/sagesigmak}{For instance},
\begin{align}
 & \sigma_4^\pm
 = -(\sigma_3^\pm)' - 2 q_\mp \sigma_1^\pm \sigma_2^\pm
 = -q_\mp''' + q_\mp^2 q_\pm' + 4 q_\mp' q_\mp q_\pm, \\
 & \sigma_5^\pm = q_\mp^{({4})} - 6 q_\mp'' q_\mp q_\pm
 - 5 q_\mp'q_\mp'q_\pm - 6 q_\mp' q_\mp q_\pm'
 - q_\mp^2 q_\pm'' + 2 q_\mp^3 q_\pm^2, \\
\nonumber
 & \sigma_6^\pm = - q_\mp^{(5)} + 8 q_\mp''' q_\mp q_\pm
 + 18 q_\mp'' q_\mp' q_\pm
 + 12 q_\mp'' q_\mp q_\pm'
 + 11 q_\mp' q_\mp' q_\pm' \\
 & \qquad \qquad + 8 q_\mp' q_\mp q_\pm'' + q_\mp^2 q_\pm'''
 - 16 q_\mp' q_\mp^2 q_\pm^2 - 6 q_\mp^3 q_\pm' q_\pm.
\end{align}
Moreover, one can obtain explicit expression for $\sigma_k^\pm$ as a linear combination of products of the functions $q_\pm$ and their derivatives in general case $($see Lemma~\ref{lem:sigma.q} below$)$.
\end{remark}
\begin{remark}
Asymptotic expansions of solutions to the system~\eqref{eq:system} obtained in this section play important role in establishing refined asymptotic formulas for eigenvalues and eigenfunctions of the corresponding BVP. In this connection, we mention that recently L.\;Rzepnicki~\cite{Rzep21} obtained sharp asymptotic formulas for deviations $\l_m - \l_m^0 = \delta_m + \rho_m$ in the case of Dirichlet BVP for $2 \times 2$ Dirac system
with $Q \in L^p([0,1]; \bC^{2 \times 2})$, $1 \le p < 2$. Namely, $\delta_m$ is explicitly expressed via Fourier coefficients and Fourier transforms of $Q_{12}$ and $Q_{21}$, while $\{\rho_m\}_{m \in \bZ} \in \ell^{p'/2}(\bZ)$ where $p'$ is the conjugate exponent.
Similar result was obtained for eigenvectors. For Sturm-Liouville operators with singular potentials, A.\;Gomilko and L.\;Rzepnicki obtained similar results in another recent paper~\cite{GomRze20}.
\end{remark}
\section{Asymptotic expansion of the characteristic determinant}
\label{sec:Delta}
In this section we apply Proposition~\ref{prop:sigma} to obtain asymptotic expansion of the characteristic determinant $\Delta_Q(\cdot)$ given by~\eqref{eq:Delta.intro} with explicit, easy to calculate coefficients.

Recall, that $\Phi(x,\l) = (\varphi_{jk}(x,\l))_{j,k=1}^2$ is the fundamental matrix solution to system~\eqref{eq:system} (uniquely) determined by the initial condition $\Phi(0,\l)=I_2$ (see~\eqref{eq:Phi.def}).
It is evident, that the eigenvalues of the
problem~\eqref{eq:system}--\eqref{eq:Udef} are the roots of the
characteristic equation $\Delta_Q(\l) := \det U(\l)=0$, where
\begin{equation}\label{eq:U}
 U(\l) :=
 \begin{pmatrix}
 U_1(\Phi_1(x,\l)) & U_1(\Phi_2(x,\l)) \\
 U_2(\Phi_1(x,\l)) & U_2(\Phi_2(x,\l))
 \end{pmatrix} =:
 \begin{pmatrix}
 u_{11}(\l) & u_{12}(\l) \\
 u_{21}(\l) & u_{22}(\l)
 \end{pmatrix}.
\end{equation}
With account of notations~\eqref{eq:Ajk.Jjk.def}
we arrive at the following expression for the characteristic determinant:
\begin{equation}\label{eq:Delta}
 \Delta_Q(\l) = J_{12} + J_{34}e^{i(b_1+b_2)\l}
 + J_{32}\varphi_{11}(\l) + J_{13}\varphi_{12}(\l)
 + J_{42}\varphi_{21}(\l) + J_{14}\varphi_{22}(\l),
\end{equation}
where $\varphi_{jk}(\l) := \varphi_{jk}(1,\l)$.
\begin{theorem} \label{th:Delta}
Let $Q_{12},Q_{21} \in W_1^n[0,1]$ and functions $\sigma^\pm _k(x), k \in \oneton$, be defined by the formulas~\eqref{eq:b+b-q+q-}, \eqref{eq:sigma1}, \eqref{eq:sigma(k+1)}. Then the
characteristic determinant $\Delta_Q(\l)$ of the
system~\eqref{eq:system} admits the following representations
\begin{align}
 \label{eq:Delta.in.Omega+.ck}
 \Delta_Q(\l) &= e^{i b_1 \l} \cdot \Bigl(J_{32}
 + \sum_{k=1}^n \frac{c_k^+}{(c\l)^k}
 + o(\l^{-n})\Bigr) \cdot (1+o(1)), \qquad \Im \l \to +\infty, \\
 \label{eq:Delta.in.Omega-.ck}
 \Delta_Q(\l) &= e^{i b_2 \l} \cdot \Bigl(J_{14} - \sum_{k=1}^n \frac{c_k^-}{(c\l)^k}
 + o(\l^{-n})\Bigr) \cdot (1+o(1)), \qquad \Im \l \to -\infty,
\end{align}
where $c = i (b_1 - b_2)$, and $c_k^\pm$ for $k \in \oneton$ are given by
\begin{align}
\label{eq:ck+}
 c_k^+ & := J_{13} (-1)^{k-1} \sigma_k^-(0) + J_{42} \sigma_k^+(1)
 - J_{14} \sum_{j=1}^{k-1} (-1)^j \sigma_j^-(0) \sigma_{k-j}^+(1), \\
\label{eq:ck-}
 c_k^- & := J_{13} (-1)^{k-1} \sigma_k^-(1) + J_{42} \sigma_k^+(0)
 + J_{32} \sum_{j=1}^{k-1} (-1)^j \sigma_j^-(1) \sigma_{k-j}^+(0).
\end{align}
\end{theorem}
\begin{proof}
According to Theorem~\ref{th:asymp} there exists a matrix solution to system~\eqref{eq:system} of the form~\eqref{eq:Y(x,l).def}, in which functions $u_1^\pm (x,\pm\l)$ and $u_2^\pm (x,\pm \l)$ admit expansions~\eqref{eq:u1} and~\eqref{eq:u2}, respectively. Combining conditions~\eqref{eq:a0b0}, \eqref{eq:bk.an.bn.0} on one hand and conditions~\eqref{eq:u2=u1.sigma}, \eqref{eq:sigma}, \eqref{eq:sigma_n(0,l)} on the other hand we obtain
\begin{equation} \label{eq:u1=1,u2=sigma=o(1)}
 u_1^\pm (0,\pm\l)=1, \qquad u_2^\pm (0,\pm \l) =
 \sigma^\pm (0,\pm \l) = o(1), \quad \l \to \infty.
\end{equation}
Inserting these expressions into~\eqref{eq:Y(x,l).def} we get
\begin{equation} \label{eq:Y(0,l)}
 Y(0,\l) = \begin{pmatrix}
 1 & \sigma^-(0,-\l) \\
 \sigma^+(0,\l) & 1
 \end{pmatrix} = I_2 + o_2(1), \quad \l \to \infty,
\end{equation}
where $o_2(1)$ is a $2 \times 2$ with $o(1)$ entries. Hence
\begin{equation} \label{eq:det(Y(l))}
 \omega(\l) := \det(Y(0,\l)) = 1 - \sigma^-(0,-\l) \sigma^+(0,\l)
 = 1+o(1), \quad \l \to \infty.
\end{equation}
Therefore, for some $R>0$ we have $\det(Y(0,\l)) \ne 0,\ |\l| > R$ and
hence $Y(x,\l)$ is a fundamental system of solutions to the
system~\eqref{eq:system.new} for all $|\l|>R$. Hence
$Y(x,\l)$ and $\Phi(x,\l)$ related by the identity
\begin{equation} \label{eq:Phi(x,l)=Y(x,l)[Y(0,l)]^-1}
 \Phi(x,\l) = Y(x,\l) [Y(0,\l)]^{-1}, \quad |\l| > R.
\end{equation}
Putting~\eqref{eq:Phi.def},~\eqref{eq:Y(x,l).def}
and~\eqref{eq:Y(0,l)} into~\eqref{eq:Phi(x,l)=Y(x,l)[Y(0,l)]^-1}
we get for $|\l|>R$
\begin{eqnarray}
 \label{eq:phi.11.sigma}
 \varphi_{11}(\l) & = & \frac{1}{\omega(\l)}
 \left(e^{b^+ \l} u_1^+(1,\l)
 - e^{b^- \l} u_2^-(1,-\l) \sigma^+(0,\l)\right), \\
 \label{eq:phi.12.sigma}
 \varphi_{12}(\l) & = & \frac{1}{\omega(\l)}
 \left(-e^{b^+ \l} u_1^+(1,\l) \sigma^-(0,-\l)
 + e^{b^- \l} u_2^-(1,-\l)\right), \\
 \label{eq:phi.21.sigma}
 \varphi_{21}(\l) & = & \frac{1}{\omega(\l)}
 \left(e^{b^+ \l} u_2^+(1,\l)
 - e^{b^- \l} u_1^-(1,-\l) \sigma^+(0,\l)\right), \\
 \label{eq:phi.22.sigma}
 \varphi_{22}(\l) & = & \frac{1}{\omega(\l)}
 \left(-e^{b^+ \l} u_2^+(1,\l) \sigma^-(0,-\l)
 + e^{b^- \l} u_1^-(1,-\l)\right).
\end{eqnarray}
Inserting~\eqref{eq:phi.11.sigma}--\eqref{eq:phi.22.sigma}
into~\eqref{eq:Delta} we get in view of~\eqref{eq:b+b-q+q-}
\begin{align} \label{eq:Delta.with.sigma}
 \Delta_Q(\l) & = J_{12} + J_{34} e^{i (b_1 + b_2) \l} \nonumber \\
 & + \frac{e^{i b_1 \l}}{\omega(\l)} \left(J_{32} u_1^+(1,\l)
 - J_{13} u_1^+(1,\l) \sigma^-(0,-\l) \right. \nonumber \\
 & \qquad \left. + J_{42} u_2^+(1,\l) - J_{14} u_2^+(1,\l)
 \sigma^-(0,-\l) \right) \nonumber \\
 & + \frac{e^{i b_2 \l}}{\omega(\l)} \left( - J_{32} u_2^-(1,-\l)
 \sigma^+(0,\l) + J_{13} u_2^-(1,-\l) \right. \nonumber \\
 & \qquad \left. - J_{42} u_1^-(1,-\l) \sigma^+(0,\l)
 + J_{14} u_1^-(1,-\l) \right).
\end{align}

First, consider the case $\Im \l \ge R$.
In view of~\eqref{eq:sigma_n(x,l)} we may assume that for these values of $\l$ the function $\sigma^+(1,\l)$ is correctly defined.
Using~\eqref{eq:u2=u1.sigma}, \eqref{eq:sigma}, \eqref{eq:sigma_n(0,l)}, \eqref{eq:P1n=1+o(1),...} and~\eqref{eq:sigma_n(x,l)} and taking into account notation~\eqref{eq:ck+} we rewrite the factor in parentheses of the third summand in~\eqref{eq:Delta.with.sigma} as follows
\begin{align} \label{eq:J32.u1...}
 & J_{32} u_1^+(1,\l) - J_{13} u_1^+(1,\l) \sigma^-(0,-\l)
 + J_{42} u_2^+(1,\l) - J_{14} u_2^+(1,\l) \sigma^-(0,-\l) \nonumber \\
 & \quad = u_1^+(1,\l) (J_{32} - J_{13} \sigma^-(0,-\l)
 + J_{42} \sigma^+(1,\l) - J_{14} \sigma^+(1,\l) \sigma^-(0,-\l)) \nonumber \\
 & \quad = (1 + o(1)) \cdot \Bigl(J_{32} - J_{13} \sum_{k=1}^n
 \frac{\sigma_k^-(0)}{(-c\l)^k}
 + J_{42} \Bigl(\sum_{k=1}^n \frac{\sigma_k^+(1)}{(c\l)^k}
 + \frac{\sigma_n^+(1,\l)}{(c\l)^n} \Bigr) \Bigr. \nonumber \\
 & \qquad \qquad \Bigl. - J_{14} \Bigl(\sum_{k=1}^n
 \frac{\sigma_k^-(0)}{(-c\l)^k}\Bigr)
 \Bigl(\sum_{k=1}^n \frac{\sigma_k^+(1)}{(c\l)^k}
 + \frac{\sigma_n^+(1,\l)}{(c\l)^n} \Bigr)\Bigr) \nonumber \\
 & \quad = (1 + o(1)) \cdot \Bigl(J_{32} + \sum_{k=1}^n (c\l)^{-k}
 \Bigl( - J_{13} (-1)^k \sigma_k^-(0) + J_{42} \sigma_k^+(1)
 \Bigr. \Bigr. \nonumber \\
 & \qquad \qquad \Bigl. \Bigl. - J_{14} \sum_{j=1}^{k-1} (-1)^j \sigma_j^-(0)
 \sigma_{k-j}^+(1) \Bigr) + o(\l^{-n})\Bigr) \nonumber \\
 & \quad = (1 + o(1)) \cdot \Bigl(J_{32} + \sum_{k=1}^n \frac{c_k^+}{(c\l)^k} + o(\l^{-n})\Bigr),
 \qquad \Im \l \to +\infty, \quad \Im \l \ge R.
\end{align}
Similar to~\eqref{eq:anmpxl} by combining~\eqref{eq:an(x,l).def} with inclusion $A_n^\pm \in X_\infty^0$, Lemma~\ref{lem:K.exp.X0} and inequality $b_2 - b_1 > 0$ we derive
\begin{equation} \label{eq:an-l.in.o()}
 |a_n^\pm (x,-\l)| = o(e^{(b_2 - b_1) \Im \l}),
 \qquad \Im \l \to +\infty, \quad \Im \l \ge R,
\end{equation}
uniformly in $x \in [0,1]$.
Combining~\eqref{eq:|bnl|<=sup|q|.sup|anl|} with~\eqref{eq:an-l.in.o()} we have
\begin{equation} \label{eq:|bn(1,l)|<=C/l e^cl}
 |b_n^-(1,-\l)| = o(e^{(b_2 - b_1) \Im \l}),
 \qquad \Im \l \to +\infty, \quad \Im \l \ge R.
\end{equation}
Now using~\eqref{eq:u1}, \eqref{eq:u2}, \eqref{eq:u1=1,u2=sigma=o(1)}, \eqref{eq:det(Y(l))}, \eqref{eq:an-l.in.o()} and~\eqref{eq:|bn(1,l)|<=C/l
e^cl} we can estimate the last summand
at~\eqref{eq:Delta.with.sigma}
\begin{align} \label{eq:eib2/omega...}
 & \abs{\frac{e^{i b_2 \l}}{\omega(\l)} \left(
 u_2^-(1,-\l) (J_{32} \sigma^+(0,\l) - J_{13}) +
 u_1^-(1,-\l) (J_{42} \sigma^+(0,\l) - J_{14}) \right)} \nonumber \\
 & \quad = e^{-b_2 \Im \l} \cdot (1+o(1)) \cdot \left|
 \left(o(1) + \frac{e^{(b_2 - b_1) \Im\l}}{\l^n} o(1)\right)
 \cdot (o(1) - J_{13}) \right. \nonumber \\
 & \quad \quad \left.+ \left(1 + o(1) + \frac{e^{(b_2 - b_1)
 \Im\l}}{\l^n} o(1)\right) \cdot (o(1)-J_{14})\right|
 = \frac{e^{-b_1 \Im \l}}{|\l|^n} o(1),
\end{align}
as $\Im \l \to +\infty$ and $\Im \l \ge R$. Since $b_1 < 0 < b_2$ then
\begin{equation} \label{eq:J12+J34e}
 \abs{J_{12} + J_{34} e^{i (b_1 + b_2) \l}}
 = \frac{e^{-b_1 \Im \l}}{|\l|^n} o(1),
 \qquad \Im \l \to +\infty, \quad \Im \l \ge R.
\end{equation}
Putting~\eqref{eq:J32.u1...},~\eqref{eq:eib2/omega...}
and~\eqref{eq:J12+J34e} into~\eqref{eq:Delta.with.sigma} we
arrive at~\eqref{eq:Delta.in.Omega+.ck}.

Now consider the case when $\Im \l \le -R$.
Relation~\eqref{eq:Delta.in.Omega-.ck} can be obtained similarly.
We reproduce only the step corresponding to the
equation~\eqref{eq:J32.u1...}. For such $\l$ the function $\sigma^-(1,-\l)$ is correctly defined and so we have
\begin{multline} \label{eq:-J32.u1-(1,-l)...}
 - J_{32} u_2^-(1,-\l) \sigma^+(0,\l) + J_{13} u_2^-(1,-\l) \\
 - J_{42} u_1^-(1,-\l) \sigma^+(0,\l) + J_{14} u_1^-(1,-\l) \\
 = u_1^-(1,-\l) (J_{14} + J_{13} \sigma^-(1,-\l)
 - J_{42} \sigma^+(0,\l) - J_{32} \sigma^-(1,-\l) \sigma^+(0,\l)) \\
 = (1+o(1)) \cdot \Bigl(J_{14} + J_{13} \Bigl(\sum_{k=1}^n
 \frac{\sigma_k^-(1)}{(-c\l)^k}
 + \frac{\sigma_n^-(1,-\l)}{(-c\l)^n} \Bigr)
 - J_{42} \sum_{k=1}^n \frac{\sigma_k^+(0)}{(c\l)^k} \Bigr.\\
 \Bigl. - J_{32} \Bigl(\sum_{k=1}^n \frac{\sigma_k^-(1)}{(-c\l)^k}
 + \frac{\sigma_n^-(1,-\l)}{(-c\l)^n} \Bigr)
 \Bigl(\sum_{k=1}^n \frac{\sigma_k^+(0)}{(c\l)^k}\Bigr) \Bigr) \\
 = (1+o(1)) \cdot \Bigl(J_{14} + \sum_{k=1}^n (c\l)^{-k} \Bigl(
 J_{13} (-1)^k \sigma_k^-(1) - J_{42} \sigma_k^+(0) \Bigr. \Bigr. \\
 \Bigl. \Bigl. - J_{32} \sum_{j=1}^{k-1} (-1)^j
 \sigma_j^-(1) \sigma_{k-j}^+(0) \Bigr) + o(\l^{-n})\Bigr) \\
 = (1 + o(1)) \cdot \Bigl(J_{14} - \sum_{k=1}^n \frac{c_k^-}{(c\l)^k}
 + o(\l^{-n})\Bigr)
\end{multline}
as $\Im \l \to -\infty$ and $\Im \l \le -R$.
The proof is now complete.
\end{proof}
\begin{remark} \label{rem:byparts}
Let us demonstrate alternative method of obtaining formulas~\eqref{eq:Delta.in.Omega+.ck}--\eqref{eq:ck-}
originally proposed in~\cite{Mal08} for Sturm-Liouville operators. It is
based on integral representation of the characteristic determinant obtained for $2\times 2$ Dirac type system
with integrable potential matrix in~\cite{LunMal16JMAA}
$($see~\eqref{eq:Delta=Delta0+_Intro}$)$:
\begin{align} \label{eq:Delta=Delta0+}
 \Delta_Q(\l) &= \Delta_0(\l) + \int^1_0 g_1 (t) e^{i b_1 \l t} dt
 + \int^1_0 g_2 (t) e^{i b_2 \l t} dt,
\end{align}
where $g_1, g_2 \in L^1[0,1]$ are expressed via the kernel of the transformation operator.

Assuming that $g_1, g_2 \in W_1^n[0,1]$ and integrating by parts, one derives
\begin{multline}
 \int^1_0 g_1 (t) e^{i b_1 \l t} dt
 = \left[ g_1 (t) \frac{e^{i b_1 \l t}}{i b_1 \l}\right]_0^1
 - \frac{1}{i b_1 \l} \int^1_0 g_1'(t) e^{i b_1 \l t} dt \\
 = \sum_{k=1}^{n} \frac{g_1^{(k-1)}(0) - g_1^{(k-1)}(1)
 e^{i b_1 \l}}{(-i b_1 \l)^k} + \frac{1}{(-i b_1 \l)^n}
 \int^1_0 g_1^{(n)}(t) e^{i b_1 \l t} dt.
\end{multline}
Inserting this formula into representation~\eqref{eq:Delta=Delta0+} yields
\begin{multline}
 \Delta_Q(\l) = \Delta_0(\l)
 + \sum_{k=1}^{n} \frac{g_1^{(k-1)}(0) - g_1^{(k-1)}(1) e^{i b_1 \l}}{(-i b_1 \l)^k} + \frac{1}{(-i b_1 \l)^n}
 \int^1_0 g_1^{(n)}(t) e^{ib_1 \l t} dt \\
 + \sum_{k=1}^{n} \frac{g_2^{(k-1)}(0) - g_2^{(k-1)}(1) e^{i b_2 \l}}{(-i b_2 \l)^k} + \frac{1}{(-i b_2 \l)^n} \int^1_0 g_2^{(n)}(t) e^{i b_2 \l t} dt,
 \quad \l \ne 0.
\end{multline}
In turn, combining this formula with~\eqref{eq:Delta0.intro} for $\Delta_0(\l)$, noting that $b_1 < 0 < b_2$ and applying the Riemann-Lebesgue Lemma implies desired asymptotic expansion~\eqref{eq:Delta+}--\eqref{eq:Delta-}. Comparing~\eqref{eq:Delta.in.Omega+.ck}--\eqref{eq:Delta.in.Omega-.ck} and~\eqref{eq:Delta+}--\eqref{eq:Delta-} we see that
\begin{equation}
 c_k^+ = - \left(\frac{b_2 - b_1}{b_1}\right)^k g_1^{(k-1)}(1), \qquad
 c_k^- = \left(\frac{b_2 - b_1}{b_2}\right)^k g_2^{(k-1)}(1),
\end{equation}
for $k \in \oneton$.
Note, however, that finding explicit form of the values $g_1^{(k-1)}(1)$ and $g_2^{(k-1)}(1)$
directly using transformation operators is rather difficult.
This is the reason why in this paper $($as well as in~\cite{LunMal13Dokl}$)$ we use approach above to compute coefficients $c_k^\pm$ by evolving Marchenko's method from~\cite{Mar77}.
\end{remark}
Let us compute values $c_k^\pm$ (defined in~\eqref{eq:ck+}--\eqref{eq:ck-}) for small values of $k$ explicitly via functions $q_\pm$. Recall that
\begin{equation} \label{eq:q+q-}
 q_+(x) = -i b_1 Q_{12}(x), \qquad q_-(x) = -i b_2 Q_{21}(x),
 \qquad x \in [0,1].
\end{equation}
\begin{lemma} \label{lem:c123}
\textbf{(i)} Let $Q_{12}, Q_{21} \in W_1^1[0,1]$. Then
\begin{align}
\label{eq:c1+}
 & c_1^+ = J_{13} q_+(0) + J_{42} q_-(1)
 = -i \bigr(J_{13} b_1 Q_{12}(0) + J_{42} b_2 Q_{21}(1)\bigl), \\
\label{eq:c1-}
 & c_1^- = J_{13} q_+(1) + J_{42} q_-(0)
 = -i \bigr(J_{13} b_1 Q_{12}(1) + J_{42} b_2 Q_{21}(0)\bigl).
\end{align}
\textbf{(ii)} Let $Q_{12}, Q_{21} \in W_1^2[0,1]$. Then
\begin{align}
\label{eq:c2+}
 & c_2^+ = J_{13} q_+'(0) - J_{42} q_-'(1) - J_{14} q_+(0) q_-(1), \\
\label{eq:c2-}
 & c_2^- = J_{13} q_+'(1) - J_{42} q_-'(0) + J_{32} q_+(1) q_-(0).
\end{align}
\textbf{(iii)} Let $Q_{12}, Q_{21} \in W_1^3[0,1]$. Then
\begin{multline} \label{eq:c3+}
 c_3^+ = J_{13} \bigl(q_+''(0) - q_+^2(0) q_-(0)\bigr)
 + J_{42} \bigl(q_-''(1) - q_-^2(1) q_+(1)\bigr) \\
 + J_{14} \bigl(q_+'(0) q_-(1) - q_+(0) q_-'(1)\bigr),
\end{multline}
\begin{multline} \label{eq:c3-}
 c_3^- = J_{13} \bigl(q_+''(1) - q_+^2(1) q_-(1)\bigr)
 + J_{42} \bigl(q_-''(0) - q_-^2(0) q_+(0)\bigr) \\
 - J_{32} \bigl(q_+'(1) q_-(0) - q_+(1) q_-'(0)\bigr).
\end{multline}
\end{lemma}
\begin{proof}
\textbf{(i)} Since $q_\pm \in W_1^1[0,1]$ then $\sigma_1^\pm = q_\mp$ (see~\eqref{eq:sigma1}). Hence it follows from~\eqref{eq:ck+}--\eqref{eq:ck-} that
\begin{align}
\label{eq:c1+.sigma}
 & c_1^+ = J_{13} \sigma_1^-(0) + J_{42} \sigma_1^+(1)
 = J_{13} q_+(0) + J_{42} q_-(1), \\
\label{eq:c1-.sigma}
 & c_1^- = J_{13} \sigma_1^-(1) + J_{42} \sigma_1^+(0)
 = J_{13} q_+(1) + J_{42} q_-(0).
\end{align}

\textbf{(ii)} Since $q_\pm \in W_1^2[0,1]$ then $\sigma_2^\pm = -q_\mp'$ (see~\eqref{eq:sigma2.a2}). Formulas~\eqref{eq:ck+}--\eqref{eq:ck-} for $k=2$ have the form
\begin{align} \label{eq:c2+.sigma}
 & c_2^+ = -J_{13} \sigma_2^-(0) + J_{42} \sigma_2^+(1)
 - J_{14} \sigma_1^-(0) \sigma_1^+(1), \\
\label{eq:c2-.sigma}
 & c_2^- = -J_{13} \sigma_2^-(1) + J_{42} \sigma_2^+(0)
 + J_{32} \sigma_1^-(1) \sigma_1^+(0).
\end{align}
Inserting formulas $\sigma_1^\pm = q_\mp$ and $\sigma_2^\pm = -q_\mp'$ into~\eqref{eq:c2+.sigma}--\eqref{eq:c2-.sigma} we arrive at~\eqref{eq:c2+}--\eqref{eq:c2-}.

\textbf{(iii)} Since $q_\pm \in W_1^3[0,1]$ then formulas~\eqref{eq:sigma1.a1}--\eqref{eq:sigma3.a3} imply that
\begin{equation} \label{eq:sigma123}
 \sigma_1^\pm = q_\mp, \qquad \sigma_2^\pm = -q_\mp', \qquad
 \sigma_3^\pm = q_\mp'' - q_\mp^2 q_\pm.
\end{equation}
Further, formulas~\eqref{eq:ck+}--\eqref{eq:ck-} for $k=3$ have the form
\begin{align}
\label{eq:c3+sigma}
 & c_3^+ = J_{13} \sigma_3^-(0) + J_{42} \sigma_3^+(1)
 + J_{14} \sigma_1^-(0) \sigma_2^+(1) - J_{14} \sigma_2^-(0) \sigma_1^+(1), \\
\label{eq:c3-sigma}
 & c_3^- = J_{13} \sigma_3^-(1) + J_{42} \sigma_3^+(0)
 - J_{32} \sigma_1^-(1) \sigma_2^+(0) + J_{14} \sigma_2^-(1) \sigma_1^+(0).
\end{align}
Inserting~\eqref{eq:sigma123} into~\eqref{eq:c3+sigma}--\eqref{eq:c3-sigma} we arrive at~\eqref{eq:c3+}--\eqref{eq:c3-}.
\end{proof}
\section{General completeness results for non-regular BVP}
\label{sec:compl.gen}
In this section we apply asymptotic expansion of the characteristic determinant from Theorem~\ref{th:Delta} to obtain a new general explicit completeness result (Theorem~\ref{th:compl.gen.2x2}) in the case of non-regular boundary conditions that substantially supplement~\cite[Theorem 5.1]{MalOri12} and~\cite[Proposition 4.5]{LunMal15JST}. We then apply Theorem~\ref{th:compl.gen.2x2} to obtain more explicit completeness results when some of the determinants $J_{jk}$ vanish. We conclude this section with an explicit criterion of completeness in the case of analytic potential $Q$ and degenerate boundary conditions of a special form.

The following statement is our first result on completeness.
\begin{theorem}[Theorem 3 in~\cite{LunMal13Dokl}] \label{th:compl.gen.2x2}
Let $Q_{12},Q_{21} \in W_1^n[0,1]$ for some $n \in \bN$,
and let the numbers $c_1^\pm, \ldots, c_n^\pm$
be defined via~\eqref{eq:ck+}--\eqref{eq:ck-}.
Assume that
\begin{equation} \label{eq:ck+ck-.ne0}
 |J_{32}| + |c_1^+| + \ldots + |c_n^+| \ne 0, \qquad
 |J_{14}| + |c_1^-| + \ldots + |c_n^-| \ne 0.
\end{equation}
Then the system of root vectors of the
BVP~\eqref{eq:system}--\eqref{eq:Udef} is complete and minimal in
$L^2([0,1];\bC^2)$.
\end{theorem}
\begin{proof}
Set $c_0^+ := J_{32}$, and $c_0^- := J_{14}$. Condition~\eqref{eq:ck+ck-.ne0} implies existence of $k^+, k^- \in \{0,1, \ldots, n\}$ such that
\begin{align}
\label{eq:ck+ne0}
 & c_{k^+}^+ \ne 0, \qquad c_k^+ = 0, \quad k \in \{0, \ldots, k^+ - 1\}, \\
 \label{eq:ck-ne0}
 & c_{k^-}^- \ne 0, \qquad c_k^- = 0, \quad k \in \{0, \ldots, k^- - 1\}.
\end{align}
Combining relations~\eqref{eq:ck+ne0}--\eqref{eq:ck-ne0} with Theorem~\ref{th:Delta} (see formulas~\eqref{eq:Delta.in.Omega+.ck}--\eqref{eq:ck-}) implies that
\begin{equation} \label{eq:Delta(it)>=k+-}
 |\Delta_Q(i t)| \ge \frac{C e^{-b_1 t}}{t^{k^+}}, \qquad
 |\Delta_Q(-i t)| \ge \frac{C e^{b_2 t}}{t^{k^-}}, \qquad t > R,
\end{equation}
for some $C, R > 0$. Hence due to Theorem~\ref{th:compl.gen} the system of root
vectors of the BVP~\eqref{eq:system}--\eqref{eq:Udef} is complete and minimal in
$L^2([0,1];\bC^2)$.
\end{proof}
\begin{remark}
Note that formulas~\eqref{eq:sigma1}--\eqref{eq:sigma(k+1)} and~\eqref{eq:ck+}--\eqref{eq:ck-} imply that coefficients $c^\pm _k$ are polynomials in $b_1$, $b_2$, $J_{32}$, $J_{14}$, $J_{42}$, $J_{13}$, $Q_{12}^{(j)}(0)$, $Q_{21}^{(j)}(0)$, $Q_{12}^{(j)}(1)$, $Q_{21}^{(j)}(1)$, $j \in \{0,1, \ldots, k-1\}$ $($see also Remark~\ref{rem:sigma123456} and Lemma~\ref{lem:c123}$)$.
\end{remark}
Combining Theorem~\ref{th:compl.gen.2x2} with Lemma~\ref{lem:c123} leads to a series of more explicit completeness results.
\begin{corollary} \label{cor:c01}
Let $Q_{12}, Q_{21} \in W_1^1[0,1]$. Then the system of root vectors of the BVP~\eqref{eq:system}--\eqref{eq:Udef} is complete and minimal in $L^2([0,1];\bC^2)$ whenever both of the following conditions hold:
\begin{align}
\label{eq:c01+}
 |J_{32}| + |J_{13} b_1 Q_{12}(0) + J_{42} b_2 Q_{21}(1)| \ne 0, \\
\label{eq:c01-}
 |J_{14}| + |J_{13} b_1 Q_{12}(1) + J_{42} b_2 Q_{21}(0)| \ne 0.
\end{align}
\end{corollary}
\begin{proof}
Combining formulas~\eqref{eq:c1+}--\eqref{eq:c1-} from Lemma~\ref{lem:c123} with conditions~\eqref{eq:c01+}--\eqref{eq:c01-} implies that
\begin{equation}
 |J_{32}| + |c_1^+| \ne 0, \qquad
 |J_{14}| + |c_1^-| \ne 0.
\end{equation}
Therefore, Theorem~\ref{th:compl.gen.2x2} implies the desired completeness property.
\end{proof}
Recall that the functions $q_\pm$ are given by~\eqref{eq:q+q-}.
\begin{corollary} \label{cor:c012}
Let $Q_{12}, Q_{21} \in W_1^2[0,1]$. Then the system of root vectors of the BVP~\eqref{eq:system}--\eqref{eq:Udef} is complete and minimal in $L^2([0,1];\bC^2)$ whenever both of the following conditions hold:
\begin{align}
\label{eq:c012+}
 |J_{32}| + |J_{13} q_+(0) + J_{42} q_-(1)|
 + |J_{13} q_+'(0) - J_{42} q_-'(1) - J_{14} q_+(0) q_-(1)| \ne 0, \\
\label{eq:c012-}
 |J_{14}| + |J_{13} q_+(1) + J_{42} q_-(0)|
 + |J_{13} q_+'(1) - J_{42} q_-'(0)
 + J_{32} q_+(1) q_-(0)| \ne 0.
\end{align}
\end{corollary}
\begin{proof}
Combining formulas~\eqref{eq:c1+}--\eqref{eq:c2-} from
with
conditions~\eqref{eq:c012+}--\eqref{eq:c012-} implies that
\begin{equation}
 |J_{32}| + |c_1^+| + |c_2^+| \ne 0, \qquad
 |J_{14}| + |c_1^-| + |c_2^-| \ne 0.
\end{equation}
As in the previous result, Theorem~\ref{th:compl.gen.2x2} completes the proof.
\end{proof}
Inserting notations~\eqref{eq:q+q-} into formulas~\eqref{eq:c012+}--\eqref{eq:c012-} one readily transforms the formulation of Corollary~\ref{cor:c012} into formulation of Proposition~\ref{prop:n=2} from the introduction.
\begin{corollary} \label{cor:c0123}
Let $Q_{12}, Q_{21} \in W_1^3[0,1]$. Assume that either condition~\eqref{eq:c012+} holds or
\begin{multline} \label{eq:c3+ne0}
 J_{13} \bigl(q_+''(0) - q_+^2(0) q_-(0)\bigr)
 + J_{42} \bigl(q_-''(1) - q_-^2(1) q_+(1)\bigr) \\
 + J_{14} \bigl(q_+'(0) q_-(1) - q_+(0) q_-'(1)\bigr) \ne 0.
\end{multline}
Further, assume that either condition~\eqref{eq:c012-} holds or
\begin{multline} \label{eq:c3-ne0}
 J_{13} \bigl(q_+''(1) - q_+^2(1) q_-(1)\bigr)
 + J_{42} \bigl(q_-''(0) - q_-^2(0) q_+(0)\bigr) \\
 - J_{32} \bigl(q_+'(1) q_-(0) - q_+(1) q_-'(0)\bigr) \ne 0.
\end{multline}
Then the system of root vectors of the BVP~\eqref{eq:system}--\eqref{eq:Udef} is complete and minimal in $L^2([0,1];\bC^2)$.
\end{corollary}
\begin{proof}
Combining formulas~\eqref{eq:c1+}--\eqref{eq:c3-}
with conditions~\eqref{eq:c3+ne0}--\eqref{eq:c3-ne0} implies that
\begin{equation}
 |J_{32}| + |c_1^+| + |c_2^+| + |c_3^+| \ne 0, \qquad
 |J_{14}| + |c_1^-| + |c_2^-| + |c_3^-| \ne 0,
\end{equation}
which coincides with condition~\eqref{eq:ck+ck-.ne0} from Theorem~\ref{th:compl.gen.2x2}.
\end{proof}
Next we compare Corollaries~\ref{cor:c01}, \ref{cor:c012} and~\ref{cor:c0123} with a very interesting recent completeness result from~\cite{Mak23} by A.S.\;Makin, which reads as follows.
\begin{theorem}[Theorem 1 in~\cite{Mak23}] \label{th:Makin}
Let $B = \diag(-1, 1)$ and $Q_{12}, Q_{21} \in L^1[0,1]$.
Assume that either $J_{32} \ne 0$ or the following conditions hold with some $\rho_1, \rho_2 > 0$:
\begin{align}
\label{eq:rho1.rho2}
 & \lim_{h \to 0} \frac{1}{h^{\rho_1}} \int_0^h Q_{12}(x) \,dx
 = \wt{Q}_{12}^0 \ne 0, \qquad
 \lim_{h \to 0} \frac{1}{h^{\rho_2}} \int_{1-h}^1 Q_{21}(x) \,dx
 = \wt{Q}_{21}^1 \ne 0, \\
\label{eq:Q120-Q211}
 & \qquad \qquad (|J_{13}|+|J_{42}|)|\rho_1 - \rho_2|
 + |J_{13} \wt{Q}_{12}^0 - J_{42} \wt{Q}_{21}^1| \ne 0.
\end{align}
Further, assume either $J_{14} \ne 0$ or the following conditions hold with some $\rho_3, \rho_4 > 0$:
\begin{align}
\label{eq:rho3.rho4}
 & \lim_{h \to 0} \frac{1}{h^{\rho_3}} \int_{1-h}^1 Q_{12}(x) \,dx
 = \wt{Q}_{12}^1 \ne 0, \qquad
 \lim_{h \to 0} \frac{1}{h^{\rho_4}} \int_0^h Q_{21}(x) \,dx
 = \wt{Q}_{21}^0 \ne 0, \\
\label{eq:Q121-Q210}
 & \qquad \qquad (|J_{13}|+|J_{42}|)|\rho_3 - \rho_4|
 + |J_{13} \wt{Q}_{12}^1 - J_{42} \wt{Q}_{21}^0| \ne 0.
\end{align}
Then the system of root functions of the operator $L_U(Q)$ is complete in $L^2([0,1]; \bC^2)$.
\end{theorem}
Restricting ourselves to the case $J_{32}=0$ and $J_{14} \ne 0$, we note, that when $Q(\cdot)$ is smooth and several first derivatives $Q_{12}^{(j)}(0)$ and $Q_{21}^{(j)}(1)$ vanish simultaneously,
then Theorem~\ref{th:Makin} can be reformulated as follows
(see also~\cite[Remark 2]{Mak23}).
\begin{corollary}[\cite{Mak23}] \label{cor:Makin}
Let $-b_1 = b_2 = 1$ and let $Q_{12}, Q_{21} \in W_1^n[0,1]$ for some $n \in \bZ_+$. Assume that $J_{32} = 0$, $J_{14} \ne 0$, and let the following conditions hold,
\begin{equation} \label{eq:Q12.Q21.k}
 Q_{12}^{(j)}(0) = Q_{21}^{(j)}(1) = 0, \quad j \in \{0,1,\ldots,n-2\},
 \qquad Q_{12}^{(n-1)}(0) Q_{21}^{(n-1)}(1) \ne 0.
\end{equation}
Then the system of root functions of the operator $L_U(Q)$ is complete in $L^2([0,1]; \bC^2)$ whenever
\begin{equation} \label{eq:Pk.Makin}
 J_{13} Q_{12}^{(n-1)}(0) \ne J_{42} (-1)^{n-1} Q_{21}^{(n-1)}(1).
\end{equation}
\end{corollary}
\begin{proof}
Since $J_{32} = 0$ and $J_{14} \ne 0$, then Theorem~\ref{th:Makin} applies if and only if condition~\eqref{eq:Q120-Q211} holds. Conditions~\eqref{eq:Q12.Q21.k} imply
\begin{align}
\label{eq:lim.int.Q120}
 & \lim_{h \to 0} \frac{1}{h^{n}} \int_0^h Q_{12}(x) \,dx
 = \frac{Q_{12}^{(n-1)}(0)}{n!} \ne 0, \\
 & \lim_{h \to 0} \frac{1}{h^{n}} \int_{1-h}^1 Q_{21}(x) \,dx
 = \frac{(-1)^{n-1} Q_{21}^{(n-1)}(1)}{n!} \ne 0.
\end{align}
Hence relations~\eqref{eq:rho1.rho2} hold with
\begin{equation}
 \rho_1 = \rho_2 = n, \qquad \wt{Q}_{12}^0 = \frac{Q_{12}^{(n-1)}(0)}{n!},
 \qquad \wt{Q}_{21}^1 = \frac{(-1)^{n-1} Q_{21}^{(n-1)}(1)}{n!}.
\end{equation}
In turn, condition~\eqref{eq:Q120-Q211} turns into
\begin{multline}
 (|J_{13}|+|J_{42}|)|\rho_1 - \rho_2|
 + \bigabs{J_{13} \wt{Q}_{12}^0 - J_{42} \wt{Q}_{21}^1} \\
 = \frac{\bigabs{J_{13} Q_{12}^{(n-1)}(0)
 - J_{42} (-1)^{n-1} Q_{21}^{(n-1)}(1)}}{n!} \ne 0,
\end{multline}
which is equivalent to~\eqref{eq:Pk.Makin}.
\end{proof}
\begin{remark} \label{rem:not.covered}
Theorem~\ref{th:compl.gen.2x2} was announced in 2013 in~\cite{LunMal13Dokl}.
After that, there appeared several papers~\cite{LunMal15JST,KosShk21,Mak23} concerned with the completeness property for BVP~\eqref{eq:system}--\eqref{eq:Udef} subject to non-regular boundary conditions.
Let us compare these recent results with Corollaries~\ref{cor:c01} and~\ref{cor:c012} in the classical Dirac case.
First, note that Corollary~\ref{cor:c01} coincides with~\cite[Proposition 4.5]{LunMal15JST}, but it requires more restrictive smoothness assumption
$Q \in W_1^1$ compared to continuity of $Q(\cdot)$ at the endpoints 0 and 1 imposed in~\cite{LunMal15JST}.

Further, let $-b_1 = b_2 = 1$, $Q_{12}, Q_{21} \in W_1^2[0,1]$, and let the following conditions hold,
\begin{equation} \label{eq:J32=0.Q12.Q21}
 J_{32}=0, \quad J_{14} \ne 0,
 \qquad Q_{12}(0) Q_{21}(1) \ne 0,
 \quad J_{13} Q_{12}(0) = J_{42} Q_{21}(1).
\end{equation}
Then condition~\eqref{eq:c01-} holds, but condition~\eqref{eq:c01+} fails.
Therefore,~\cite[Proposition 4.5]{LunMal15JST} is not applicable.
Moreover, \cite[Theorem~1]{Mak23} $($see~Theorem~\ref{th:Makin} above for its formulation$)$ also cannot be applied.
Indeed, under assumptions~\eqref{eq:J32=0.Q12.Q21}, Theorem~\ref{th:Makin} is equivalent to Corollary~\ref{cor:Makin} with $n=1$, but the last condition in~\eqref{eq:J32=0.Q12.Q21} ensures that condition~\eqref{eq:Pk.Makin} does not hold.
On the other hand, if
\begin{equation}
 J_{13} Q_{12}'(0) + J_{42} Q_{21}'(1) + i J_{14} Q_{12}(0) Q_{21}(1) \ne 0,
\end{equation}
then
condition~\eqref{eq:c012+} holds, and Corollary~\ref{cor:c012}
ensures the desired completeness property.
\end{remark}
\section{Properties of the functions \texorpdfstring{$\sigma_j^\pm(x)$}{sigmajx}}
\label{sec:sigmajx}
Explicit expressions for values $c^\pm _k$ (given by~\eqref{eq:ck+}--\eqref{eq:ck-}) for large $k$ are somewhat cumbersome in general case. However, if either some of the values $Q_{12}^{(j)}(0)$, $Q_{12}^{(j)}(1)$, $Q_{21}^{(j)}(0)$, $Q_{21}^{(j)}(1)$ are zero, or two of the determinants $J_{14}, J_{32}, J_{13}, J_{42}$ are zero, then conditions of Theorem~\ref{th:compl.gen.2x2} can be expressed explicitly via values $Q_{12}^{(j)}(0)$, $Q_{12}^{(j)}(1)$, $Q_{21}^{(j)}(0)$, $Q_{21}^{(j)}(1)$ in a very simple fashion.
To establish these explicit completeness results, we need to ``decipher'' formulas~\eqref{eq:sigma1}--\eqref{eq:sigma(k+1)} for the functions $\sigma_j^\pm(\cdot)$. Let us introduce some notations,
\begin{align}
 & \bZ_+^k := \{\alp = (\alp_1, \ldots, \alp_k):
 \alp_1, \ldots, \alp_k \in \bZ_+\}, \quad \bZ_+ = \{0, 1, 2, 3, \ldots\}, \\
 & \Sigma_\alp := \alp_1 + \ldots + \alp_k, \qquad m_{\alp} := \max\{\alp_1, \ldots, \alp_k\}, \\
 & q^{(\alp)} := q^{(\alp_1)} \cdot \ldots \cdot q^{(\alp_k)},
 \qquad q \in W_1^{m_{\alp}}[0,1],
\end{align}
for $\alp = (\alp_1, \ldots, \alp_k) \in \bZ_+^k$.
\begin{lemma} \label{lem:sigma.q}
Let $Q_{12}, Q_{21} \in W_1^n[0,1]$ and functions $\sigma^{\pm}_k(x), k \in
 \oneton$, be defined via formulas~\eqref{eq:b+b-q+q-}, \eqref{eq:sigma1},
\eqref{eq:sigma(k+1)}. Then the following identities hold,
\begin{equation} \label{eq:sigmak=qk+sum.q}
 \sigma^{\pm}_j = (-1)^{j-1} q_{\mp}^{(j-1)}
 + \sum_{\atop{k \in \bN, \ 2k+1 \le j,}
 {\atop{\alp \in \bZ_+^{k+1},\ \beta \in \bZ_+^k,}
 {\Sigma_\alp + \Sigma_\beta = j - 1 - 2k}}}
 C_{j,k,\alpha,\beta} \cdot q_{\mp}^{(\alp)} \cdot q_{\pm}^{(\beta)},
 \qquad j \in \oneton,
\end{equation}
with some $C_{j,k, \alpha, \beta}$ that depend only on sequences $\alpha = (\alpha_1, \ldots, \alpha_{k+1})$, $\beta = (\beta_1, \ldots, \beta_k)$ and indexes $j, k$ $($in particular, these coefficients are the same in formulas for $\sigma_j^+$ and $\sigma_j^-)$. Moreover,
\begin{equation} \label{eq:Cj1ab}
 C_{j, 1, \alpha, \beta} = (-1)^j, \qquad j \ge 3,
 \quad \alpha=(0,0) \in \bZ_+^2, \quad \beta=(j-3) \in \bZ_+^1,
\end{equation}
i.e.\ the coefficient at $q_{\mp} \cdot q_{\mp} \cdot q_{\pm}^{(j-3)}$ in the sum in the r.h.s.\ of~\eqref{eq:sigmak=qk+sum.q} equals to $(-1)^{j}$.

\end{lemma}
\begin{proof}
Identity~\eqref{eq:sigmak=qk+sum.q} for $j=1$ coincides with~\eqref{eq:sigma1}, while for $j=2$ it is proved in~\eqref{eq:sigma2.a2}. Let $j \ge 3$ and assume that identity~\eqref{eq:sigmak=qk+sum.q} is valid for $\sigma_1^\pm$, $\ldots$, $\sigma_{j-1}^\pm$. To simplify the exposition of the proof let us introduce some notations. First, we set
\begin{equation} \label{eq:gamh.beta0.alp0.def}
 \gam_h := (h) \in \bZ_+^1, \quad h \in \bZ_+, \qquad
 \beta_{\varnothing} := () \in \bZ_+^0, \qquad
 \alp_0 := (0,0) \in \bZ_+^2,
\end{equation}
i.e.\ $\beta_{\varnothing}$ is an empty sequence. Next, we note that
\begin{equation}
 (-1)^{h-1} q_\pm^{(h-1)} = C_{h,0,\alpha,\beta} \cdot q_{\mp}^{(\alp)} \cdot q_{\pm}^{(\beta)}, \qquad \alp = \gam_{h-1} \in \bZ_+^1, \quad \beta = \beta_{\varnothing} \in \bZ_+^0,
\end{equation}
for $h \in \oneton$. Here $q_{\pm}^{(\beta)} = q_{\pm}^{(\beta_\varnothing)} := 1$ and $C_{h,0,\alpha,\beta} = (-1)^{h-1}$. Let us also set
\begin{align}
\label{eq:Psij.def}
 \Psi_h & := \{(k,\alp,\beta): k \in \bN, \ 2k+1 \le h,
 \ \alp \in \bZ_+^{k+1},\ \beta \in \bZ_+^k,
 \ \Sigma_\alp + \Sigma_\beta = h - 1 - 2k\}, \\
\nonumber
 \Psi_h^0 & := \{(0, \gam_{h-1}, \beta_\varnothing)\} \cup \Psi_h \\
\label{eq:Psij0.def}
 & = \{(k,\alp,\beta): k \in \bZ_+, \ 2k+1 \le h,
 \ \alp \in \bZ_+^{k+1},\ \beta \in \bZ_+^k,
 \ \Sigma_\alp + \Sigma_\beta = h - 1 - 2k\},
\end{align}
for $h \in \bN$. With account of these notations, induction hypothesis reads as follows:
\begin{equation} \label{eq:sigmah}
 \sigma^{\pm}_h = (-1)^{h-1} q_\pm^{(h-1)}
 + \sum_{(k,\alp,\beta) \in \Psi_h}
 C_{h,k,\alpha,\beta} \cdot q_{\mp}^{(\alp)} \cdot q_{\pm}^{(\beta)}
 = \sum_{(k,\alp,\beta) \in \Psi_h^0}
 C_{h,k,\alpha,\beta} \cdot q_{\mp}^{(\alp)} \cdot q_{\pm}^{(\beta)},
\end{equation}
for $h \in \oneto{j-1}$
where
\begin{align}
 & C_{h,0,\alp,\beta} = (-1)^{h-1}, & \alp = \gam_{h-1} \in \bZ_+^1,
 \quad \beta = \beta_\varnothing \in \bZ_+^0, \quad 1 \le h < j, \\
 & C_{h,1,\alp,\beta} = (-1)^h, & \alp = \alp_0 \in \bZ_+^2,
 \quad \beta = \gam_{h-3} \in \bZ_+^1, \quad 3 \le h < j.
\end{align}
Combining induction hypothesis~\eqref{eq:sigmah} with formula~\eqref{eq:sigma(k+1)} yields,
\begin{multline} \label{eq:triple.sum}
 \sigma_j^\pm = -(\sigma_{j-1}^\pm)'
 - q_\pm \cdot \sum_{h=1}^{j-2} \sigma_h^\pm \sigma_{j-1-h}^\pm \\
 = (-1)^{j-1} q_\mp^{(j-1)}
 - \sum_{(k,\alp,\beta) \in \Psi_{j-1}}
 C_{j-1,k,\alpha,\beta}
 \cdot \left(q_{\mp}^{(\alp)} \cdot q_{\pm}^{(\beta)}\right)' \\
 - \sum_{h=1}^{j-2} \ \sum_{(k_1,\wt{\alp}_1,\wt{\beta}_1) \in \Psi_h^0}
 \ \sum_{(k_2,\wt{\alp}_2,\wt{\beta}_2) \in \Psi_{j-1-h}^0}
 C_{h,k_1,\wt{\alp}_1,\wt{\beta}_1} C_{j-1-h,k_2,\wt{\alp}_2,\wt{\beta}_2} \times \\
 \times q_{\mp}^{(\wt{\alp}_1)} q_{\mp}^{(\wt{\alp}_2)} \times
 q_{\pm}^{(\wt{\beta}_1)} q_{\pm}^{(\wt{\beta}_2)} q_\pm.
\end{multline}
For $u \in (u_1, \ldots, u_h) \in \bZ_+^h$ and $u \in (v_1, \ldots, v_k) \in \bZ_+^k$ we denote
\begin{equation}
 u^\frown v := (u_1, \ldots, u_h, v_1, \ldots, v_k),
\end{equation}
the concatenation of the sequences $u$ and $v$. With account of this notation, for each summand in the triple sum above we have
\begin{equation}
 q_\mp^{(\wt{\alp}_1)} q_\mp^{(\wt{\alp}_2)} \times
 q_\pm^{(\wt{\beta}_1)} q_\pm^{(\wt{\beta}_2)} q_\pm
 = q_\mp^{(\alp)} \cdot q_\pm^{(\beta)},
\end{equation}
where
\begin{align}
 \alp = \wt{\alp}_1 ^\frown \wt{\alp}_2 \in \bZ_+^{k_1 + k_2 + 2}, \quad
 \beta = \wt{\beta}_1 ^\frown \wt{\beta}_2 ^\frown \gam_0 \in \bZ_+^{k_1 + k_2 + 1},
 \quad k = k_1 + k_2 + 1, \\
 \Sigma_\alp + \Sigma_\beta = \Sigma_{\wt{\alp}_1} + \Sigma_{\wt{\beta}_1} + \Sigma_{\wt{\alp}_1} + \Sigma_{\wt{\beta}_1} + 0 = j - 3 - 2 k_1 - 2 k_2 = j - 1 - 2k.
\end{align}
It follows from inequalities $2 k_1 + 1 \le h$ and $2 k_2 + 1 \le j-1-h$ that $2k+1 \le j$. Therefore, the triple sum in~\eqref{eq:triple.sum} has the desired form
\begin{equation} \label{eq:sumkj,alp,beta}
 \sum_{(k,\alp,\beta) \in \Psi_j}
 \wt{C}_{j,k,\alpha,\beta} \cdot q_{\mp}^{(\alp)} \cdot q_{\pm}^{(\beta)},
\end{equation}
with certain coefficients $\wt{C}_{j,k,\alpha,\beta}$.

Further we note that for $\alp = (\alp_1, \ldots, \alp_{k+1}) \in \bZ_+^{k+1}$ and $\beta = (\beta_1, \ldots, \beta_k) \in \bZ_+^k$ we have
\begin{equation} \label{eq:qaqb'}
 \left(q_{\mp}^{(\alp)} \cdot q_{\pm}^{(\beta)}\right)'
 = \left(q_{\mp}^{(\alp)}\right)' \cdot q_{\pm}^{(\beta)}
 + q_{\mp}^{(\alp)} \cdot \left(q_{\pm}^{(\beta)}\right)'
 = \sum_{h=1}^{k+1} q_{\mp}^{(\wt{\alp}_h)} \cdot q_{\pm}^{(\beta)}
 + \sum_{h=1}^k q_{\mp}^{(\alp)} \cdot q_{\pm}^{(\wt{\beta}_h)},
\end{equation}
where
\begin{align}
 & \wt{\alp}_h = (\alp_1, \ldots, \alp_{h-1}, \alp_h + 1,
 \alp_{h+1}, \ldots, \alp_{k+1}), \qquad h \in \oneto{k+1}, \\
 & \wt{\beta}_h = (\beta_1, \ldots, \beta_{h-1}, \beta_h + 1,
 \beta_{h+1}, \ldots, \beta_k), \qquad h \in \oneto{k}.
\end{align}
Evidently,
\begin{align}
 & \Sigma_{\wt{\alp}_h} + \Sigma_\beta = \Sigma_\alp + \Sigma_\beta + 1,
 \qquad h \in \oneto{k+1}, \\
 & \Sigma_\alp + \Sigma_{\wt{\beta}_h} = \Sigma_\alp + \Sigma_\beta + 1,
 \qquad h \in \oneto{k}.
\end{align}
Therefore for the second summand in~\eqref{eq:triple.sum} we have
\begin{equation} \label{eq:sum.qaqb'}
 \sum_{(k,\alp,\beta) \in \Psi_{j-1}}
 C_{j-1,k,\alpha,\beta}
 \cdot \left(q_{\mp}^{(\alp)} \cdot q_{\pm}^{(\beta)}\right)'
 =
 \sum_{(k,\alp,\beta) \in \Psi_j}
 \wh{C}_{j,k,\alpha,\beta} \cdot q_{\mp}^{(\alp)} \cdot q_{\pm}^{(\beta)},
\end{equation}
with certain coefficients $\wh{C}_{j,k,\alpha,\beta}$.
Inserting~\eqref{eq:sum.qaqb'} and expression~\eqref{eq:sumkj,alp,beta} for the triple sum from~\eqref{eq:triple.sum} into~\eqref{eq:triple.sum} we arrive at the desired formula~\eqref{eq:sigmak=qk+sum.q}.

Finally, we prove relation~\eqref{eq:Cj1ab} via induction. For $j=3$ it is evident from~\eqref{eq:sigma3.a3}. Let $j \ge 4$ and assume that relation~\eqref{eq:Cj1ab} is valid for $j-1$. With account of notations~\eqref{eq:gamh.beta0.alp0.def} it takes the form
\begin{equation} \label{eq:C.j-1.1ab}
 C_{j-1, 1, \alpha_0, \gam_{j-4}} = (-1)^{j-1}.
\end{equation}
Let us prove that $C_{j, 1, \alpha_0, \gam_{j-3}} = (-1)^j$.
Let us analyze how monomial $q_{\mp} \cdot q_{\mp} \cdot q_\pm^{(j-3)}$ can appear in the expression~\eqref{eq:triple.sum} for $\sigma_j^\pm$. It is evident that each monomial in the triple sum from~\eqref{eq:triple.sum} has a factor $q_\pm$ not present in $q_{\mp} \cdot q_{\mp} \cdot q_\pm^{(j-3)}$. Hence the monomial $q_{\mp} \cdot q_{\mp} \cdot q_\pm^{(j-3)}$ can only appear in the second sum. Since this monomial has only one derivative factor $q_\pm^{(j-3)}$ it follows from~\eqref{eq:qaqb'} that it is only contained in this derivative term
\begin{multline}
 C_{j-1,1,\alpha_0,\gam_{j-4}}
 \cdot \left(q_{\mp} \cdot q_{\mp} \cdot q_{\pm}^{(j-4)}\right)' \\
 = 2 C_{j-1,1,\alpha_0,\gam_{j-4}} \cdot
 q_{\mp}' \cdot q_{\mp} \cdot q_{\pm}^{(j-4)}
 + C_{j-1,1,\alpha_0,\gam_{j-4}} \cdot
 q_{\mp} \cdot q_{\mp} \cdot q_{\pm}^{(j-3)}.
\end{multline}
Taking into account this observaion and induction hypothesis~\eqref{eq:C.j-1.1ab}, and comparing~\eqref{eq:sigmak=qk+sum.q} and~\eqref{eq:triple.sum} yields
\begin{equation}
 C_{j, 1, \alpha_0, \gam_{j-3}} = -C_{j-1, 1, \alpha_0, \gam_{j-4}}
 = -(-1)^{j-1} = (-1)^j.
\end{equation}
The proof is now complete.
\end{proof}
We need the following simple corollaries of formula~\eqref{eq:sigmak=qk+sum.q}.
\begin{lemma} \label{lem:sigma.q=0}
Under the assumptions of Lemma~\ref{lem:sigma.q}, let
\begin{equation} \label{eq:m0.def}
 a \in [0,1], \qquad m \in \{0,1,\ldots,n-1\}, \qquad m_0 := \min\{2m+2, n\}.
\end{equation}
Then the following implications hold
\begin{align}
\label{eq:q-j=0}
 & q_-^{(j)}(a) = 0, \ \ 0 \le j < m \quad \Longrightarrow \quad
 \sigma^+_{j}(a) = (-1)^{j-1} q_-^{(j-1)}(a), \ \ 1 \le j \le m_0, \\
\label{eq:q+j=0}
 & q_+^{(j)}(a) = 0, \ \ 0 \le j < m \quad \Longrightarrow \quad
 \sigma^-_{j}(a) = (-1)^{j-1} q_+^{(j-1)}(a), \ \ 1 \le j \le m_0.
\end{align}
\end{lemma}
\begin{proof}
Let us prove~\eqref{eq:q-j=0} (implication~\eqref{eq:q+j=0} can be proved similarly). Assume that $q_-^{(j)}(a) = 0, \ \ 0 \le j < m$. In particular, if $m = 0$, then no assumptions on $q_-$ are needed.
Let $j \in \oneto{m_0}$. Since $j \le n$, then Lemma~\ref{lem:sigma.q} ensures formula~\eqref{eq:sigmak=qk+sum.q} for $\sigma_j^+(\cdot)$. Consider a factor $q_-^{(\alp)}(a)$ of any individual summand in the sum in~\eqref{eq:sigmak=qk+sum.q} for $\sigma_j^+(a)$, where
\begin{equation} \label{eq:alp.sum}
 \alp=(\alp_1, \ldots, \alp_{k+1}) \in \bZ_+^{k+1}, \qquad
 \alp_1 + \ldots + \alp_{k+1} \le j-1-2k,
\end{equation}
and $k \in \bN$, $2k+1 \le j$. Since $j \le m_0 \le 2m+2$, it follows from the second relation in~\eqref{eq:alp.sum} that
\begin{equation}
 \min\{\alp_1, \ldots, \alp_{k+1}\} \le \frac{j-1-2k}{k+1} \le \frac{j-3}{2}
 \le \frac{2m-1}{2} < m.
\end{equation}
Hence, it follows from our assumption that $q_-^{(\alp_h)}(a) = 0$ for some $h \in \oneto{k+1}$. Therefore, $q_-^{(\alp)}(a) = 0$.
Thus, it follows from~\eqref{eq:sigmak=qk+sum.q} that
\begin{equation} \label{eq:sigmaj=qj}
 \sigma^+_j(a) = (-1)^{j-1} q_-^{(j-1)}(a),
\end{equation}
which finishes the proof.
\end{proof}
\begin{corollary} \label{cor:q0.sigma0}
Let $Q_{12}, Q_{21} \in W_1^n[0,1]$,
$a \in [0,1]$, and $m \in \{0,1,\ldots,n-1\}$.

\textbf{(i)} Let the following conditions hold,
\begin{equation} \label{eq:q-ja}
 q_-^{(j)}(a) = 0, \quad j \in \{0, \ldots, m-1\},
 \qquad q_-^{(m)}(a) \ne 0.
\end{equation}
Then
\begin{equation} \label{eq:sigma+ja}
 \sigma^+_{j}(a) = 0, \quad j \in \oneto{m}, \qquad
 \sigma^+_{m + 1}(a) \ne 0.
\end{equation}

\textbf{(ii)} Let the following conditions hold,
\begin{equation} \label{eq:q+ja}
 q_+^{(j)}(a) = 0, \quad j \in \{0, \ldots, m-1\},
 \qquad q_+^{(m)}(a) \ne 0.
\end{equation}
Then
\begin{equation} \label{eq:sigma-ja}
 \sigma^-_{j}(a) = 0, \quad j \in \oneto{m}, \qquad
 \sigma^-_{m + 1}(a) \ne 0.
\end{equation}
\end{corollary}
\begin{proof}
\textbf{(i)} Since $m < n$, it is evident that $m_0 \ge m+1$ where $m_0$ is given by~\eqref{eq:m0.def}. Combining this observation with condition~\eqref{eq:q-ja} and implication~\eqref{eq:q-j=0} from Lemma~\ref{lem:sigma.q=0} yields that
\begin{equation}
 \sigma^+_{j}(a) = (-1)^{j-1} q_-^{(j-1)}(a), \qquad j \in \oneto{m+1}.
\end{equation}
Combining these relations with condition~\eqref{eq:q-ja} trivially implies~\eqref{eq:sigma+ja}.

\textbf{(ii)} The proof is similar.
\end{proof}
\section{Explicit refined completeness results for non-regular BVP}
\label{sec:compl.refined}
In this final section, under additional algebraic assumptions either on the potential matrix $Q(\cdot)$ or on the boundary conditions, we establish a series of more explicit refined completeness results that clarify somewhat cumbersome conditions of Theorem~\ref{th:compl.gen.2x2}.

The following function is highly involved in most of the results of this section,
\begin{equation} \label{eq:Pdef}
 P(x) := J_{13} b_1 Q_{12}(x) + J_{42} b_2 Q_{21}(1-x), \qquad x \in [0,1].
\end{equation}
The next result covers the general case of non-regular boundary conditions,
assuming that some of the values $Q_{12}^{(j)}(0)$, $Q_{21}^{(j)}(1)$ vanish.
\begin{proposition} \label{prop:Q12.Q21=0}
Let $Q_{12}, Q_{21} \in W_1^n[0,1]$ for some $n \in \bN$.
Assume also that the following conditions hold,
\begin{equation} \label{eq:Q120.Q211=0}
 J_{32} = 0, \qquad J_{14} \ne 0, \qquad
 Q_{12}^{(j)}(0) = Q_{21}^{(j)}(1) = 0 \quad\text{for}\quad
 0 \le j < m,
\end{equation}
with some $m \in \oneto{n-1}$.
Then the system of root vectors of the BVP~\eqref{eq:system}--\eqref{eq:Udef} is
complete and minimal in $L^2([0,1];\bC^2)$ whenever at least one of the following conditions holds
\begin{align}
\label{eq:Pn0}
 & \text{(i)} \quad P^{(n_0)}(0) \ne 0 \quad\text{for some}\quad
 n_0 \in \{m, m+1, \ldots, \wt{m}\}, \quad \wt{m} := \min\{2m, n-1\}, \\
\label{eq:P2m+1}
 & \text{(ii)} \quad 2m+2 \le n \quad \text{and} \quad
 P^{(2m+1)}(0) + J_{14} b_1 b_2 Q_{12}^{(m)}(0) Q_{21}^{(m)}(1) \ne 0.
\end{align}
\end{proposition}
\begin{proof}
In notations~\eqref{eq:q+q-} condition~\eqref{eq:Q120.Q211=0} takes the form
\begin{equation} \label{eq:q+.q-=0}
 q_+^{(j)}(0) = q_-^{(j)}(1) = 0 \quad\text{for}\quad
 0 \le j < m.
\end{equation}
In turn, Lemma~\ref{lem:sigma.q=0} implies that
\begin{equation} \label{eq:sigmapm=qmp}
 \sigma^-_{j}(0) = (-1)^{j-1} q_+^{(j-1)}(0), \qquad
 \sigma^+_{j}(1) = (-1)^{j-1} q_-^{(j-1)}(1), \qquad j \in \oneto{m_0},
\end{equation}
where $m_0 = \min\{2m+2,n\}$. Since $J_{14} \ne 0$, to apply Theorem~\ref{th:compl.gen.2x2}, we just need to verify that $c_k^+ \ne 0$ for some $k \in \oneton$. First let $k \le 2m+1$. Then
\begin{equation}
 \min\{j-1,k-j-1\} \le \frac{k-2}{2} \le \frac{2m-1}{2} < m,
 \qquad j \in \oneto{k-1}.
\end{equation}
Hence condition~\eqref{eq:q+.q-=0} implies that
\begin{equation} \label{eq:q+q-=0}
 q_+^{(j-1)}(0) q_-^{(k-j-1)}(1) = 0, \qquad j \in \oneto{k-1}.
\end{equation}
Inserting relations~\eqref{eq:sigmapm=qmp} and~\eqref{eq:q+q-=0} into formula~\eqref{eq:ck+} and taking into account definition~\eqref {eq:Pdef} we arrive at
\begin{multline}
 c_k^+ = J_{13} (-1)^{k-1} \sigma_k^-(0) + J_{42} \sigma_k^+(1)
 - J_{14} \sum_{j=1}^{k-1} (-1)^j \sigma_j^-(0) \sigma_{k-j}^+(1) \\
 = J_{13} q_+^{(k-1)}(0) + J_{42} (-1)^{k-1} q_-^{(k-1)}(1)
 - J_{14} \sum_{j=1}^{k-1}
 \underset{= 0}{\underbrace{q_+^{(j-1)}(0) q_-^{(k-j-1)}(1)}}
 = P^{(k-1)}(0).
\end{multline}
In turn, if condition~\eqref{eq:Pn0} holds, then $c_{n_0+1}^+ = P^{(n_0)}(0) \ne 0$ and the desired completeness property is ensured by Theorem~\ref{th:compl.gen.2x2}.

Finally we assume condition~\eqref{eq:P2m+1}.
Similarly as above we have
\begin{multline}
 c_{2m+2}^+
 = P^{(2m+1)}(0) - J_{14} q_+^{(m)}(0) q_-^{(m)}(1)
 - J_{14} \sum_{\atop{j=1}{j \ne m+1}}^{2m+1}
 \underset{= 0}{\underbrace{q_+^{(j-1)}(0) q_-^{(2m+1-j)}(1)}} \\
 = P^{(2m+1)}(0) + J_{14} b_1 b_2 Q_{12}^{(m)}(0) Q_{21}^{(m)}(1) \ne 0.
\end{multline}
Again, Theorem~\ref{th:compl.gen.2x2} completes the proof.
\end{proof}
\begin{remark}
Note, that when $m=0$, then Proposition~\ref{prop:Q12.Q21=0} is still valid. In this case no assumption on $Q(\cdot)$ is posed, and this result coincides with Corollary~\ref{cor:c012} above. Thus, Proposition~\ref{prop:Q12.Q21=0} can be viewed as a non-trivial generalization of Corollary~\ref{cor:c012}.
\end{remark}
We can now establish significant enhancement of Corollary~\ref{cor:Makin} by Makin. Recall that $\ceil{x}$, $x \in \bR$, denotes the smallest integer $k$ such that $k \ge x$.
\begin{corollary} \label{cor:Makin.gen1}
Let $Q_{12}, Q_{21} \in W_1^n[0,1]$ for some $n \in \bN$. Assume that $J_{32} = 0$, $J_{14} \ne 0$, and let the following conditions hold,
\begin{equation} \label{eq:Pn.Q12.Q21j}
 P^{(n-1)}(0) \ne 0, \quad
 Q_{12}^{(j)}(0) = Q_{21}^{(j)}(1) = 0,
 \quad 0 \le j < m,
 \quad\text{where}\quad m = \text{\scalebox{0.9}{$\ceil{\frac{n-1}{2}}$}},
\end{equation}
and $P(\cdot)$ is given by~\eqref {eq:Pdef}. Then the system of root functions of the operator $L_U(Q)$ is complete and minimal in $L^2([0,1]; \bC^2)$.
\end{corollary}
\begin{proof}
The proof is immediate from Propositon~\ref{prop:Q12.Q21=0}(i) applied
with $m = \ceil{\frac{n-1}{2}}$ and $n_0 = n-1$. Indeed, by definition of $\ceil{\cdot}$ we have $n_0 = n-1 \le 2m$. Hence the first condition in~\eqref{eq:Pn.Q12.Q21j} ensures condition~\eqref{eq:Pn0}.
\end{proof}
\begin{remark} \label{rem:Makin1}
Let us compare Corollary~\ref{cor:Makin.gen1} with Corollary~\ref{cor:Makin} by Makin from~\cite{Mak23}. Namely, let $-b_1 = b_2 = 1$ and $Q_{12}, Q_{21} \in W_1^n[0,1]$ for some $n \ge 3$, and assume that $P^{(n-1)}(0) \ne 0$ $($which coincides with condition~\eqref{eq:Pk.Makin} in the Dirac case$)$. Then Corollary~\ref{cor:Makin} ensures completeness property under more restrictive assumption~\eqref{eq:Q12.Q21.k} compared to our relaxed assumption~\eqref{eq:Pn.Q12.Q21j} from Corollary~\ref{cor:Makin.gen1}. Namely, Corollary~\ref{cor:Makin} requires additional $2\floor{\frac{n-1}{2}} \ge 2$ zero derivative values compared to~\eqref{eq:Pn.Q12.Q21j},
\begin{equation} \label{eq:Q12.Q21j}
 Q_{12}^{(j)}(0) = Q_{21}^{(j)}(1) = 0, \quad m \le j \le n-2,
 \quad\text{where}\quad m = \text{\scalebox{0.85}{$\ceil{\frac{n-1}{2}}$}}.
\end{equation}
Therefore, Proposition~\ref{prop:Q12.Q21=0} demonstrates that completeness property is valid under less restrictive condition than the one in~\cite[Theorem~1]{Mak23},
when $Q(\cdot)$ is rather smooth.
\end{remark}

Next result substantially complements Theorem~5.1 from~\cite{MalOri12}. Part (i) of the next theorem is announced in~\cite[Theorem 4]{LunMal13Dokl}.
\begin{theorem} \label{th:J32=J14=0}
Let $J_{32}=J_{14}=0$, $J_{13} J_{42} \ne 0$, and $Q_{12}, Q_{21}
\in W_1^n[0,1]$ for some $n \in \bN$.
Let for some $n_0, n_1 \in \{0, 1, \ldots, n-1\}$ the following conditions hold
\begin{align}
\label{eq:Pk0}
 P^{(k)}(0) &= 0, \qquad k \in \{0, 1, \ldots, n_0 - 1\}, \\
\label{eq:Pk1=0}
 P^{(k)}(1) &= 0, \qquad k \in \{0, 1, \ldots, n_1 - 1\},
\end{align}
where $P(\cdot)$ is given by~\eqref{eq:Pdef}.
Then the system of root vectors of the BVP~\eqref{eq:system}--\eqref{eq:Udef} is complete and minimal in $L^2([0,1];\bC^2)$ whenever one of the following conditions holds
\begin{align}
\label{eq:Pn1}
 & \text{(i)} \quad |n_1 - n_0| \le 1, \qquad P^{(n_0)}(0) \ne 0,
 \quad P^{(n_1)}(1) \ne 0. \\
\nonumber
 & \text{(ii)} \quad n_0 \ge n_1 + 2, \qquad P^{(n_1)}(1) \ne 0, \\
\label{eq:P.n1+2}
 & \qquad \qquad \qquad \qquad J_{13} P^{(n_1+2)}(0)
 + (-1)^{n_1} J_{42} b_1^2 Q_{12}^2(0) P^{(n_1)}(1) \ne 0, \\
\nonumber
 & \text{(iii)} \quad n_1 \ge n_0 + 2, \qquad P^{(n_0)}(0) \ne 0, \\
\label{eq:P.n0+2}
 & \qquad \qquad \qquad \qquad J_{13} P^{(n_0+2)}(1)
 + (-1)^{n_0} J_{42} b_1^2 Q_{12}^2(1) P^{(n_0)}(0) \ne 0.
\end{align}
\end{theorem}
\begin{proof}
With account of notations~\eqref{eq:q+q-} and~\eqref{eq:Pdef} we have
\begin{align}
 \label{eq:J13q+j0}
 & J_{13} q_+^{(j)}(0) + J_{42} (-1)^{j} q_-^{(j)}(1) = -i P^{(j)}(0),
 \qquad 0 \le j < n, \\
 \label{eq:J13q+j1}
 & J_{13} q_+^{(j)}(1) + J_{42} (-1)^{j} q_-^{(j)}(0) = -i P^{(j)}(1),
 \qquad 0 \le j < n.
\end{align}
Combining these relations with conditions~\eqref{eq:Pk0}--\eqref{eq:Pk1=0} yields
\begin{align}
\label{eq:J13q+n0}
 & J_{13} q_+^{(j)}(0) = J_{42} (-1)^{j-1} q_-^{(j)}(1),
 \qquad 0 \le j < n_0, \\
\label{eq:J13q+n1}
 & J_{13} q_+^{(j)}(1) = J_{42} (-1)^{j-1} q_-^{(j)}(0),
 \qquad 0 \le j < n_1.
\end{align}
Since $J_{13} J_{42} \ne 0$, it follows from these identities that
\begin{equation} \label{eq:q+q-0=q-q+1}
 q_-^{(j_0)}(1) q_+^{(j_1)}(1) = (-1)^{j_0 + j_1}
 q_+^{(j_0)}(0) q_-^{(j_1)}(0), \qquad 0 \le j_0 < n_0, \quad 0 \le j_1 < n_1.
\end{equation}
Further, since $J_{14} = J_{32} = 0$, then the coefficients $c_j^{\pm}$ given by~\eqref{eq:ck+}--\eqref{eq:ck-} are simplified to,
\begin{equation} \label{eq:ck+-0}
 c_j^+ = J_{13} (-1)^{j-1} \sigma_j^-(0) + J_{42} \sigma_j^+(1), \qquad
 c_j^- = J_{13} (-1)^{j-1} \sigma_j^-(1) + J_{42} \sigma_j^+(0),
\end{equation}
for $j \in \oneton$.
Inserting identity~\eqref{eq:sigmak=qk+sum.q} (see~\eqref{eq:Psij.def}--\eqref{eq:sigmah}) and relation~\eqref{eq:J13q+j1} into each of the equalities in~\eqref{eq:ck+-0} yields
\begin{align}
\nonumber
 c_j^+ & =
 \underset{= -i P^{(j-1)}(0)}{\underbrace{J_{13} q_+^{(j-1)}(0)
 + J_{42} (-1)^{j-1} q_-^{(j-1)}(1)}} \\
\label{eq:cj+full}
 & + \sum_{(k,\alp,\beta) \in \Psi_j} C_{j,k,\alpha,\beta} \cdot
 \left( J_{13} (-1)^{j-1} q_+^{(\alp)}(0) \cdot q_-^{(\beta)}(0)
 + J_{42} q_-^{(\alp)}(1) \cdot q_+^{(\beta)}(1) \right), \\
\nonumber
 c_j^- & =
 \underset{= -i P^{(j-1)}(1)}{\underbrace{J_{13} q_+^{(j-1)}(1)
 + J_{42} (-1)^{j-1} q_-^{(j-1)}(0)}} \\
\label{eq:cj-full}
 & + \sum_{(k,\alp,\beta) \in \Psi_j} C_{j,k,\alpha,\beta} \cdot
 \left( J_{13} (-1)^{j-1} q_+^{(\alp)}(1) \cdot q_-^{(\beta)}(1)
 + J_{42} q_-^{(\alp)}(0) \cdot q_+^{(\beta)}(0) \right).
\end{align}
Let $(k,\alp,\beta) \in \Psi_j$ with $\alp = (\alp_1, \ldots, \alp_{k+1}) \in \bZ_+^{k+1}$ and $\beta = (\beta_1, \ldots, \beta_k) \in \bZ_+^k$. First assume that
\begin{equation} \label{eq:alph<n1}
 \alp_h < n_0, \quad h \in \oneto{k+1}, \qquad
 \beta_h < n_1, \quad h \in \oneto{k}.
\end{equation}
Combining relations~\eqref{eq:J13q+n1} and~\eqref{eq:q+q-0=q-q+1} with equality $\Sigma_\alp + \Sigma_\beta = j-1-2k$ implies
\begin{multline} \label{eq:J42.q1}
 J_{42} q_-^{(\alp)}(1) \cdot q_+^{(\beta)}(1)
 = J_{42} q_-^{(\alp_{k+1})}(1)
 \times \prod_{h=1}^k q_-^{(\alp_h)}(1) q_+^{(\beta_h)}(1) \\
 = J_{13} (-1)^{\alp_{k+1} - 1} q_+^{(\alp_{k+1})}(0) \times \prod_{h=1}^k
 (-1)^{\alp_h + \beta_h} q_+^{(\alp_h)}(0) q_-^{(\beta_h)}(0) \\
 = J_{13} (-1)^{\Sigma_\alp + \Sigma_\beta - 1} \cdot q_+^{(\alp)}(0)
 \cdot q_-^{(\beta)}(0)
 = J_{13} (-1)^j q_+^{(\alp)}(0) \cdot q_-^{(\beta)}(0).
\end{multline}
Similarly, combining relations~\eqref{eq:J13q+n0} and~\eqref{eq:q+q-0=q-q+1} yields
\begin{equation} \label{eq:J42.q0}
 J_{42} q_-^{(\alp)}(0) \cdot q_+^{(\beta)}(0)
 = J_{13} (-1)^j q_+^{(\alp)}(1) \cdot q_-^{(\beta)}(1),
\end{equation}
whenever
\begin{equation} \label{eq:alph<n0}
 \alp_h < n_1, \quad h \in \oneto{k+1}, \qquad
 \beta_j < n_0, \quad h \in \oneto{k}.
\end{equation}

Now let $j \le \min\{n_0, n_1\} + 2$. For any triplet $(k, \alp, \beta) \in \Psi_j$ we have
\begin{align}
\label{eq:alp<j}
 & \alpha_h \le j-1-2k \le j-3 < \min\{n_0, n_1\}, \qquad h \in \oneto{k+1}, \\
\label{eq:beta<j}
 & \beta_h \le j-1-2k \le j-3 < \min\{n_0, n_1\}, \qquad h \in \oneto{k}.
\end{align}
Therefore, both relations~\eqref{eq:alph<n1} and~\eqref{eq:alph<n0} hold. In turn, combining identities~\eqref{eq:J42.q1} and~\eqref{eq:J42.q0} with equalities~\eqref{eq:cj+full} and~\eqref{eq:cj-full} implies
\begin{equation} \label{eq:cj+.cj-=Pj-1}
 c_j^+ = -i P^{(j-1)}(0), \qquad
 c_j^- = -i P^{(j-1)}(1), \qquad j \le \min\{n_0, n_1\} + 2.
\end{equation}

\textbf{(i)} First assume that condition~\eqref{eq:Pn1} holds. Then
\begin{equation}
 n_0 + 1 \le \min\{n_0, n_1\} + 2, \qquad n_1 + 1 \le \min\{n_0, n_1\} + 2.
\end{equation}
Hence combining relations~\eqref{eq:cj+.cj-=Pj-1} with condition $P^{(n_0)}(0) P^{(n_1)}(1) \ne 0$ yields
\begin{equation} \label{eq:cn0+1.cn1+1}
 c_{n_0+1}^+ = -i P^{(n_0)}(0) \ne 0, \qquad
 c_{n_1+1}^- = -i P^{(n_1)}(1) \ne 0.
\end{equation}
Therefore, Theorem~\ref{th:compl.gen.2x2} implies the desired completeness property, which finishes the proof in this case.

\textbf{(ii)} Now assume that conditions~\eqref{eq:P.n1+2} holds. Since $n_1 \le n_0 -2 \le n_0 + 1$, then as in the previous step
\begin{equation}
 c_{n_1+1}^- = -i P^{(n_1)}(1) \ne 0.
\end{equation}
However, since $n_0 \ge n_1 + 2$ and condition~\eqref{eq:Pk0} holds, then
\begin{equation}
 c_j^+ = -i P^{(j-1)}(0) = 0, \qquad j \in \oneto{n_1 + 2}.
\end{equation}
To this end, we evaluate $c_{n_1+3}^+$ and show that the second condition in~\eqref{eq:P.n1+2} yields $c_{n_1+3}^+ \ne 0$. Set $j = n_1+3$ and let $(k, \alp, \beta) \in \Psi_j$. Similarly to inequalities~\eqref{eq:alp<j}--\eqref{eq:beta<j} we conclude that
\begin{align}
\label{eq:alp<n0}
 & \alpha_h \le j-1-2k \le j-3 = n_1 < n_0, \qquad h \in \oneto{k+1}, \\
\label{eq:beta=n1}
 & \beta_h \le j-1-2k \le j-3 = n_1, \qquad h \in \oneto{k}.
\end{align}
Hence condition~\eqref{eq:alph<n1} holds for all triples $(k,\alp,\beta)$ except the case when $\beta_h = j-3 = n_1$ for some $h$. This is possible only when
\begin{equation}
 k=1, \quad \alp=\alp_0=(0,0) \in \bZ_+^2 \quad\text{and}\quad
 \beta=\gam_{h-3}=(h-3) \in \bZ_+^1.
\end{equation}
Therefore, identity~\eqref{eq:J42.q1} holds for each summand in the sum in the r.h.s.\ in~\eqref{eq:cj+full} except this one exceptional triple. In turn, with account of formula~\eqref{eq:Cj1ab}, relation~\eqref{eq:cj+full} simplifies too
\begin{multline} \label{eq:c.n1+3.+}
 c_j^+ = -i P^{(j-1)}(0) + (-1)^j \left(
 J_{13} (-1)^{j-1} q_+^2(0) \cdot q_-^{(j-3)}(0)
 + J_{42} q_-^2(1) \cdot q_+^{(j-3)}(1) \right) \\
 = -i P^{(n_1+2)}(0) - J_{13} q_+^2(0) \cdot q_-^{(n_1)}(0)
 - J_{42} q_-^2(1) \cdot (-1)^{n_1} q_+^{(n_1)}(1).
\end{multline}
Since $n_0 > 0$, it follows from~\eqref{eq:J13q+n0} that $J_{13} q_+(0) = -J_{42} q_-(1)$, which implies that
\begin{equation} \label{eq:J42q-2}
 J_{42} q_-^2(1) = \frac{J_{13}}{J_{42}} J_{13} q_+^2(0)
 = -\frac{J_{13}}{J_{42}} J_{13} b_1^2 Q_{12}^2(0).
\end{equation}
Inserting~\eqref{eq:J42q-2} and~\eqref{eq:J13q+j1} into~\eqref{eq:c.n1+3.+} and taking into account condition~\eqref{eq:P.n1+2}, we arrive at
\begin{multline}
 c_{n_1+3}^+ = -i P^{(n_1+2)}(0) - J_{13} q_+^2(0) \left(q_-^{(n_1)}(0)
 + \frac{J_{13}}{J_{42}} \cdot (-1)^{n_1} q_+^{(n_1)}(1)\right) \\
 = -i \left(P^{(n_1+2)}(0)
 - \frac{J_{13}}{J_{42}} q_+^2(0) \cdot (-1)^{n_1} P^{(n_1)}(1)\right) \\
 = \frac{-i}{J_{42}} \left( J_{42} P^{(n_1+2)}(0)
 + (-1)^{n_1} J_{13} b_1^2 Q_{12}^2(0) P^{(n_1)}(1)\right) \ne 0.
\end{multline}
As above, Theorem~\ref{th:compl.gen.2x2} finishes the proof in this case.

\textbf{(iii)} The case of condition~\eqref{eq:P.n0+2} is treated similarly, which completes the proof.
\end{proof}
\begin{corollary} \label{cor:P0.P1=0}
Let $J_{32}=J_{14}=0$, $J_{13} J_{42} \ne 0$, $Q_{12}, Q_{21} \in W_1^n[0,1]$ for some $n \ge 3$, and let $P(\cdot)$ be given by~\eqref{eq:Pdef}. Assume
that the following two conditions hold
\begin{align}
\label{eq:Q12.Pn0}
 & \qquad P^{(n_0)}(0) \ne 0 \quad\text{for some}\quad
 n_0 \in \{0,1,\ldots,n-3\}, \\
\label{eq:Pk=0}
 & Q_{12}(1) \ne 0, \quad\text{and}\quad
 P^{(k)}(1) = 0 \quad\text{for each}\quad k \in \{0,1,\ldots,n-1\}.
\end{align}
Then the system of root vectors of the BVP~\eqref{eq:system}--\eqref{eq:Udef} is
complete and minimal in $L^2([0,1];\bC^2)$.
\end{corollary}
\begin{proof}
Replacing $n_0$ with a smaller value if needed, we can assume condition~\eqref{eq:Pk0}. Condition~\eqref{eq:Pk=0} implies condition~\eqref{eq:Pk1=0} with $n_1 = n_0 + 2 < n$. It is also evident that $P^{(n_0 + 2)}(1) = 0$. Since $Q_{12}(1) \ne 0$ and $P^{(n_0)}(0) \ne 0$, then condition~\eqref{eq:P.n0+2} holds. Theorem~\ref{th:J32=J14=0}(iii) now finishes the proof.
\end{proof}
\begin{remark} \label{rem:Dirac.vs.SL}
Note that condition~\eqref{eq:Pk=0} holds, in particular, if $P(x) = 0$ for $x \in [a, 1]$ with some $a \in (0,1)$. Thus, Corollary~\ref{cor:P0.P1=0} shows substantial difference between Sturm-Liouville equation and $2 \times 2$ Dirac system in the case of non-regular boundary conditions.

Namely, consider Sturm-Liouville equation $-y'' + q(x) y = \l^2 y$, $x \in [0,1]$, subject to degenerate boundary conditions $y(0) - \alp y(1) = y'(0) + \alp y'(1) = 0$, $\alp \ne 0$. Let $\wt{P}(x) := q(x) - q(1-x)$, $x \in [0,1]$. Based on~\cite{Mal08,Mak14}, if $\wt{P}^{(k)}(0) \ne 0$ for some $k \in \bZ_+$, then the system of root vectors of such BVP is complete in $L^2[0,1]$. On the other hand, if $\wt{P}(x) = 0$ for $x \in [a,1]$ with some $a \in (0,1)$, then according to~\cite{Mal08} this system of root vectors is incomplete and has infinite defect.

On the other hand, as Corollary~\ref{cor:P0.P1=0} shows, in the case of Dirac operator, completeness is possible if $P(x) = 0$ for $x \in [a,1]$ as long as $P^{(k)}(0) \ne 0$ for some $k \in \bZ_+$ and $Q_{12}(1) \ne 0$. Moreover, if also $P(x) = 0$ for $x \in [0,\delta]$ with some $\delta \in (0,1)$, then according to~\cite[Proposition~5.12]{MalOri12}, \cite[Proposition~4.13]{LunMal15JST} the system of root vectors of the BVP for the $2 \times 2$ Dirac equation with $J_{14} = J_{32} = 0$ is also incomplete.
\end{remark}
Theorem~\ref{th:J32=J14=0}(i) has inconvenient restriction $|n_0 - n_1| \le 1$. We can relax it if some of derivatives $Q^{(j)}(0)$, $Q^{(j)}(1)$ vanish, as the following result shows.
\begin{proposition} \label{prop:Q=0.P0.P1}
Let $J_{32}=J_{14}=0$, $J_{13} J_{42} \ne 0$, and $Q_{12}, Q_{21}
\in W_1^n[0,1]$ for some $n \in \bN$.
Let for some $m_0, m_1 \in \{0, 1, \ldots, n-1\}$ the following conditions hold
\begin{align}
\label{eq:Q120.Q211=0.m0}
 Q_{12}^{(j)}(0) = Q_{21}^{(j)}(1) = 0, \qquad 0 \le j < m_0, \\
\label{eq:Q121.Q210=0.m1}
 Q_{12}^{(j)}(1) = Q_{21}^{(j)}(0) = 0, \qquad 0 \le j < m_1.
\end{align}
Further, let for some $n_0, n_1 \in \{0, 1, \ldots, n-1\}$ the following conditions hold
\begin{align}
\label{eq:Pk0.0ne0}
 P^{(j)}(0) &= 0, \quad j \in \{0, 1, \ldots, n_0 - 1\},
 \qquad P^{(n_0)}(0) \ne 0, \\
\label{eq:Pk1.0ne0}
 P^{(j)}(1) &= 0, \quad j \in \{0, 1, \ldots, n_1 - 1\},
 \qquad P^{(n_1)}(1) \ne 0,
\end{align}
where $P(\cdot)$ is given by~\eqref{eq:Pdef}. Then the system of root vectors of the BVP~\eqref{eq:system}--\eqref{eq:Udef} is complete and minimal in $L^2([0,1];\bC^2)$ whenever the following condition holds
\begin{equation} \label{eq:n0-n1}
 - 2 m_0 - 1 \le n_0 - n_1 \le 2 m_1 + 1.
\end{equation}
In particular, this holds if $n_0 \le 3 m_1 + 1$ and $n_1 \le 3 m_0 + 1$.
\end{proposition}
\begin{proof}
As in the proof of Theorem~\ref{th:J32=J14=0} conditions~\eqref{eq:Pk0.0ne0}--\eqref{eq:Pk1.0ne0} imply identities~\eqref{eq:J13q+j0}--\eqref{eq:q+q-0=q-q+1}.
With account of notation~\eqref{eq:q+q-}, conditions~\eqref{eq:Q120.Q211=0.m0}--\eqref{eq:Q121.Q210=0.m1} take the form
\begin{align}
\label{eq:q+q-.m0}
 q_+^{(j)}(0) = q_-^{(j)}(1) = 0, \qquad 0 \le j < m_0, \\
\label{eq:q+q-.m1}
 q_+^{(j)}(1) = q_-^{(j)}(0) = 0, \qquad 0 \le j < m_1.
\end{align}
First let $j = n_0 + 1$. We will prove
that $c_j^+ = c_{n_0+1}^+ \ne 0$. Let us analyze individual summands
\begin{equation} \label{eq:Q+ab.def}
 Q^+_{j,k,\alp, \beta}
 := J_{13} (-1)^{j-1} q_+^{(\alp)}(1) \cdot q_-^{(\beta)}(1)
 + J_{42} q_-^{(\alp)}(0) \cdot q_+^{(\beta)}(0)
\end{equation}
in the sum in the r.h.s.\ of formula~\eqref{eq:cj+full} for $c_j^+$. To this end, let $(k,\alp,\beta) \in \Psi_j$, i.e.
\begin{align}
 & \alp=(\alp_1, \ldots, \alp_{k+1}) \in \bZ_+^{k+1}, \qquad
 \beta=(\beta_1, \ldots, \beta_k) \in \bZ_+^k, \\
\label{eq:alp+beta}
 & \alp_1 + \ldots + \alp_{k+1} + \beta_1 + \ldots + \beta_k = j-1-2k,
\end{align}
where $k \in \bN$, $2k+1 \le j$. Note that if $\alp_h < m_1$ for some $h \in \oneto{k+1}$, then condition~\eqref{eq:q+q-.m1} implies that $q_+^{(\alp)}(1) = q_-^{(\alp)}(0) = 0$. Similarly, if $\beta_h < m_0$ for some $h \in \oneto{k}$, then condition~\eqref{eq:q+q-.m0} implies that $q_-^{(\beta)}(1) = q_+^{(\beta)}(0) = 0$. Therefore, $Q^+_{j,k,\alp, \beta} = 0$ if either $\alp_h < m_1$ for some $h$ or $\beta_h < m_0$ for some $h$.

Next consider the opposite case, i.e.
\begin{equation}
 \alp_h \ge m_1, \quad h \in \oneto{k+1}, \qquad
 \beta_h \ge m_0, \quad h \in \oneto{k}.
\end{equation}
Combining this inequalities with relation~\eqref{eq:alp+beta}, definition $j=n_0+1$ and the second inequality in~\eqref{eq:n0-n1} implies
\begin{align}
\label{eq:alph<n0-2}
 & \alp_h \le j-1-2k - k m_1 - k m_0 \le n_0 - 2 - m_1 - m_0 < n_0,
 \quad 1 \le h \le k+1, \\
\label{eq:beth<n0-2}
 & \beta_h \le j-1-2k - (k+1) m_1 - (k-1) m_0 \le n_0 - 2 - 2 m_1 < n_1,
 \quad 1 \le h \le k.
\end{align}
Hence, condition~\eqref{eq:alph<n1} holds and just as in the proof of Theorem~\ref{th:J32=J14=0}, combining identities~\eqref{eq:J13q+n0}--\eqref{eq:q+q-0=q-q+1} with condition~\eqref{eq:alph<n1} implies identity~\eqref{eq:J42.q1}, which again means that $Q^+_{j,k,\alp, \beta} = 0$. Hence, formula~\eqref{eq:cj+full} implies that $c_j^+ = -i P^{(j-1)}(0) = -i P^{(n_0)}(0) \ne 0$ due to the second condition in~\eqref{eq:Pk0.0ne0}.

Similarly, using inequality $n_1 - 2 - 2 m_0 < n_0$ we can prove that
\begin{equation} \label{eq:Q-ab.def}
 Q^-_{j,k, \alp, \beta} := J_{13} (-1)^{j-1} q_+^{(\alp)}(0)
 \cdot q_-^{(\beta)}(0) + J_{42} q_-^{(\alp)}(1) \cdot q_+^{(\beta)}(1) = 0,
 \quad (k,\alp,\beta) \in \Psi_j,
\end{equation}
where $j = n_1 + 1$. In turn, this implies that $c_j^- = c_{n_1+1}^- = -i P^{(n_1)}(1) \ne 0$ due to the second condition in~\eqref{eq:Pk1.0ne0}. Theorem~\ref{th:compl.gen.2x2} implies the desired completeness property, which finishes the proof of the main statement.

It remains to consider the case when $n_0 \le 3 m_1 + 1$ and $n_1 \le 3 m_0 + 1$. Note that the last condition in~\eqref{eq:Pk0.0ne0} means that
\begin{equation}
 P^{(n_0)}(0) = J_{13} b_1 Q_{12}^{(n_0)}(0)
 + J_{42} (-1)^{n_0} b_2 Q_{21}^{(n_0)}(1) \ne 0.
\end{equation}
Which implies that either $Q_{12}^{(n_0)}(0) \ne 0$ or $Q_{21}^{(n_0)}(1) \ne 0$. Combining this fact with~\eqref{eq:Q120.Q211=0.m0} implies that $n_0 \ge m_0$. Hence
\begin{equation}
 n_0 - n_1 \ge m_0 - n_1 \ge m_0 - 3 m_0 - 1 = -2 m_0 - 1.
\end{equation}
Similarly, we derive that $n_1 \ge m_1$ and that $n_0 - n_1 \le 2 m_1+1$. Therefore, inequalities~\eqref{eq:n0-n1} hold whenever $n_0 \le 3 m_1 + 1$ and $n_1 \le 3 m_0 + 1$.
\end{proof}
\begin{remark}
\textbf{(i)} If $m_0=0$, then condition~\eqref{eq:Q120.Q211=0.m0} holds automatically since there are no $j$ satisfying inequality $0 \le j < m_0$. Similarly, if $m_1=0$, then condition~\eqref{eq:Q121.Q210=0.m1} holds automatically. In particular, if $m_0=m_1=0$, then condition~\eqref{eq:n0-n1} turns into $|n_0 - n_1| \le 1$, and so in this case Proposition~\ref{prop:Q=0.P0.P1} is equivalent to Theorem~\ref{th:J32=J14=0}(i).

\textbf{(ii)} Emphasize that in condition~\eqref{eq:Q120.Q211=0.m0} we do not assume that either $Q_{12}^{(m_0)}(0) \ne 0$ or $Q_{21}^{(m_0)}(1) \ne 0$. Similar can be said on condition~\eqref{eq:Q121.Q210=0.m1}.
\end{remark}
In turn, with sufficient number of zero derivatives $Q^{(j)}(0)$, $Q^{(j)}(1)$ we can eliminate any restrictions between $n_0$ and $n_1$ in Theorem~\ref{th:J32=J14=0}(i).
\begin{corollary} \label{cor:Makin.gen2}
Let $J_{32}=J_{14}=0$, $J_{13} J_{42} \ne 0$, and let $Q_{12}, Q_{21}
\in W_1^n[0,1]$ for some $n \in \bN$. Let $P(\cdot)$ be given by~\eqref{eq:Pdef}.
Let the following conditions hold
\begin{align}
\label{eq:Q12.Q21=0}
 & Q^{(j)}(0) = Q^{(j)}(1) = 0, \quad j \in \{0,1,\ldots,m-1\},
 \quad\text{where}\quad m = \text{\scalebox{0.9}{$\ceil{\frac{n-2}{3}}$}}, \\
\label{eq:Pn00.Pn11}
 & P^{(n_0)}(0) P^{(n_1)}(1) \ne 0 \qquad\text{for some}\quad
 n_0, n_1 \in \{m, m+1, \ldots, n-1\}.
\end{align}
Then the system of root vectors of the BVP~\eqref{eq:system}--\eqref{eq:Udef} is complete and minimal in $L^2([0,1];\bC^2)$.
\end{corollary}
\begin{proof}
It is clear that conditions~\eqref{eq:Q120.Q211=0.m0}--\eqref{eq:Q121.Q210=0.m1} hold with $m_0 = m_1 = m$. It follows from the definition of $\ceil{x}$ that $3m \ge n-2$. Hence $n_0 \le n-1 \le 3m+1 = 3m_1+1$ and $n_1 \le n-1 \le 3m+1 = 3m_0+1$. The second part of Proposition~\ref{prop:Q=0.P0.P1} completes the proof.
\end{proof}
Let us compare Corollary~\ref{cor:Makin.gen2} with Theorem~\ref{th:Makin} by Makin. First we reformulate Theorem~\ref{th:Makin} in the case of smooth potential when $J_{32} = J_{14} = 0$ and $J_{13} J_{42} \ne 0$.
\begin{corollary}[\cite{Mak23}] \label{cor:Makin2}
Let $B = \diag(-1,1)$, $J_{32}=J_{14}=0$, $J_{13} J_{42} \ne 0$, and let $Q_{12}, Q_{21} \in W_1^n[0,1]$ for some $n \in \bN$. Assume that
\begin{equation} \label{eq:Q12.Q21.01}
 Q_{12}^{(\rho_1)}(0) \cdot Q_{12}^{(\rho_2)}(1) \cdot
 Q_{21}^{(\rho_3)}(0) \cdot Q_{21}^{(\rho_4)}(1) \ne 0,
\end{equation}
for some $\rho_1, \rho_2, \rho_3, \rho_4 \in \{0,1,\ldots, n-1\}$.
Let for some $n_0, n_1 \in \{0,1,\ldots, n-1\}$ the following conditions hold
\begin{align}
\label{eq:Q12.Q21.Pn0}
 & Q_{12}^{(j)}(0) = Q_{21}^{(j)}(1) = 0,
 \quad j \in \{0,1,\ldots,n_0-1\}, \qquad P^{(n_0)}(0) \ne 0, \\
\label{eq:Q12.Q21.Pn1}
 & Q_{12}^{(j)}(1) = Q_{21}^{(j)}(0) = 0,
 \quad j \in \{0,1,\ldots,n_1-1\}, \qquad P^{(n_1)}(1) \ne 0,
\end{align}
where $P(\cdot)$ is given by~\eqref{eq:Pdef}. Then the system of root vectors of the BVP~\eqref{eq:system}--\eqref{eq:Udef} is complete in $L^2([0,1];\bC^2)$.
\end{corollary}
\begin{proof}
Condition guarantees existence of $\rho_1, \rho_2, \rho_3, \rho_4 \le n$ for which conditions~\eqref{eq:rho1.rho2} and~\eqref{eq:rho3.rho4} hold.
Let us verify condition~\eqref{eq:Q120-Q211}. Condition~\eqref{eq:Q12.Q21.Pn0} guarantees that $\rho_1, \rho_2 \ge n_0+1$. In turn since
\begin{equation}
 0 \ne P^{(n_0)}(0) = J_{13} Q_{12}^{(n_0)}(0) - J_{42} (-1)^{n_0} Q_{21}^{(n_0)}(1),
\end{equation}
it follows that either $Q_{12}^{(n_0)}(0) \ne 0$ or $Q_{21}^{(n_0)}(1) \ne 0$. Assume without loss of generality that $Q_{12}^{(n_0)}(0) \ne 0$, then $\rho_1 = n_0 + 1$ (see~\eqref{eq:lim.int.Q120}). If $Q_{21}^{(n_0)}(1) = 0$ then $\rho_2 > \rho_1$. Since by assumption $J_{13} J_{42} \ne 0$, then $(|J_{13}|+|J_{42}|)|\rho_1 - \rho_2| \ne 0$ and hence condition~\eqref{eq:Q120-Q211} holds. If $Q_{21}^{(n_0)}(1) \ne 0$ then $\rho_1 = \rho_2 = n_0+1$ and condition~\eqref{eq:Q120-Q211} is verified as in the proof of Corollary~\ref{cor:Makin}. Condition~\eqref{eq:Q121-Q210} can be verified similarly. Theorem~\ref{th:Makin} completes the proof.
\end{proof}
\begin{remark}
Let us compare Corollary~\ref{cor:Makin.gen2} with Corollary~\ref{cor:Makin2} by Makin~\cite{Mak23}. Namely, let $-b_1 = b_2 = 1$ and let assumptions of Corollary~\ref{cor:Makin.gen2} be satisfied $($recall that $m = \ceil{\frac{n-2}{3}})$, which ensures the desired completeness property. On the other hand, for Corollary~\ref{cor:Makin2} to guarantee the completeness property, condition~\eqref{eq:Q12.Q21=0} should be satisfied as well as additional $2(n_0+n_1-2m)$ derivative values at the endpoints should vanish,
\begin{align}
 & Q_{12}^{(j)}(0) = Q_{21}^{(j)}(1) = 0,
 \qquad j \in \{m,m+1,\ldots,n_0-1\}, \\
 & Q_{12}^{(j)}(1) = Q_{21}^{(j)}(0) = 0,
 \qquad j \in \{m,m+1,\ldots,n_1-1\}.
\end{align}
Thus, Proposition~\ref{prop:Q=0.P0.P1} requires much less than~\cite[Theorem~1]{Mak23} in the discussed case.
\end{remark}
Finally we study the case $J_{32} = J_{13} = 0$. As it turns out, completeness conditions are the most simple and explicit in this case.
\begin{corollary} \label{cor:J32=J42=J13=0}
Let $Q_{12}, Q_{21} \in W_1^n[0,1]$ for some $n \ge 2$. Let the following condition hold
\begin{equation} \label{eq:Q12j0.Q21j1}
 J_{32} = J_{13} = J_{42} = 0, \qquad J_{14} \ne 0, \qquad
 Q_{12}^{(j_0)}(0) Q_{21}^{(j_1)}(1) \ne 0,
\end{equation}
for some $j_0, j_1 \in \{0, 1, \ldots, n-2\}$ such that $j_0 + j_1 \le n-2$.
Then the system of root vectors of the BVP~\eqref{eq:system}--\eqref{eq:Udef} is
complete and minimal in $L^2([0,1];\bC^2)$.
\end{corollary}
\begin{proof}
As in the previous results, we will verify condition~\eqref{eq:ck+ck-.ne0} from Theorem~\ref{th:compl.gen.2x2}. Since $J_{14} \ne 0$, it follows that $c_0^- = J_{14} \ne 0$. Due to the assumption $J_{32} = J_{13} = J_{42} = 0$, formula~\eqref{eq:ck+} simplifies to
\begin{equation} \label{eq:ck+.J14}
 c_k^+ = - J_{14} \sum_{j=1}^{k-1} (-1)^j \sigma_j^-(0) \sigma_{k-j}^+(1),
 \qquad k \in \oneton.
\end{equation}
Since $Q_{12}^{(j_0)}(0) \ne 0$, without loss of generality we can assume that
$Q_{12}^{(j)}(0) = 0$ for $j \in \{0, \ldots, j_0-1\}$.
With account of notation~\eqref{eq:q+q-} we have
\begin{equation} \label{eq:q+j=0.k.ne.0}
 q_+^{(j)}(0) = 0, \quad j \in \{0, \ldots, j_0-1\},
 \qquad q_+^{(j_0)}(0) \ne 0.
\end{equation}
Combining Corollary~\ref{cor:q0.sigma0}(ii)
with condition~\eqref{eq:q+j=0.k.ne.0} now implies
\begin{equation} \label{eq:sigma-.j=0}
 \sigma^-_{j}(0) = 0, \quad j \in \oneto{j_0}, \qquad
 \sigma^-_{j_0 + 1}(0) = (-1)^{j_0} q_+^{(j_0)}(0) \ne 0.
\end{equation}
Similarly, we infer from~\eqref{eq:q-j=0} and condition $Q_{21}^{(j_1)}(1) \ne 0$ that
\begin{equation} \label{eq:sigma+.j=0}
 \sigma^+_{j}(1) = 0, \quad j \in \oneto{j_1}, \qquad
 \sigma^+_{j_1 + 1}(1) = (-1)^{j_1} q_-^{(j_1)}(1) \ne 0.
\end{equation}
In turn, combining~\eqref{eq:sigma-.j=0} with~\eqref{eq:sigma+.j=0} yields
\begin{equation} \label{eq:sigma.j.sigma.k-j=0}
 \sigma_j^-(0) \sigma_{k^+ -j}^+(1) = 0, \quad k^+ := j_0 + j_1 + 2 \le n, \quad
 j \ne j_0+1,
\end{equation}
for $1 \le j \le k^+ - 1$. Combining~\eqref{eq:ck+.J14} with relations~\eqref{eq:sigma-.j=0}--\eqref{eq:sigma.j.sigma.k-j=0} and condition $J_{14} \ne 0$, we arrive at
\begin{equation} \label{eq:ck+.J14.ne0}
 c_{k^+}^+ = - J_{14} (-1)^{j_0+1} \sigma_{j_0+1}^-(0) \sigma_{j_1+1}^+(1)
 \ne 0.
\end{equation}
Therefore, $c_0^- \ne 0$ and $c_{k^+}^+ \ne 0$. Hence Theorem~\ref{th:compl.gen.2x2} implies the desired completeness property.
\end{proof}
\begin{remark}
If $j_0 = j_1 = 1$ in condition~\eqref{eq:Q12j0.Q21j1} $($i.e.\ $Q_{12}(0) Q_{21}(1) \ne 0)$,
then Proposition~\ref{prop:2x2.notR} guarantees the desired completeness property under much more relaxed smoothness assumption of continuity of $Q(\cdot)$ at the endpoints, as opposed to smoothness assumption $Q \in W_1^2([0,1]; \bC^{2 \times 2})$ imposed by Corollary~\ref{cor:J32=J42=J13=0}. This shows some advantage of the Birkhoff solution approach used in~\cite{LunMal15JST}.
\end{remark}
If $J_{32} = J_{13} = 0$ and $J_{42} \ne 0$, then condition of completeness is also very simple and explicit.
\begin{corollary} \label{cor:J32=J13=0}
Let $Q_{12}, Q_{21} \in W_1^n[0,1]$ for some $n \in \bN$. Let the following condition hold
\begin{equation} \label{eq:J32=J13=0}
 J_{32} = J_{13} = 0, \qquad J_{42} \ne 0, \qquad
 Q_{21}^{(j_1)}(1) \ne 0, \qquad |J_{14}| + |Q_{21}^{(j_0)}(0)| \ne 0,
\end{equation}
for some $j_0, j_1 \in \{0, 1, \ldots, n-1\}$. Then the system of root vectors of the BVP~\eqref{eq:system}--\eqref{eq:Udef} is
complete and minimal in $L^2([0,1];\bC^2)$.
\end{corollary}
\begin{proof}
Since $J_{32} = J_{13} = 0$, formulas~\eqref{eq:ck+}--\eqref{eq:ck-} for
coefficients $c_k^{\pm}$, $k \in \{0,1, \ldots, n\}$, take the form
\begin{align}
\label{eq:ck+.J32=J13=0}
 & c_0^+ = 0, \qquad c_k^+ := J_{42} \sigma_k^+(1)
 - J_{14} \sum_{j=1}^{k-1} (-1)^j \sigma_j^-(0) \sigma_{k-j}^+(1), \\
\label{eq:ck-.J32=J13=0}
 & c_0^- = J_{14}, \qquad c_k^- := J_{42} \sigma_k^+(0).
\end{align}
Since $Q_{21}^{(j_1)}(1) \ne 0$, then just as in the proof of Corollary~\ref{cor:J32=J42=J13=0} we infer relation~\eqref{eq:sigma+.j=0}. In turn, combining~\eqref{eq:ck+.J32=J13=0} with~\eqref{eq:sigma+.j=0} and condition $J_{42} \ne 0$ implies
\begin{equation} \label{eq:cj1+1}
 c_{j_1 + 1}^+ = J_{42} \sigma_{j_1 + 1}^+(1)
 - J_{14} \sum_{j=1}^{j_1} (-1)^j \sigma_j^-(0) \cdot
 \underset{\text{$ = 0$}}{\underbrace{\sigma_{j_1+1-j}^+(1)}}
 = J_{42} (-1)^{j_1} q_-^{(j_1)}(1) \ne 0.
\end{equation}

First, let $J_{14} \ne 0$. Then
combining~\eqref{eq:cj1+1} with Theorem~\ref{th:compl.gen.2x2} implies the desired completeness property.

Now, let $J_{14} = 0$. Then condition~\eqref{eq:J32=J13=0} implies that $Q_{21}^{(j_0)}(0) \ne 0$. Similarly to how relation~\eqref{eq:sigma-.j=0} was established in the proof of Corollary~\ref{cor:J32=J42=J13=0} we derive that
\begin{equation} \label{eq:sigma+.j0=0}
 \sigma^+_{j}(0) = 0, \quad j \in \oneto{j_0}, \qquad
 \sigma^+_{j_0 + 1}(0) = (-1)^{j_0} q_-^{(j_0)}(0) \ne 0.
\end{equation}
Since $J_{42} \ne 0$, then combining~\eqref{eq:sigma+.j0=0} with~\eqref{eq:ck-.J32=J13=0} implies that $c_{j_0+1}^+ \ne 0$. Theorem~\ref{th:compl.gen.2x2} completes the proof.
\end{proof}
For reader convenience we also formulate the previous results in a similar case $J_{14} = J_{42} = 0$.
\begin{corollary} \label{cor:J14=J42=0}
Let $Q_{12}, Q_{21} \in W_1^n[0,1]$ for some $n \in \bN$ and let
\begin{equation}
 J_{14} = J_{42} = 0, \qquad Q_{21}^{(j_0)}(0) \ne 0,
\end{equation}
for some $j_0 \in \{0, 1, \ldots, n-1\}$. Let one of the following conditions hold
\begin{align}
 (i) \quad & J_{13} = 0, \quad J_{32} \ne 0, \quad Q_{12}^{(j_1)}(1) \ne 0
 \quad\text{for some}\quad 0 \le j_1 \le n-2-j_0, \\
 (ii) \quad & J_{13} \ne 0, \quad |J_{32}| + |Q_{21}^{(j_1)}(1)| \ne 0,
 \quad\text{for some}\quad j_1 \in \{0, 1, \ldots, n-1\}.
\end{align}
Then the system of root vectors of the BVP~\eqref{eq:system}--\eqref{eq:Udef} is complete and minimal in $L^2([0,1];\bC^2)$.
\end{corollary}
In the case when potential matrix function $Q(\cdot)$ is analytic,
Corollary~\ref{cor:J32=J13=0} yields the following criterion of completeness of the system of root vectors for BVP~\eqref{eq:system}--\eqref{eq:Udef} for a certain class of boundary conditions. Note that if $b_1+b_2 \ne 0$ then any degenerate boundary conditions (either $\Delta_0(\cdot) \equiv 0$ or $\Delta_0(\cdot)$ has no zeros) can be transformed to a form~\eqref{eq:y2(1)=0} using transformations $y_1 \leftrightarrow y_2$ and $x \rightarrow 1-x$.
\begin{corollary} \label{cor:criterion}
Let $Q(\cdot)$ be an analytic matrix function on $[0,1]$, and let boundary
conditions~\eqref{eq:Udef} be of the form
\begin{equation} \label{eq:y2(1)=0}
 \alpha_1 y_1(0) + \alpha_2 y_2(0) + \alpha_3 y_1(1) = 0,
 \quad y_2(1)=0, \quad \alpha_2 \ne 0.
\end{equation}
Then the system of root vectors of the BVP~\eqref{eq:system},
\eqref{eq:y2(1)=0} is complete in $L^2([0,1];\bC^2)$ if and only if
$Q_{21}(\cdot) \not \equiv 0$.
\end{corollary}
\begin{proof}
It follows from the form of boundary conditions~\eqref{eq:y2(1)=0} that
$J_{32} = J_{13} = 0$ and $J_{42} = -\alpha_2 \ne 0$.

Let $Q_{21}(\cdot) \not\equiv 0$. Since $Q_{21}(\cdot)$ is analytic on $[0, 1]$, then for some $j_0, j_1 \in \bZ_+$ we have
\begin{equation}
 Q_{21}^{(j_1)}(1) \ne 0 \quad\text{and}\quad Q_{21}^{(j_0)}(0) \ne 0,
\end{equation}
which trivially implies condition~\eqref{eq:J32=J13=0}.
Hence Corollary~\ref{cor:J32=J13=0} implies that the system of root vectors of the BVP~\eqref{eq:system}, \eqref{eq:y2(1)=0} is complete in $L^2([0,1];\bC^2)$.

If $Q_{21}(\cdot) \equiv 0$ then for any solution $y = \col(y_1, y_2)$ of the
system~\eqref{eq:system}, \eqref{eq:y2(1)=0} we have
\begin{equation} \label{eq:y2'=y2}
 -i b_2^{-1} y_2'(x) = \l y_2(x), \quad x \in [0, 1], \qquad y_2(1) = 0.
\end{equation}
Clearly~\eqref{eq:y2'=y2} is an initial value problem that has the only solution
$y_2(\cdot) \equiv 0$. Thus, any root vector of the system~\eqref{eq:system},
\eqref{eq:y2(1)=0} is of the form $y_n = \col(y_{n,1}, 0)$ and hence the
system of root vectors of the BVP~\eqref{eq:system}, \eqref{eq:y2(1)=0} is
incomplete in $L^2([0,1];\bC^2)$ and of infinite co-dimension. The proof is now complete.
\end{proof}


\begin{thebibliography}{10}
\addcontentsline{toc}{section}{References}
\bibitem{AgiMalOri12} A.V.~Agibalova, M.M.~Malamud and L.L.~Oridoroga, On the completeness of general boundary value problems for $2 \times 2$ first-order systems of ordinary differential equations, \emph{Methods of Functional Analysis and Topology}, \textbf{18} (1) (2012), pp. 4--18.

\bibitem{BirLan23} G.D.~Birkhoff and R.E.~Langer, The boundary problems and developments associated with a system of ordinary differential equations of the first order, \emph{Proc. Amer. Acad. Arts Sci.} \textbf{58} (1923), pp. 49--128.
\bibitem{DjaMit10} P.~Djakov and B.~Mityagin, Bari-Markus property for Riesz projections of 1D periodic Dirac operators, \emph{Math. Nachr.} \textbf{283} (3) (2010), pp. 443--462.
\bibitem{DjaMit12Crit} P.~Djakov and B.~Mityagin, Criteria for existence of Riesz bases consisting of root functions of Hill and 1D Dirac operators, \emph{J. Funct. Anal.} \textbf{263} (8) (2012), pp. 2300--2332.
\bibitem{DjaMit12UncDir} P.~Djakov and B.~Mityagin, Unconditional convergence of spectral decompositions of 1D Dirac operators with regular boundary conditions, \emph{Indiana Univ. Math. J.} \textbf{61} (1) (2012), pp. 359--398.
\bibitem{Gin71} Yu.P.~Ginzburg, The almost invariant spectral propeties of contractions and the multiplicative properties of analytic operator-functions, \emph{Funct. Anal. Appl.} \textbf{5} (3) (1971), pp. 197--205.
\bibitem{GomRze20} A.M.~Gomilko and L.~Rzepnicki, On asymptotic behaviour of solutions of the Dirac system and applications to the Sturm-Liouville problem with a singular potential, \emph{Journal of Spectral Theory} \textbf{10} (3) (2020), pp.~747--786.
\bibitem{Iba23} E.C.~Ibadov, On the Properties of the Root Vector Function Systems of 2$m$th-Order Dirac Type Operator with an Integrable Potential, \emph{Diff Equat} \textbf{59} (2023), pp.~1295--1314.
\bibitem{KosShk21} A.P.~Kosarev and A.A.~Shkalikov, Spectral asymptotics of solutions of a $2 \times 2$ system of first-order ordinary differential equations,
\emph{Math. Notes} \textbf{110} (5-6) (2021), pp. 967--971.
\bibitem{KosShk23} A.P.~Kosarev and A.A.~Shkalikov, Asymptotics in the spectral parameter for solutions of $2 \times 2$ systems of ordinary differential equations,
\emph{Math. Notes} \textbf{114} (3-4) (2023), pp. 472--488 (arXiv:2212.0622).
\bibitem{KurAbd18} V.M.\;Kurbanov and A.M.\;Abdullayeva, Bessel property and basicity of the system of root vector-functions of Dirac operator with summable coefficient, \emph{Operators and Matrices} \textbf{12} (4) (2018), pp. 943--954.
\bibitem{KurGad20} V.M. Kurbanov and G.R. Gadzhieva, Bessel inequality and the basis property for $2m \times 2m$ Dirac type system with an integrable potential, \emph{Differential Equations} \textbf{56} (5) (2020), pp. 573--584.
\bibitem{Lun23} A.A.~Lunyov, Criterion of Bari basis property for $2 \times 2$ Dirac-type operators with strictly regular boundary conditions, \emph{Math. Nachr.} \textbf{296} (9) (2023), pp.\,4125-4151.
\bibitem{LunMal13Dokl} A.A.~Lunyov and M.M.~Malamud, On the completeness of the root vectors for first order systems, \emph{Dokl. Math.} \textbf{88}(3) (2013), pp. 678--683.
\bibitem{LunMal14IEOT} A.A.~Lunyov and M.M.~Malamud, On Spectral Synthesis for Dissipative Dirac Type Operators, \emph{Integr. Equ. Oper. Theory} \textbf{90} (2014), pp. 79--106.
\bibitem{LunMal14Dokl} A.A.~Lunyov and M.M.~Malamud, On the Riesz Basis Property of the Root Vector System for Dirac-Type $2 \times 2$ Systems, \emph{Dokl. Math.} \textbf{90} (2) (2014), pp. 556--561.
\bibitem{LunMal15JST} A.A.~Lunyov and M.M.~Malamud, On the completeness and Riesz basis property of root subspaces of boundary value problems for first order systems and applications, \emph{J. Spectral Theory} \textbf{5} (1) (2015), pp. 17--70.
\bibitem{LunMal16JMAA} A.A.~Lunyov and M.M.~Malamud, On the Riesz basis property of root vectors system for $2 \times 2$ Dirac type operators, \emph{J. Math. Anal. Appl.} \textbf{441} (2016), pp. 57--103 (arXiv:1504.04954).
\bibitem{LunMal21} A.A.\ Lunyov and M.M.\ Malamud, On transformation operators and Riesz basis property of root vectors system for $n \times n$ Dirac type operators. Application to the Timoshenko beam model. arXiv:2112.07248 (Submitted on 14 Dec 2021).
\bibitem{LunMal22JDE} A.A.\;Lunyov and M.M.\;Malamud, Stability of spectral characteristics of boundary value problems for $2 \times 2$ Dirac type systems. Applications to the damped string, \emph{J.\;Differential Equations} \textbf{313} (2022), pp.\,633--742 (arXiv:2012.11170).
\bibitem{LunMal22POMI} A.~Lunev, M.~Malamud, On characteristic determinants of boundary value problems for Dirac type systems. (Russian) Zap. Nauchn. Sem. S.-Peterburg. Otdel. Mat. Inst. Steklov. (POMI) 516 (2022), Matematicheskie Voprosy Teorii Rasprostraneniya Voln. 52, 69--120.
\bibitem{Mak14} A.S.~Makin, On the completeness of the system of root functions of the Sturm-Liouville operator with degenerate boundary conditions.
\emph{Differential Equations} \textbf{50} (6) (2014), pp. 835--839.
\bibitem{Mak20} A.S.~Makin, Regular boundary value problems for the Dirac operator, \emph{Doklady Mathematics} \textbf{101} (3) (2020), pp. 214--217.
\bibitem{Mak21DE} A.S.~Makin, On the spectrum of two-point boundary value problems for the Dirac operator, \emph{Differential Equations} \textbf{57} (8) (2021), pp. 993--1002.
\bibitem{Mak22} A.S.\ Makin, On convergence of spectral Expansions of Dirac Operators with Regular Boundary Conditions, \emph{Math. Nachr.} \textbf{295} (1) (2022), pp. 189--210 (arXiv:1902.02952).
\bibitem{Mak23} A.S.\ Makin, On the completeness of root function system of the Dirac operator with two-point boundary conditions, arXiv:2304.06108.
\bibitem{Mal94} M.M.~Malamud, Similarity of Volterra operators and related questions of the theory of differential equations of fractional order, \emph{Trans. Moscow Math. Soc.} \textbf{55} (1994), pp. 57--122.
\bibitem{Mal99} M.M.~Malamud, Questions of uniqueness in inverse problems for systems of differential equations on a finite interval, \emph{Trans. Moscow Math. Soc.} \textbf{60} (1999), pp. 173--224.
\bibitem{Mal08} M.M.~Malamud, On the completeness of a system of root vectors of the Sturm-Liouville operator with general boundary conditions, \emph{Funct. Anal. Appl.} \textbf{42} (3) (2008), pp. 198--204.
\bibitem{MalOri00} M.M.~Malamud and L.L.~Oridoroga, Completeness theorems for systems of differential equations, \emph{Funct. Anal. Appl.} \textbf{34} (4) (2000), pp. 308--310.
\bibitem{MalOri10} M.M.~Malamud and L.L.~Oridoroga, On the completeness of the system of root vectors for second-order systems,
\emph{Dokl. Math.} \textbf{82} (3) (2010), pp.~899--904.
\bibitem{MalOri12} M.M.~Malamud and L.L.~Oridoroga, On the completeness of root subspaces of boundary value problems for first order systems of ordinary differential equations, \emph{J. Funct. Anal.} \textbf{263} (2012), pp. 1939--1980.
\bibitem{Mar77} V.A.~Marchenko, Sturm-Liouville operators and applications, \emph{Operator Theory: Advances and Appl.} \textbf{vol. 22}, Birkh\"{a}user Verlag, Basel (1986).
\bibitem{ZMNovPit80} S.P.~Novikov, S.V.~Manakov, L.P.~Pitaevskij and V.E.~Zakharov, Theory of solitons. The inverse scattering method. Springer-Verlag (1984).
\bibitem{Rzep21} L.~Rzepnicki, Asymptotic behavior of solutions of the Dirac system with an integrable potential, \emph{Integral Equations Operator Theory}, \textbf{93}, Article number: 55 (2021), 24 p, arXiv:2011.06510.
\bibitem{SavSad15} A.M.~Savchuk and I.V.~Sadovnichaya, The Riesz basis property with brackets for the Dirac system with a summable potential. \emph{J. Math. Sci. (N.Y.)} \textbf{233} (4) (2018), pp. 514--540.
\bibitem{SavShk14} A.M.~Savchuk and A.A.~Shkalikov, The Dirac Operator with Complex-Valued Summable Potential, \emph{Math. Notes} \textbf{96} (5-6) (2014), pp. 777--810.
\bibitem{SavShk20} A.M.~Savchuk and A.A.~Shkalikov, Asymptotic analysis of solutions of odinary differential equations with distribution coefficients. \emph{Mat. Sb.} \textbf{211} (11) (2020), pp. 129--166 (in Russian).
\end{thebibliography}
\end{document}